\documentclass[11pt]{article}
\usepackage{amssymb, amsmath, amsthm, oldgerm,srcltx}
\usepackage{graphicx}
\usepackage{xcolor}
\usepackage{titletoc}
\newcommand{\tc}{\textcolor{black}}

\newtheorem{thm}{Theorem}[section]
\newtheorem{lem}{Lemma}[section]
\newtheorem{cor}[lem]{Corollary}
\newtheorem{prop}[thm]{Proposition}
\newtheorem{rem}[thm]{Remark}
\numberwithin{equation}{section}

\newcommand{\abs}[1]{\left\vert#1\right\vert}

\newcommand{\E}{\mathbf{E}\,}

\newcommand{\R}{\mathbf{R}}

\newcommand{\re}{\mathrm{Re}\;\!}
\newcommand{\im}{\mathrm{Im}\;\!}
\newcommand{\Tr}{\mathrm{Tr}\;\!}
\newcommand{\q}{\quad}

\newenvironment{Proof of}{\removelastskip\par\medskip
\noindent{\em Proof of} \rm}{\penalty-20\null\hfill$\square$\par\medbreak}

\def\be{\begin{equation}}
\def\en{\end{equation}}
\def\bee{\begin{eqnarray*}}
\def\ene{\end{eqnarray*}}

\def\E{{\bf E}}

\def\R{{\mathbb R}}

\def\Tr{{\rm Tr}\,}
\def\Im{{\rm Im}\,}

\def\<{\left<}
\def\>{\right>}

\def\1{{\bf 1}}
\def\4{\kern1pt}

\begin{document}
\bibliographystyle{}

\vspace{1in}
 \date{}
\title{\bf Rate of  Convergence of the Expected Spectral Distribution Function to the Marchenko -- Pastur  Law}

\author{{\bf F. G\"otze}\\{\small Faculty of Mathematics}
\\{\small University of Bielefeld}\\{\small Germany}
\and {\bf A. Tikhomirov}$^{1}$\\{\small Department of Mathematics
}\\{\small Komi Research Center of Ural Division of RAS,}\\{\small Syktyvkar State University}
\\{\small  Russia}
}
\maketitle
 \footnote{$^1$Research supported   by SFB 701 ``Spectral Structures and Topological Methods in Mathematics'' University of Bielefeld.
Research supported   by grant  RFBR N~14-01-00500  and by Program of Fundamental Research Ural Division of RAS, Project  12-P-1-1013}

\date{}

\maketitle
\begin{abstract}
Let $\mathbf X=(X_{jk})$ denote a $n\times p$ random matrix with
entries $X_{jk}$, which are independent for $1\le j\le n, 1\le  k\le p$. \tc{Let $n,p$ tend to infinity such that $\frac np=y+O(n^{-1})\in(0,1]$. For those values of $n,p$ we investigate}
the rate of convergence of the expected  spectral distribution function of
the matrix $\mathbf W=\frac1{ p}\mathbf X\mathbf X^*$ to the Marchenko-Pastur  law with parameter $y$. \tc{Assuming the   conditions  $\E X_{jk}=0$,
$\E X_{jk}^2=1$ and}
\begin{equation}
\sup_{n,p\ge1}\sup_{1\le j\le n,1\le k\le p}\E|X_{jk}|^4=: \mu_4<\infty,\quad
\sup_{n,p\ge1}
\sup_{1\le j\le n,1\le k\le p}|X_{jk}|\le Dn^{\frac14},\notag
\end{equation}
we  show that the
 Kolmogorov distance between the {\it expected} spectral
distribution of the sample covariance matrix  $\mathbf W$
and the  Marchenko -- Pastur  law is of order $O(n^{-1})$.
\end{abstract}
\maketitle

\titlecontents{section}[1.5em]{\addvspace{1pt}}{\contentslabel{1.5em}}{}{\titlerule*[0.72pc]{.}\contentspage}

\section{Introduction}
\setcounter{equation}{0}

 The present paper is a continuation of the papers \cite{GT:2014}, \cite{GT:2014a}, where we   proved   
 \tc{non improvable bounds} for the Kolmogorov distance between the expected spectral distribution function of Wigner matrices  and the semicircular 
distribution function. In this paper we estimate the  Kolmogorov distance between the expected spectral distribution function of sample covariance  matrices  
and the Marchenko -- Pastur distribution function.

Consider a family $\mathcal X = \{X_{jk}\}$, $1 \leq j \le n, 1\leq k \leq p$,
of independent real random variables defined on some probability space
$(\Omega,{\textfrak M},\Pr)$, for any $n\ge 1$ and $p\ge1$.  Introduce the  matrices
\begin{displaymath}\notag
 \mathbf X = \ \frac{1}{\sqrt{p}} \left(\begin{array}{cccc}
 X_{11} &  X_{12} & \cdots &  X_{1p} \\
 X_{21} & X_{22} & \cdots &  X_{2p} \\
\vdots & \vdots & \ddots & \vdots \\
 X_{n1} &  X_{n2} & \cdots &  X_{np} \\
\end{array}
\right)
\end{displaymath}
and corresponding sample covariance matrices
$$
\mathbf W=\mathbf X\mathbf X^*.
$$
Here and in what follows we denote by $\mathbf A^*$ the complex conjugate  of matrix $\mathbf A$.
The matrix $\mathbf W$ has a random spectrum $\{s_1^2,\dots,s_n^2\}$ and an
associated spectral empirical distribution function
$\mathcal F_{n}(x) = \frac{1}{n}\,{\rm card}\,\{j \leq n: s_j^2 \leq
x\}, \quad x \in \R$.
Averaging over the random values $X_{ij}(\omega)$, define the expected
(non-random) empirical distribution functions
$ F_{n}(x) = \E\,\mathcal F_{n}(x)$.
Let $G_y(x)$ denote the Marchenko -- Pastur distribution function with parameter $y$ with density
$g_y(x)=G_y'(x)=\frac1{2\pi x}\sqrt{(x-a^2)(b^2-x)}\mathbb I_{[a^2,b^2]}(x)$, where $\mathbb I_{[a^2,b^2]}(x)$
denotes the indicator--function of the interval $[a^2,b^2]$ and $a^2=(1-\sqrt y)^2$ and $b^2=(1+\sqrt y)^2$. 
 The rate of convergence to the Marchenko -- Pastur law has been studied by several authors.  For a  detailed discussion of  previous results see  \cite{GT:2014}
 and \cite{GT:2013a}.
 In what follows we shall assume that $p=p(n)$ such that
\begin{equation}\label{yn}
 y_n:=\frac np, \qquad  |y_n- y| \le c_y n^{-1}, 
\end{equation}
for some constant $c_y >0$.
 We shall estimate  the  Kolmogorov distance between $\mathcal F_n(x)$ and
 the distribution function
 $G_{y}(x)$, \tc{that is,} 
 $\Delta_n:=\sup_x| F_n(x)-G_{y}(x)|$.

The main result of this paper is the following
\begin{thm}\label{main} Let $\E X_{jk}=0$, $\E X_{jk}^2=1$.  Assume that
there exists a constant $\mu_4>0$ such that   
\begin{equation}\label{moment}
 \sup_{n,p\ge1}\,\sup_{1\le j\le n,1\le k\le p}\E|X_{jk}|^4=: \mu_4<\infty.
\end{equation}
 Furthermore, assume that there exists a constant $D$ such that for all $n$
\begin{equation}\label{trun}
 \sup_{1\le j\le n,1\le k\le p}|X_{jk}|\le Dn^{\frac14}.
\end{equation}

Assuming \eqref{yn},  for $y \in (0,1]$  there exists a positive  constant $C=C(D,\mu_4,y,c_y)$ depending on
 $D$, $\mu_4$, $y$ and $c_y$ only
such that,
\begin{equation} \label{kolmog}
\Delta_n=\sup_x| F_n(x)-G_{y}(x)|\le C n^{-1}.
\end{equation}

\end{thm}
\begin{cor}\label{cormain}Let $\E X_{jk}=0$, $\E X_{jk}^2=1$.  Assume that
\begin{equation}\label{moment1}
 \sup_{n\ge1}\,\sup_{1\le j\le k\le n}\E|X_{jk}|^8=: \mu_8<\infty.
\end{equation}
Assuming \eqref{yn}, for any $y\in (0,1]$    there exists  positive  constants $C=C(\mu_8, y, c_y)$  depending on  
 $\mu_8$, $y$ and $c_y$ only
such that,
\begin{equation} \label{kolmog1}
\Delta_n\le C n^{-1}.
\end{equation}
\end{cor}
\begin{rem}
 It is straightforward to check that by assumption \eqref{yn}
\begin{equation}
\sup_x|G_{y_n}(x)-G_y(x)|\le  C n^{-1},
\end{equation}
with a constant $C$ depending on $y$ and $c_y$.
Thus without loss of generality we shall assume in the following proofs that $y=y_n$.
\end{rem}

For any distribution function $F(x)$ we define the Stieltjes transform $s_F(z)$, for $z=u+iv$ with $v>0$, via formula
\begin{equation}
 s_F(z)=\int_{-\infty}^{\infty}\frac1{x-z}dF(x). 
\end{equation}
We introduce \tc{the} symmetrized distribution function
$$
\widetilde{\mathcal F}_n(x)=\frac{1+\tc{\text{ \rm sign} (x)} \mathcal F_n(x^2)}2.
$$
 We denote the Stieltjes transform of $\mathcal F_n(x)$
by $m_n(z)$ and the Stieltjes transform of \tc{the }Marchenko -- Pastur  law with parameter $y$ by $S_y(z)$.  Let
$\mathbf R=\mathbf R(z)$ be the resolvent matrix of $\mathbf W$ given by
$\mathbf R=(\mathbf W-z\mathbf I_n)^{-1}$,
for all $z=u+iv$ with $v\ne 0$. Here and in what follows $\mathbf I_n$ denotes the identity matrix of dimension $n$.
Sometimes we shall omit the \tc{sub-index} in the notation of an identity  matrix.
Denote by $m_n(z)$ the Stieltjes transform of the distribution function $\mathcal F_n(x)$.
It is a well-known fact that
\begin{equation}
 m_n(z)=\frac1n\sum_{j=1}^n\frac1{s^2_j-z}=\frac1n\Tr \mathbf R.
\end{equation}
The Stieltjes transform $S_y(z)$ of the Marchenko --Pastur distribution
satisfies the equation
\begin{equation}\label{stmarpas}
yzS_y^2(z)+(y-1+z)S_y(z)+1=0
\end{equation}
(see, for example, equality (3.10) in \cite{GT:2009}). For the Stieltjes transform $ s_y(z)=zS_y(z^2)$ of the symmetrized Marchenko -- Pastur distribution
we have
\begin{equation}
1+(z+\frac{y-1}z)  s_y(z)+y{ s}^2_y(z)=1
\end{equation}
(see, for instance,  equality (3.11)
in \cite{GT:2009}).
We introduce the $(n+p)\times(n+p)$ Hermitian  matrix by
\begin{equation}\notag
 \mathbf V=\begin{bmatrix}&\mathbf O&\mathbf X\\&\mathbf X^*&\mathbf O\end{bmatrix}.
\end{equation}
It is well known that \tc{the} eigenvalues of the matrix $\mathbf V$ are $-s_1,\ldots,-s_n,s_n,\ldots, s_1$ and $0$ of multiplicity $p-n$.
Introduce the resolvent matrix $\widetilde{\mathbf R}$ of the matrix $\mathbf V$ by
\begin{equation}
 \widetilde{\mathbf R}=(\mathbf V-z\mathbf I_{n+p})^{-1}.
\end{equation}
It is straightforward to check that
\begin{equation}
 \widetilde m_n(z)=zm_n(z^2)=\frac1{2n}\Tr\widetilde {\mathbf R}+\frac{1-y}{2zy}.
\end{equation}

In what follows we shall consider the symmetrized distribution only.\tc{If it is clear from the context we shall omit the symbol $\, \widetilde{\cdot}\, $ in the notation of 
distribution functions, Stieltjes transforms, resolvent matrices and etc.}

Let 
\begin{equation}\label{v0}
v_0:= A_0n^{-1}
\end{equation}
 and $\gamma(z):=\min\{|a-|u||,|b-|u||\}$, for $z=u+iv$.
Introduce the region 
$\mathbb G=\mathbb G(A_0, n,\varepsilon)\subset\mathbb C_+$ 
\begin{equation}\notag
 \mathbb G:=\{z=u+iv\in\mathbb C_+: a+\varepsilon\le |u|\le b-\varepsilon,\,V\ge v\ge v_0/\sqrt{\gamma(z)}\}.
\end{equation}
Let  $\varkappa>0$  be \tc{a} positive number such that
\begin{equation} \label{constant}
 \frac1{\pi}\int_{|u|\le \varkappa}\frac1{u^2+1}du=\frac34.
\end{equation}
\newpage
\tc{On the level of Stieltjes transforms our results are based on the following approximations.}
\begin{thm}\label{stiltjesmain}Let  $\frac12>\varepsilon>0$ be positive numbers \tc{(depending on $v_0$, see \eqref{v0})} such that
\begin{equation}\label{avcond}
 \varepsilon^{\frac32}:=2v_0\varkappa.
\end{equation}Assuming the conditions of Theorem \ref{main}, there exists a positive constant $C=C(D,A_0,\mu_4, y)$ depending on 
$D$, $A_0$, $\mu_4$ and $y$ only, such that, for 
$z\in\mathbb G$
\begin{align}\notag
 |\E m_n(z)-s_y(z)|\le \frac C{n v^{\frac34}}+\frac C{n^{\frac32}v^{\frac32}|(z+\frac{y-1}z)^2-4y|^{\frac14}}.
\end{align}

\end{thm}

For the readers convenience we decided to exposed the full proof for the Marchenko-Pastur case again since in nearly any of the arguments of the 60 page proof of the Wigner result in \cite{GT:2014} adjustments and rewritings for this case had
to be done. For more details see the sketch of the proof  below. 

\newpage

\tableofcontents
\newpage

\bigskip

\section{Sketch of the Proof}

{\noindent \bf 1.} As \tc{in previous work} \cite{GT:2009} we use the so called Hermitization of \tc{a matrix $\mathbf X$ and instead of the spectrum of the matrix $\mathbf W$ we consider the spectrum of the block-matrix} 
$$
\mathbf V=\begin{bmatrix}&\mathbf O&\mathbf X\\&\mathbf X^*&\mathbf O\end{bmatrix},
$$
where $\mathbf O$ denotes the matrix with zero entries. Thus the  proof will be very similar to the proof for the expected
 convergence to the Wigner law  in \cite{GT:2014} but using resolvent identities reflecting  the jump of size $1-y$ at zero in the spectrum. Notable larger changes appear in Section 7 (see point {\bf 4}  below).

{\noindent \bf2.} Furthermore, similarly to \tc{\cite{GT:2014},} we start with an estimate of the Kolmogorov-distance to the \tc{symmetrized} Marchenko -- Pastur  distribution via
an integral over the difference of the corresponding Stieltjes transforms along a contour 
in the upper half-plane using a  smoothing inequality \eqref{smoothing11}. 
The resulting bound \eqref{smoothing11} involves an integral over a segment, say 
$V=4\sqrt y$, at a fixed distance from the real axis and a segment  $u+i A_0 n^{-1}(\min\{(a-\abs{u}),(b-|u|)\}^{-\frac 12}, \; u+iV$,
$a+\varepsilon\le |u|\le b-\varepsilon$
at a distance of order $n^{-1}$ but avoiding to come close to the endpoints $a$ and $b$ of the support.
These segments are part of the boundary of an $n$-dependent region $\mathbb G$ where
bounds of Stieltjes transforms are needed. Since the Stieltjes-transform
and the diagonal elements $R_{jj}(z)$ of the resolvent ($\mathbf R=(\mathbf V-z\mathbf I_{n+p})^{-1}$) of the matrix $\mathbf V$ are uniformly bounded 
on the segment with $\Im z=V$  by $1/V$ (see Section \ref{firsttypeint})  proving a bound of order 
$O(n^{-1})$  for the latter segment near the $x$-axis is the essential problem.

{\noindent \bf 3.} In order to investigate this crucial part of the error 
we start  with the 2nd resolvent  or self-consistency equation for the  Stieltjes transform resp.
the quantities $R_{jj}(z)$ of $\mathbf V$ (see \eqref{repr001} below) based on the difference of the resolvent 
of $\mathbf V^{(j)}$ ($j$th row and column removed) and $\mathbf V$.  
The  necessary bounds of
$\E|R_{jj}|^q$ for large $q=O(1)$ we prove analogously to \cite{GT:2014}.

{\noindent \bf 4.} In Section \ref{expect} we prove a bound  for the error $\E\Lambda_n=\E m_n(z)-s_y(z)$ of the
form 
$n^{-1} v^{-\frac34}+(n v)^{-\frac32}|(z+\frac{y-1}z)^2-4y|^{-\frac14} $
which suffices to prove the rate $O(n^{-1})$ in Theorem  \ref{main}. Here we use a series of martingale-type
decompositions to evaluate the {\it expectation} $\E m_n(z)$ combined  with the bound $\E |\Lambda_n|^2 \le C (nv)^{-2}$ of Lemma  \ref{lam1*} in the Appendix
which is again based on a recursive inequality  for
$\E |\Lambda_n|^2$ in \eqref{finalrek*}. A direct 
application of this bound to estimate the error terms
 $\varepsilon_{j3}$ would result in a less
precise bound of order $O(n^{-1}\log n)$ in Theorem \ref{main}. Bounds of such type will be shown
for the Kolmogorov distance of the empirical {\it random} 
spectral distribution to Marchenko -- Pastur law in a separate
paper. For the expectation we provide sharper bounds
in Section  \ref{better} involving $m'_n(z)$. Note here that in the Marchenko-Pastur case a new term  $c/z^2$ appears.

{\noindent \bf 5.}
The necessary auxiliary bounds for all these steps
are collected in the Appendix.

\section{Bounds for the  Kolmogorov Distance Between Distribution Functions  via  Stieltjes Transforms}\label{smoothviastil}
To bound $\Delta_n$ we shall use an approach developed in
G\"otze and Tikhomirov \cite{GT:2009} and \cite{GT:2004}. We modify a bound for the  
Kolmogorov distance between
distribution functions based on their Stieltjes transforms obtained in \cite{GT:2005}, Lemma 2.1.
Let $\widetilde  G_y(x)$ denote the distribution function defined by the equality
\begin{equation}\label{symmetr}
\widetilde  G_y(x)=\frac{1+\tc{\text{\rm sign}\4(x)} \4G_y(x^2)}2,
\end{equation}
Recall that $G_y(x)$ is Marchenko--Pastur distribution function with parameter $y\in(0,1])$.
The distribution function $\widetilde  G_y(x)$ has a density
\begin{equation}\label{symmetr1}
 \widetilde G_y'(x)=\frac1{2\pi |x|}\sqrt{(x^2-a^2)(b^2-x^2)}\mathbb I\{a\le |x|\le b\}.
\end{equation}
For $y=1$ the distribution function $\widetilde  G_y(x)$ is the distribution function of 
the semi-circular law.
 Given $\frac{\sqrt y}2\ge\varepsilon>0$ introduce the interval $\mathbb J_{\varepsilon}=
[1-\sqrt y+\varepsilon,1+\sqrt y-\varepsilon]$ and
$\mathbb J'_{\varepsilon}=
[1-\sqrt y+\frac12\varepsilon,1+\sqrt y-\frac12\varepsilon]$.
For any $x$ such that $|x|\in[1-\sqrt y,1+\sqrt y]$, define
$\gamma=\gamma(x):=\sqrt y-||x|-1|$. Note that $0\le\gamma\le\sqrt y$.
For any $x:\,|x|\in\mathbb J_{\varepsilon}$, we have $\gamma\ge\varepsilon$, respectively, 
for any $x:\,|x|\in\mathbb J_{\varepsilon}'$,
 we have $\gamma\ge\frac12\varepsilon$.
For a  distribution function $F$ denote by $S_F(z)$ its Stieltjes transform,
$$
S_F(z)=\int_{-\infty}^{\infty}\frac1{x-z}dF(x).
$$
\begin{prop}\label{smoothing}Let $v>0$ and $H>0$ and $\varepsilon>0$ be positive numbers such 
that
\begin{equation}\notag
 \tau=\frac1{\pi}\int_{|u|\le H}\frac1{u^2+1}du=\frac34,
\end{equation}
and
\begin{equation}\label{avcond1}
 2vH\le \varepsilon^{\frac32}.
\end{equation}
If $\widetilde G_y$ denotes the distribution function of the symmetrized (as in \eqref{symmetr})
Marchenko--Pastur law, and $F$ is any distribution function,
 there exist some absolute constants $C_1$, $C_2$, $C_3$ depending on $y$ only such that
\begin{align}
\Delta(F,\widetilde G_y)&:= \sup_{x}|F(x)-\widetilde G_y(x)|\notag\\&\le
2\sup_{x:|x|\in\mathbb J'_{\varepsilon}}\Big|\im\int_{-\infty}^x(S_F(u+i\frac v{\sqrt{\gamma}})
-S_{\widetilde G_y}(u+i\frac v{\sqrt{\gamma}}))du\Big|+C_1v
+C_2\varepsilon^{\frac32}\notag
\end{align}
with $C_1=\begin{cases}\frac {2H^2\sqrt 3}{\pi^2\sqrt{y(1-\sqrt y)}}&\text{ if }0<y<1,\\ 
\frac{H^2}{\pi}&\text{ if }y=1,\end{cases}$ and 
$C_2=\begin{cases}\frac {4}{\pi \sqrt{y(1-\sqrt y)}}&\text{ if }0<y<1,\\ \frac1{\pi}
&\text{ if }y=1.\end{cases}$.

\end{prop}	
\begin{rem}
 \begin{equation}\notag
  H=\text{\rm tg}\frac{3\pi}8=1+\sqrt 2.
 \end{equation}

\end{rem}

\begin{cor}\label{Cauchy}
 Under the conditions  of Proposition \ref{smoothing}, for any $V>v$, the following 
inequality holds
\begin{align}
 \sup_{x\in\mathbb J'_{\varepsilon}}&\left|\int_{-\infty}^x(\im(S_F(u+iv')
-S_{\widetilde G_y}(u+iv'))du\right|\notag\\&\qquad\le
\int_{-\infty}^{\infty}|S_F(u+iV)-S_{\widetilde G_y}(u+iV)|du\notag\\&\qquad\qquad+
\sup_{x\in\mathbb J'_{\varepsilon}}\left|\int_{v'}^V\left(S_F(x+iu)-S_{\widetilde G_y}(x+iu)\right)du
\right|.\notag
\end{align}

\end{cor}
\begin{proof}Let $x:\4|x|\in\mathbb J'_{\varepsilon}$ be fixed. Let 
$\gamma=\gamma(x)=\min\{|x|-1+\sqrt y,\,1+\sqrt y-|x|\}$.
 Set $z=u+iv'$ with $v'=\frac v{\sqrt{\gamma}}$,
$v'\le V$. Since the functions of $S_F(z)$ and $S_{\widetilde G_y}(z)$ are analytic in 
the upper half-plane, it is enough to use Cauchy's theorem. We can write
\begin{equation}\notag
\int_{-\infty}^{x}\im(S_F(z)-S_{\widetilde G_y}(z))du=\lim_{L\to\infty}\int_{-L}^x
(S_F(u+iv')-S_{\widetilde G_y}(u+iv'))du,
\end{equation}
for $x\in \mathcal J'_{\varepsilon}$.
Since $v'=\frac v{\sqrt{\gamma}}\le \frac{\varepsilon}{2H}$, without loss of generality 
we may assume that $v'\le 2$.
By Cauchy's integral formula, we have
\begin{align}
 \int_{-L}^x(S_F(z)-S_{\widetilde G_y}(z))du&=\int_{-L}^x(S_F(u+iV)
-S_{\widetilde G_y}(u+iV))du\notag\\&\qquad
+\int_{v'}^V(S_F(-L+iu)-S_{\widetilde G_y}(-L+iu))du\notag\\&\qquad\qquad-\int_{v'}^V(S_F(x+iu)
-S_{\widetilde G_y}(x+iu))du.\notag
\end{align}
Denote by $\xi$ (resp. $\eta$) a random variable with distribution function $F(x)$ 
(resp. $\widetilde G_y(x)$). Then we have
\begin{equation}
 |S_F(-L+iv')|=\left|\E\frac1{\xi+L-iv'}\right|\le {v'}^{-1}\Pr\{|\xi|>L/2\}+\frac2L.
\end{equation}
Similarly,
\begin{equation}\notag
 |S_{\widetilde G_y}(-L+iv')|\le {v'}^{-1}\Pr\{|\eta|>L/2\}+\frac2L.
\end{equation}
These inequalities imply that
\begin{equation}
\left|\int_{v'}^V(S_F(-L+iu)-S_{G_y}(-L+iu))du\right|\to 0\quad\text{as}\quad L\to\infty,
\end{equation}
which completes the proof.
\end{proof}
Combining the results of Proposition \ref{smoothing} and Corollary \ref{Cauchy}, we get
\begin{cor}\label{smoothing2}
 Under the conditions of Proposition \ref{smoothing} the following inequality holds
\begin{align}\label{smoothing11}
 \Delta(F,\widetilde G_y)&\le 2\int_{-\infty}^{\infty}|S_F(u+iV)-S_{\widetilde G_y}(u+iV)|du
+C_1v+C_2\varepsilon^{\frac32}\notag\\&
  +2\sup_{x\in\mathbb J'_{\varepsilon}}\int_{v'}^V|S_F(x+iu)-S_{\widetilde G_y}(x+iu)|du,
\end{align}
where $v'=\frac v{\sqrt{\gamma}}$ with $\gamma=\min\{|x|-1+\sqrt y,\,1+\sqrt y-|x|\}$.

\end{cor}
\section{The proof of Theorem \ref{main}}
We shall apply Corollary \ref{smoothing2} to bound \tc{the}
Kolmogorov distance between\tc{the} expected spectral distribution $ F_n$ and |tc{the}
Marchenko--Pastur distribution $G_y$. 
We denote the Stieltjes transform of $\mathcal F_n(x)$
by $m_n(z)$ and \tc{the} Stieltjes transform of the Marchenko--Pastur law by $s_y(z)$.
We shall use the ``symmetrization'' of the spectrum sample covariance matrix as in  
\cite{GT:2009}.
Introduce the $(p+n)\times (p+n)$ matrix
\begin{equation}
 \mathbf V=\begin{bmatrix}& \mathbf O\quad&\mathbf X\\&\mathbf X^*
\quad&\mathbf O\end{bmatrix},
\end{equation}
where $\mathbf O$ denotes a matrix with zero entries.
Note that the eigenvalues of the matrix $\mathbf V$ are $\pm s_1,\ldots,\pm s_n,$ and $0$ 
with multiplicity $p-n$.
Let
$\mathbf R=\mathbf R(z)$ denote the resolvent matrix of  $\mathbf V$ defined by the equality
$$
\mathbf R=(\mathbf V-z\mathbf I_{n+p})^{-1},
$$
for all $z=u+iv$ with $v\ne 0$. Here and in what follows $\mathbf I_k$ denotes the identity 
matrix of order $k$.
Sometimes we shall omit the  sub index in the notation of the identity  matrix.
If we consider the Stieltjes transforms $s_y(z)$ of the ``symmetrized'' Marchenko-Pastur 
distribution $\widetilde G_y(x)$
(see formula \eqref{symmetr}), then it is straightforward to check that $s_y(z)=z\4 S_y(z^2)$ 
and
\begin{equation}\label{reprst}
ys_y^2(z)+(\frac{y-1}z+z) s_y(z)+1=0
\end{equation}
(see Section 3 in \cite{GT:2009}).
Furthermore, for the
Stieltjes transform $\widetilde m_n(z)$  of the ``symmetrized`` empirical  spectral distribution function
\begin{equation}\label{symmetr0}
\widetilde{\mathcal F}_n(x)=\frac{1+\tc{\text{\rm sign} \4 (x)} \4 \mathcal F_n(x^2)}2,
\end{equation}
we have
\begin{equation}\notag \widetilde m_n(z)=\frac1n\sum_{j=1}^n
R_{jj}=\frac1n\sum_{j=n+1}^{n+p}R_{jj}+\frac{1-y}{yz}
\end{equation}
 (see, for instance, Section 3
in \cite{GT:2009}).
Note that \tc{by  definition of  a symmetrized distribution \eqref{symmetr0} we have}
\begin{equation}\notag
 \sup_x|\mathcal F_n(x)-G_y(x)|=2\sup_x|\widetilde {\mathcal F}_n(x)-\widetilde G_y(x)|.
\end{equation}
In  what follows we shall consider symmetrized random values   only.
We shall omit the symbol $''\,\widetilde{\cdot}\,''$ in the notation of the distribution 
function and its Stieltjes transform.
Let $\mathbb T_j=\{1,\ldots,n\}\setminus \{j\}$.
 For $j=1,\ldots, n$, introduce the matrices $\mathbf V^{(j)}$,
 obtained from $\mathbf V$ by deleting the $j$-th row and $j$-th column, and define the 
corresponding  resolvent matrix $\mathbf R^{(j)}$  by the equality
$\mathbf R^{(j)}=(\mathbf V^{(j)}-z\mathbf I_{n+p-1})^{-1}$.  Using \tc{the Schur} decomposition formula we may show that
\begin{align}\label{matrshur}
 \mathbf R=\begin{bmatrix}&z(\mathbf X\mathbf X^*-z^2\mathbf I)^{-1}&\mathbf X(\mathbf X^*\mathbf X-z^2\mathbf I)^{-1}\\
           &(\mathbf X^*\mathbf X-z^2\mathbf I)^{-1}\mathbf X^*& z(\mathbf X^*\mathbf X-z^2\mathbf I)^{-1}
           \end{bmatrix}.
\end{align}
From this representation it follows
\begin{equation}\notag
 \frac1p\sum_{k=1}^pR_{k+n,k+n}=ym_n(z)-\frac{1-y}z.
\end{equation}

We shall use the representation, for $j=1,\ldots, n$,
\be\label{diagres}
R_{jj}=\frac1{-z-\frac1p{{\sum_{k,l=1}^p}}X_{jk}X_{jl}R^{(j)}_{k+n,l+n}}
\en
(see,
for example, Section 3 in \cite{GT:2009}). We may rewrite it as
follows
\be\label{repr01}
 R_{jj}=-\frac1{z+ym_n(z)+\frac{y-1}z}+
\frac1{z+ym_n(z)+\frac{y-1}z}\varepsilon_jR_{jj}, \en
where
$\varepsilon_j=\varepsilon_{j1}+\varepsilon_{j2}+\varepsilon_{j3}$ and
\begin{align}
\varepsilon_{j1}&:=\frac1p{\sum_{k=1}^p}(X_{jk}^2-1)R^{(j)}_{k+n,k+n},\quad
\varepsilon_{j2}:=\frac1p{\sum_{1\le k\ne
l\le p}}X_{jk}X_{jl}R^{(j)}_{k+n,l+n},\notag\\
\varepsilon_{j3}&:=\frac1p\Bigl(\sum_{l=1}^p R^{(j)}_{l+n,l+n}-\sum_{l=1}^p R_{l+n,l+n}\Bigr).
\notag
\end{align}


 We choose $V=4\sqrt y$ and $v_0$ as defined in \eqref{v0} and introduce the quantity 
 $\varepsilon=(2av_0)^{\frac23}$ .
 We shall denote in what follows by $C$  a generic constant depending on $\mu_4$ and $D$ only.
\subsection{Estimation of the First Integral in \eqref{smoothing11}  for $V=4\sqrt y$}\label{firsttypeint}
Denote by $\mathbb T=\{1,\ldots,n\}$ and by $\mathbb T_{\mathbb A}=\mathbb T\setminus\mathbb A$ ($\mathbb A\subset\mathbb T$). In the following we shall systematically
  use for any  $n\times p $ matrix $\mathbf X$
 together with its resolvent $\mathbf R$, its Stieltjes transform $m_n$ etc. the corresponding quantities $\mathbf X^{(\mathbb A)}$, $\mathbf R^{(\mathbb A)}$ 
 and $m_n^{(\mathbb A)}$
 for the corresponding  sub matrix  with entries $X_{jk}, j \in \mathbb T_{\mathbb A}$ , $k=1,\ldots,p$.
Observe that
\begin{equation}\label{maj}
 m_n^{(\mathbb A)}(z)=\frac1n\sum_{j\in\mathbb T_{\mathbb A}}\frac z{(s^{(\mathbb A)}_j)^2-z^2}.
\end{equation}
By $\mathfrak M^{(\mathbb A)}$ we denote the $\sigma$-algebra generated by $X_{lk}$ with $l\in\mathbb T_{\mathbb A}$, $k=1,\ldots,p$.
If $\mathbb A=\emptyset$ we shall omit the set $\mathbb A$ as exponent index. 

In this Section we shall consider $z=u+iV$ with $V=4\sqrt y$. We shall use the representation \eqref{repr01}.

Let
\begin{equation}\notag
\Lambda_n:=\Lambda_n(z):=m_n(z)-s_y(z)=\frac1n\sum_{j=1}^nR_{jj}-s_y(z).    
\end{equation}
\tc{It follows from \eqref{stmarpas}, that for the symmetrized Marchenko -- Pastur law we have}
\begin{equation}\label{scmain0}
 s_y(z)=-\frac1{z+\frac{y-1}z+ys_y(z)}\text{ and }|s_y(z)|\le 1/\sqrt y.
\end{equation}
See, for instance \cite{GT:2013a}, Lemma 9.3.
Summing \tc{the equalities} \eqref{repr01} in $j=1,\ldots,n$ and solving with respect $\Lambda_n$, we get 
\begin{equation}\label{lambda'}
\Lambda_n=  m_n(z)-s_y(z)=\frac{T_n}{z+y(m_n(z)+s_y(z))+\frac{y-1}z},
\end{equation}
where
\begin{equation}\notag
 T_n=\frac1n\sum_{j=1}^n\varepsilon_jR_{jj}.
\end{equation}
Note that for $V=4\sqrt y$
\begin{align}\label{3.5}
y \max\{|s_y(z)|,|m_n(z)|\}&\le \frac{\sqrt y}{4}\le \frac12|z+ys_y(z)+\frac{y-1}z|, \notag\\ y|s_y(z)-m_n(z)|\le &\frac{\sqrt y}2\le \frac12|z+ys_y(z)+\frac{y-1}z| \text{ a.s.}
\end{align}
This implies
\begin{align}\label{inequal10}
 |z+y(m_n(z)+s_y(z))+\frac{y-1}z|&\ge\frac12|z+ys_y(z)+\frac{y-1}z|=\frac1{2|s_y(z)|},\notag\\  |z+ym_n(z)+\frac{y-1}z|&\ge \frac12|ys_y(z)+z+\frac{y-1}z|.
\end{align}
The last inequalities and equality \eqref{lambda'} imply as well that, for $V=4\sqrt y$,
\begin{equation}\label{mnz}
|m_n(z)|\le |s_y(z)|(1+2|T_n(z)|).
\end{equation}

Using  equality \eqref{lambda'}, we may write, 
\begin{align}\label{lal1}
 \E\Lambda_n&=\E\frac1{n(z+y(m_n(z)+s_y(z))+\frac{y-1}z)}\sum_{j=1}^n(\varepsilon_{j1}+\varepsilon_{j2})\notag\\&
+\E\frac1{n(z+y(m_n(z)+s_y(z))+\frac{y-1}z)}\sum_{j=1}^n\varepsilon_{j3}R_{jj} \notag\\&+\E\frac1{n(z+m_n(z)
+s_y(z)+\frac{y-1}z)^2}\sum_{j=1}^n(\varepsilon_{j1}+\varepsilon_{j2})\varepsilon_{j}R_{jj}.
\end{align}
First we note that, by \eqref{matrshur}, for $\nu=1,2$
\begin{align}
 \E&\frac1{n(z+y(m_n(z)+s_y(z))+\frac{y-1}z)}\sum_{j=1}^n\varepsilon_{j\nu}=
 \sum_{j=1}^n\E\frac1{n(z+y(m_n^{(j)}(z)+s_y(z))+\frac{y-1}z)}\varepsilon_{j\nu}\notag\\&+
 \E\frac1{n(z+y(m_n(z)+s_y(z))+\frac{y-1}z)}\sum_{j=1}^n\frac{\varepsilon_{j\nu}\varepsilon_{j3}}{z+y(m_n^{(j)}(z)+s_y(z))+\frac{y-1}z}.\notag
\end{align}
Observe that, for $\nu=1,2$, 
\begin{equation}\label{lal5}
 \E\frac1{n(z+y(m_n^{(j)}(z)+s_y(z))+\frac{y-1}z)}\varepsilon_{j\nu}=0.
\end{equation}
Furthermore, using relations \eqref{scmain0} and \eqref{3.5}, we get
\begin{align}
|\E\frac1{n(z+y(m_n(z)+s_y(z))+\frac{y-1}z)}&\sum_{j=1}^n\frac{\varepsilon_{j\nu}\varepsilon_{j3}}{z+y(m_n^{(j)}(z)+s_y(z))+\frac{y-1}z}| \notag\\&\qquad\qquad\qquad\le
\frac{C|s_y(z)|^2}n\sum_{j=1}^n\E|\varepsilon_{j\nu}\varepsilon_{j3}|.\notag
\end{align}
Applying Lemmas \ref{basic2}, \ref{basic5},  and \ref{basic8}, we conclude
\begin{align}\label{lal2}
 \Big|\E\frac1{n(z+y(m_n(z)+s_y(z))+\frac{y-1}z)}\sum_{j=1}^n&\frac{\varepsilon_{j\nu}
 \varepsilon_{j3}}{z+y(m_n^{(j)}(z)+s_y(z))+\frac{y-1}z}\Big|
 \notag\\&\qquad\qquad\qquad\le Cn^{-1}|s_y(z)|^2.
\end{align}
We note now that, for $j=1,\ldots,n$
\begin{align}
 \sum_{l=1}^pR_{l+n,l+n}=nm_n(z)-\frac{p-n}z,\quad
 \sum_{l=1}^pR^{(j)}_{l+n,l+n}=nm_n^{(j)}(z)-\frac{p-n+1}z.\notag
\end{align}
This implies that
\begin{equation}\label{rr21}
 \frac1p(\sum_{l=1}^pR_{l+n,l+n}-\sum_{l=1}^pR^{(j)}_{l+n,l+n})=y(m_n(z)-m_n^{(j)}(z))+\frac1{pz}.
\end{equation}
\tc{On the} other hand side
\begin{align}
 \Tr \mathbf R=2nm_n(z)-\frac{p-n}z,\quad
 \Tr \mathbf R^{(j)}=2nm_n^{(j)}(z)-\frac{p-n+1}z.\notag
\end{align}
We get
\begin{equation}\label{rr22}
 m_n(z)-m_n^{(j)}(z)=\frac1{2n}(\Tr \mathbf R-\Tr \mathbf R^{(j)})-\frac1{2nz}.
\end{equation}
Comparing \eqref{rr21} and \eqref{rr22}, we conclude
\begin{align}
 \varepsilon_{j3}=\frac1p(\sum_{l=1}^pR_{l+n,l+n}-\sum_{l=1}^pR^{(j)}_{l+n,l+n})=\frac y{2n}(\Tr \mathbf R-\Tr \mathbf R^{(j)})+\frac y{2nz}.
\end{align}
Multiply this equality by $R_{jj}$, summing in $j=1,\ldots,n$, and using that $(\Tr\mathbf R-\Tr\mathbf R^{(j)})R_{jj}=-\frac{d}{dz}R_{jj}$,
we get
\begin{align}\label{lal3}
 \frac1n\sum_{j=1}^n\varepsilon_{jj}R_{jj}=-\frac d{dz}(\frac yn\sum_{j=1}^nR_{jj})+\frac y{2nz}\frac1n\sum_{j=1}^nR_{jj}=
 -y\frac d{dz}m_n(z)+\frac y{2nz}m_n(z).
 \end{align}
Continuing with
\begin{equation}\notag
 |\frac d{dz}m_n(z)|=|\frac d{dz}(\frac1{2n}\Tr\mathbf R+\frac{p-n}{2nz})|=|\frac1{2n}\Tr\mathbf R^2-\frac{p-n}{2nz^2}|
\end{equation}
and
\begin{equation}\notag
 |\Tr \mathbf R^2|\le v^{-1}\im(\Tr\mathbf R),
\end{equation}
we get
\begin{equation}\notag
 |\frac d{dz}m_n(z)|\le \frac C{v}\im m_n(z)+\frac {C(1-y)}{|z|^2}.
\end{equation}
The last inequality and inequality \eqref{inequal10} together imply, 
for $V=4\sqrt y$, 
\begin{align}\notag
 |\frac1{n(z+m_n(z)+s_y(z)+\frac{y-1}z)}\sum_{j=1}^n\varepsilon_{j3}R_{jj} |\le \frac{C|s_y(z)|^2}{n}+\frac {C(1-y)|s_y(z)|}{n|z|^2}.
\end{align}
Applying inequalities \eqref{inequal10}, we get, for $V=4\sqrt y$,
\begin{align}
\Big|\E\frac1{n(z+m_n(z)+s_y(z)+\frac{y-1}z)}\sum_{j=1}^n
\frac{\varepsilon_{j\nu}\varepsilon_{j3}}{z+m_n(z)+s_y(z)+\frac{y-1}z}\Big|\notag\\ \le
\frac {C|s_y(z)|^2}n\sum_{j=1}^n\E|\varepsilon_{j\nu}\varepsilon_j|.\notag
\end{align}
According to Lemmas \ref{basic2}, \ref{basic5}, and \ref{basic8}, we obtain
\begin{align}\label{lal4}
 \Big|\E\frac1{n(z+y(m_n(z)+s_y(z))+\frac{y-1}z)}\sum_{j=1}^n
\frac{\varepsilon_{j\nu}\varepsilon_{j3}}{z+y(m_n(z)+s_y(z))+\frac{y-1}z}\Big| \le
\frac {C|s_y(z)|^2}n.
\end{align}
Combining now inequalities   \eqref{lal2}, \eqref{lal3}, \eqref{lal1}, and relations \eqref{lal1}, \eqref{lal5}, we conclude
\begin{equation}\notag
 |\E\Lambda_n|\le \frac{C|s_y(z)|^2}n+\frac{C|s_y(z)|}{n|z|^2}.
\end{equation}
Furthermore
\begin{equation}\notag
 \int_{-\infty}^{\infty}|s_y(u+iV)|^2du\le \int_{-\infty}^{\infty}(\int_{-\infty}^{\infty}\frac1{(x-u)^2+V^2}du)dG_y(x)\le \frac1{2\pi V},
\end{equation}
 and
 \begin{align}\notag
 \int_{-\infty}^{\infty}\frac{du}{u^2+V^2}\le \frac1{\pi V}.
 \end{align}

These inequalities  imply that
\begin{equation}\label{final!}
 \int_{-\infty}^{\infty}|\E\Lambda_n(u+iV)|du\le \frac Cn.
\end{equation}

\subsection{The Bound of the \tc{Second Integral} in \eqref{smoothing11}}
To finish the proof of Theorem \ref{main} we need to bound the second integral
in \eqref{smoothing2} for $z\in\mathbb G$, $v_0=C_7n^{-1}$ and $\varepsilon=C_8v_0^{\frac23}$, where the constant $C_8$ is chosen such that so that condition 
\eqref{avcond1} holds.
We shall use the results of Theorem \ref{stiltjesmain}.
According to these results we have, for $z\in\mathbb G$,
\begin{equation}\label{jpf}
 |\E\Lambda_n|\le\frac C{n v^{\frac34}}+\frac C{n^{\frac32}v^{\frac32}|(z+\frac{y-1}z)^2-4y|^{\frac14}}.
\end{equation}
We have
\begin{align}
\int_{v_0/\sqrt{\gamma}}^V|\E(m_n(x+iv)-s(x+iv))|dv&\le
 \frac1{n}\int_{\frac{v_0}{\sqrt{\gamma}}}^V\frac{dv}{v^{\frac34}}+\frac1{n\sqrt n\gamma^{\frac14}}\int_{\frac{v_0}{\sqrt{\gamma}}}^V\frac{dv}{v^{\frac32}}.\notag
\end{align}
After integrating we get
\begin{equation}\label{final!!}
 \int_{v_0/\sqrt{\gamma}}^V|\E(m_n(x+iv)-s_y(x+iv))|dv\le \frac C{n}+\frac{ C\gamma^{\frac14}}{n\sqrt n\gamma^{\frac14}v_0^{\frac12}}\le \frac Cn.
\end{equation}

Inequalities \eqref{final!} and \eqref{final!!} complete the proof of Theorem \ref{main}.
Thus Theorem \ref{main} is proved.

\section{Proof of Corollary \ref{cormain}}To prove the Corollary \ref{cormain}
we consider truncated random variables $\widehat X_{jl}$ defined by
\begin{equation}\label{trunc000}
 \widehat X_{jl}:=X_{jl}\mathbb I\{|X_{jl}|\le cn^{\frac14} \}
\end{equation}
and introduce  matrices
\begin{equation}\notag
 \widehat{\mathbf X}=\frac1{\sqrt p}(\widehat X_{jl}), \quad \widehat{\mathbf V}=\begin{bmatrix}&\mathbf O&\widehat{\mathbf X}\\&(\widehat{\mathbf X})^*&\mathbf O\end{bmatrix}.
\end{equation}
Let $\widehat {\mathcal F}_n(x)$ denote the empirical spectral distribution function of the matrix $\widehat{\mathbf V}$.
\begin{lem}\label{trunc}
 Assuming the conditions of Theorem \ref{main} there exists a  constant $C>0$ depending on $\mu_8$ only such that
 \begin{equation}\notag
\E\{\sup_x|\mathcal F_n(z)-\widehat {\mathcal F}_n(x)|\}\le \frac{C}{n}.
 \end{equation}

\end{lem}
\begin{proof}
 We shall use  the rank inequality of Bai. See \cite{BaiSilv:2010},  Theorem A.43, p. 503.
 According this inequality
 \begin{equation}\notag
  \E\{\sup_x|{\mathcal F}_n(x)-\widehat {\mathcal F}_n(x)|\}\le \frac 2n\E\{\text{\rm rank}(\mathbf X-\widehat{\mathbf X})\}.
 \end{equation}
Observing that the rank of a matrix is not larger then numbers of its non-zero entries, we may write
\begin{align}
 \E\{\sup_x|{\mathcal F}_n(x)-\widehat{\mathcal F}_n(x)|\}&\le \frac 2n\sum_{j=1}^n\sum_{k=1}^p\E\mathbb I\{|X_{jk}|\ge Cn^{\frac14}\}\notag\\&\le \frac1{n^3}\sum_{j,k=1}^n\E|X_{jk}|^8\le \frac{C\mu_8}{n}.\notag
\end{align}
Thus, the Lemma is proved.
\end{proof}
 
We shall compare the Stieltjes transform of the matrix 
$\widehat{\mathbf V}$ and the matrix obtained from $\widehat{\mathbf V}$ by centralizing and normalizing its entries. 
Introduce  $\widetilde X_{jk}=\widehat X_{jk}-\E\widehat X_{jk}$ and  $\widetilde{\mathbf X}=\frac1{\sqrt p}(\widetilde X_{jk})_{j,k=1}^n$. 
We normalize the r.v.'s $\widetilde {X}_{jk}$. 
Let $\sigma_{jk}^2=\E|\widetilde X_{jk}|^2$. We define the r.v.'s 
$\breve X_{jk}=\sigma_{jk}^{-1}\widetilde X_{jk}$. Let $\breve{\mathbf X}=
\frac1{\sqrt p}(\breve X_{jk})_{j,k=1}^n$.
Finally, let $\breve m_n(z)$ denote Stieltjes transform of empirical spectral distribution 
function of the matrix $\breve{\mathbf V}=\begin{bmatrix}&\mathbf O&\breve{\mathbf X}\\&{\breve{\mathbf X}}^*&\mathbf O\end{bmatrix}$.
\begin{rem}\label{trunc00}
Note that
 \begin{equation}\label{trunc2}
  |\breve X_{jl}|\le D_1n^{\frac14}, \quad\E \breve X_{jl}=0 \text{ and }\E {\breve X}_{jk}^2=1,
 \end{equation}
 for some absolute constant $D_1$. That means that the matrix $\breve{\mathbf X}$ satisfies the conditions of Theorem \ref{stiltjesmain}.
\end{rem}
\begin{lem}\label{trunc2*}There exists some absolute constant $C$ depending on $\mu_8$ such that
\begin{equation}\notag
 |\E(\widetilde m_n(z)-\breve m_n(z))|\le \frac{C}{n^{\frac32}v^{\frac32}}.
 \end{equation}
\end{lem}
\begin{proof}Note that
\begin{align}
 \breve m_n(z)&=\frac1{2n}\Tr(\breve{\mathbf V}-z\mathbf I)^{-1}+\frac{p-n}{2nz}=:\frac1{2n}\Tr\breve{\mathbf R}+\frac{p-n}{2nz},\notag\\
 \widetilde m_n(z)&=\frac1{2n}\Tr(\widetilde{\mathbf V}-z\mathbf I)^{-1}+\frac{p-n}{2nz}=:\frac1{2n}\Tr\widetilde{\mathbf R}+\frac{p-n}{2nz}.\notag
\end{align}
Therefore,
\begin{equation}\label{su1}
 \widetilde m_n(z)-\breve m_n(z)=\frac1{2n}\Tr(\widetilde{\mathbf R}-\breve{\mathbf R})
 =\frac1n\Tr (\widetilde{\mathbf V}-\breve{\mathbf V})\widetilde{\mathbf R}\breve{\mathbf R}.
\end{equation}
Using the simple inequalities $|\Tr\mathbf A\mathbf B|\le \|\mathbf A\|_2\|\mathbf B\|_2$ and $\|\mathbf A\mathbf B\|_2\le\|\mathbf A\|\|\mathbf B\|_2$, we get
\begin{equation}\label{dif1}
|\E(\widetilde m_n(z)-\breve m_n(z))|\le n^{-1}\E^{\frac12}\|\widetilde{\mathbf R}\|^2\|\breve{\mathbf R}\|_2^2\E^{\frac12}\|\widetilde{\mathbf V}-\breve{\mathbf V}\|_2^2.
\end{equation}
Furthermore, we note that, 
\begin{align}\label{dif2}
 \widetilde{\mathbf V}-\breve{\mathbf V}=\frac1{\sqrt p}
 \begin{bmatrix}&\mathbf O&\widetilde{\mathbf X}-\breve{\mathbf X}\\&
(\widetilde{\mathbf X}-\breve{\mathbf X})^*&\mathbf O \end{bmatrix}
\end{align}
and
\begin{equation}\notag
 \|\widetilde{\mathbf V}-\breve{\mathbf V}\|_2\le 2\max_{1\le j,k\le n}\{1-\sigma_{jk}\}\|\breve{\mathbf X}\|_2.
\end{equation}
Since 
\begin{equation}\notag
 0<1-\sigma_{jk}\le 1-\sigma_{jk}^2\le Cn^{-\frac32}\mu_8,
\end{equation}
therefore
\begin{equation}\label{ss1}
\E \|\widetilde{\mathbf V}-\breve{\mathbf V}\|_2^2\le C\mu_8^2n^{-2}.
\end{equation}

Applying Lemma \ref{resol00}, inequality \eqref{res1}, in the Appendix and inequality \eqref{ss1}, we obtain
\begin{align}
 |\E(\widetilde m_n(z)-\breve m_n(z))|\le Cn^{-\frac32}v^{-\frac32}(\frac1n\sum_{j=1}^n\E|\breve{\mathbf R}_{jj}|)^{\frac12}.\notag
\end{align}
According Remark \ref{trunc00}, we may apply Corollary \ref{cor8} in Section \ref{diag} with $q=1$ to prove the claim.
Thus, Lemma \ref{trunc2*} is proved.
\end{proof}

Denote by $\widetilde m_n(z)$ the Stieltjes transform of the empirical distribution function of the matrix $\widetilde{\mathbf V}$ and let $\widehat m_n(z)$ 
denote the  Stieltjes transform of the matrix
$\widehat{\mathbf V}$.
\begin{lem}\label{trunc1}For some absolute constant $C>0$ we have
\begin{equation}\notag
 |\E(\widetilde m_n(z)-\widehat m_n(z))|\le \frac{C\mu_8}{n^{\frac32}v^{\frac32}}.
 \end{equation}
\end{lem}
\begin{proof}
Similar to \eqref{su1}, we write
\begin{equation}\notag
 \widetilde m_n(z)-\widehat m_n(z)=\frac1n\Tr(\widetilde{\mathbf R}-\widehat{\mathbf R})
 =\frac1n\Tr (\widetilde{\mathbf V}-\widehat{\mathbf V})\widetilde{\mathbf R}\widehat{\mathbf R}.
\end{equation}
This yields
\begin{equation}\label{dif1*}
\E|\widetilde m_n(z)-\widehat m_n(z)|\le n^{-1}\E\|\widehat{\mathbf R}\|\|\widetilde{\mathbf R}\|_2\|\E\widehat{\mathbf V}\|_2.
\end{equation}
Furthermore, we note that, by definition \eqref{trunc000} and condition \eqref{moment1}, we have
\begin{align}\label{dif2*}
 |\E \widehat X_{jk}|&\le Cn^{-\frac 74}\mu_8.
\end{align}
Applying Lemma \ref{resol00}, inequality \eqref{res1}, in the Appendix and inequality \eqref{dif2*}, we obtain using $\|\widehat{\mathbf R}\|\le v^{-1}$,
\begin{align}\notag
 \E|\widetilde m_n(z)-\widehat m_n(z)|\le n^{-\frac74}v^{-\frac32}\E^{\frac12}|\widetilde m_n(z)|.
\end{align}
By Lemma \ref{trunc2*},
\begin{equation}\notag
 \E|\widetilde m_n(z)|\le \E|\breve m_n(z)|+C,
\end{equation}
for some constant $C$ depending on $\mu_8$ and $A_0$. According to Corollary \ref{cor8} in Section \ref{diag} with $q=1$
\begin{equation}\notag
 \E|\breve m_n(z)|\le \frac1n\sum_{j=1}^n\E|\breve R_{jj}|\le C,
\end{equation}
with a constant $C$ depending on $\mu_4$, $D$.
Using these inequalities, we get
\begin{equation}\notag
|\E\widetilde m_n(z)-\widehat m_n(z)|\le\frac {C\mu_8}{n^{\frac74}v^{\frac32}}\le \frac {C\mu_8}{n^{\frac32}v^{\frac32}}.
\end{equation}
Thus Lemma \ref{trunc1} is proved.
\end{proof}
\begin{cor}
 Assuming the conditions of Corollary \ref{cormain}, we have for $z\in \mathbb G$,
 \begin{equation}\notag
  |\E\widehat m_n(z)-s_y(z)|\le \frac{C}{(nv)^{\frac32}}+\frac C{n^2v^2\sqrt{\gamma}}.
 \end{equation}

\end{cor}
\begin{proof}
 The proof immediately follows from the inequality
 \begin{equation}\notag
  |\E\widehat m_n(z)-s_y(z)|\le |\E (\widehat m_n(z)-\breve m_n(z))|+|\E\breve m_n(z)-s(z)|,
 \end{equation}
Lemmas \ref{trunc2*} and \ref{trunc1} and Theorem \ref{stiltjesmain}.
\end{proof}

The proof of Corollary \ref{cormain} follows now from Lemma \ref{trunc}, Corollary \ref{smoothing2}, inequality \eqref{final!} and inequality
\begin{equation}\notag
 \sup_{x\in\mathbb J_{\varepsilon}}\int_{v_0/\sqrt{\gamma}}^V|\E\widehat m_n(x+iv)-s(x+iv)|dv\le \frac Cn.
\end{equation}

\section{The Crucial Results} The main problem in proving Theorem \ref{stiltjesmain} is the the derivation of the following bound 
\begin{equation}\notag
 \E|R_{jj}|^q\le C^q,
\end{equation}
for $j=1,\ldots,n$ and any $z\in\mathbb G$.
This bound in the case of Wigner matrices was shown in \cite{GT:2014}.
To prove this bound we used an approach similar to that of Lemma 3.4 in \cite{SchleinMaltseva:2013}. 
We succeeded in the case of finite moments only developing  new  bounds of quadratic forms of the following type
\begin{equation}\notag
 \E|\frac1n\sum_{1\le l\ne k\le p}X_{jl}X_{jk}R^{(j)}_{k+n,l+n}|^q\le\left(\frac {Cq}{\sqrt {nv}}\right)^q. 
\end{equation}
These estimates are based on a recursive scheme of using  Rosenthal's and Burkholder's inequalities.
\subsection{The Key Lemma}\label{key} In this Section we provide  auxiliary lemmas needed for the proof of Theorem \ref{main}.
Recall that the
Stieltjes transform of an empirical  spectral distribution function $\mathcal F_n(x)$, say $m_n(z)$, is given by
\begin{equation}\label{trace}
m_n(z)=\frac1n\sum_{j=1}^n
R_{jj}=\frac1{2n}\Tr\mathbf  R+\frac{p-n}{2nz}.
\end{equation}

For any $\mathbb J\subset \mathbb T=\{1,\ldots,n\}$ denote $\mathbb T_{\mathbb J}=\mathbb T\setminus\mathbb J$.
For any $\mathbb J\subset \mathbb T$ and $j\in\mathbb T_{\mathbb J}$ define  the quadratic form,
$$
Q^{(\mathbb J,j)}:=\frac1p\sum_{l=1}^{p-1}\Big|\sum_{k=1}^{l-1}X_{jk}R^{(\mathbb J,j)}_{k+n,l+n}\Big|^2.
$$
Similar, for $\widehat {\mathbb J}\subset\widehat {\mathbb T}=\{1,\ldots,p\}$ and $j\in \widehat{\mathbb T}_{\widehat{\mathbb J}}$,
$$
Q^{(\widehat{\mathbb J},j)}:=\frac1p\sum_{l=1}^{n-1}\Big|\sum_{k=1}^{l-1}X_{kj}R^{(\widehat{\mathbb J}+n,j+n)}_{kl}\Big|^2.
$$
%
\begin{thm}\label{bp1*} Assuming the conditions of Theorem \ref{main} there exist  constants $A_1, C, C_3$ depending on $\mu_4$ and $D$ only such that 
we have for $v\ge v_0$  and  $q\le A_1(nv)^{\frac14}$ and for any $\mathbb J\subset\mathbb T$ such that $|\mathbb J|\le C\log n$,

\begin{align}\label{q01}
 \E (Q^{(\mathbb J,j)})^q\le(C_3q)^{2q}v^{-q}.
\end{align}
Respectively, for any $\widehat{\mathbb J}\subset\widehat{\mathbb T}$,
\begin{align}\label{q2}
 \E (Q^{(\widehat{\mathbb J},j)})^q\le(C_3q)^{2q}v^{-q}.
\end{align}
 
\end{thm}
\begin{cor}\label{q1*}Assuming the conditions of Theorem \ref{main} and for $z=u+iV$ with $V=4\sqrt y$, we have
 \begin{align}\notag
 \E (Q^{(\mathbb J,j)})^q\le C^qq^{2q},
\end{align}
and 
\begin{align}\notag
 \E (Q^{(\widehat{\mathbb J},j)})^q\le C^qq^{2q}.
\end{align}

\end{cor}
\begin{proof}
 The result immediately follows from Theorem \ref{bp1*}
\end{proof}

\begin{Proof of} {\it Theorem \ref{bp1*}}. We proof  inequality \eqref{q01} only. The inequality \eqref{q2} is proved similar.
For the proof of Theorem \ref{bp1*} we prove several auxiliary Lemmas.
 %

For any $\mathbb J\subset \mathbb T$ introduce $\mathbb T_{\mathbb J}=\mathbb T\setminus\mathbb J$.
We introduce the quantity, for some $\mathbb J\subset \mathbb T$,
$$
B_q^{(\mathbb J)}:=\Big[\frac1p\sum_{l=2}^{p}\Big(\sum_{k=1}^{l-1}|R_{k+n,l+n}^{(\mathbb J)}|^2\Big)^q\Big].
$$
By Lemma \ref{resol00},  inequality \eqref{res2} in the Appendix, we have
\begin{equation}\label{ineq00n}
\E B_q^{(\mathbb J)}\le v^{-q}\frac1n\sum_{l=1}^p\E|R_{l+n,l+n}^{(\mathbb J)}|^q.
\end{equation}
Furthermore, introduce the quantities
\begin{align}\label{qf}
Q^{(\mathbb J,j)}_{\nu}&=\sum_{l=2}^{p}\Big|\sum_{k=1}^{l-1}X_{jk}a_{lk}^{(\mathbb J,j,\nu)}\Big|^2,\notag\\
Q^{(\mathbb J,j)}_{\nu1}&=\sum_{l=2}^pa_{ll}^{(\mathbb J,j,\nu+1)},\notag\\
Q^{(\mathbb J,j)}_{\nu2}&=\sum_{l=2}^p(X_{jl}^2-1)a_{ll}^{(\mathbb J,j,\nu+1)},\notag\\
Q^{(\mathbb J,j)}_{\nu3}&=\sum_{ 1\le l\ne k\le p}X_{jk}X_{jl}a_{kl}^{(\mathbb J,j,\nu+1)},
\end{align}
where, $a^{(\mathbb J,j,\nu)}_{kl}$ are defined recursively via 
\begin{align}\label{defaq}
a_{kl}^{(\mathbb J,j,0)}&=\frac1{\sqrt p} R^{(\mathbb J,j)}_{k+n,l+n},\notag\\
a_{kl}^{(\mathbb J,j,\nu+1)}&=\sum_{ r=k\vee l+1}^pa_{rl}^{(\mathbb J,j,\nu)}\overline a_{rk}^{(\mathbb J,j,\nu)},\text{ for }\nu=0,\ldots,L-1,
\end{align}
where $k\vee l:=\max\{k,l\}$.
Using these notations we have
\begin{equation}\label{r1}
 Q^{(\mathbb J,k)}_{\nu}=Q^{(\mathbb J,k)}_{\nu1}+Q^{(\mathbb J,k)}_{\nu2}+Q^{(\mathbb J,k)}_{\nu3}.
\end{equation}
\begin{lem}\label{arr0}
 Under the conditions of Theorem \ref{main} we have
 \begin{align}\label{ar1}
\sum_{l=1}^p |a_{kl}^{(\mathbb J,j,\nu+1)}|^2\le \Big(\sum_{r=1}^p|a_{kr}^{(\mathbb J,j,\nu)}|^2\Big)
\Big(\sum_{l,r=1}^p|a_{lr}^{(\mathbb J,j,\nu)}|^2\Big).
\end{align}
Moreover,
 \begin{align}\label{ar2}
  \sum_{l,k=1}^p|a_{kl}^{(\mathbb J,j,\nu+1)}|^2\le (\sum_{k,l=1}^p|a_{kl}^{(\mathbb J,j,\nu)}|^2)^2.
 \end{align}

\end{lem}
\begin{proof}
 We apply Cauchy -- Schwartz inequality and obtain
 \begin{equation}\notag
  |a_{l,k}^{(\mathbb J,j,\nu+1)}|^2\le \sum_{r=k\vee\ l+1}^p|a_{rl}^{(\mathbb J,j,\nu)}|^2\sum_{r=k\vee l+1}^p|a_{kr}^{(\mathbb J,j,\nu)}|^2.
 \end{equation}
Summing in $k$ and $l$, \eqref{ar1} and \eqref{ar2} follow.

\end{proof}
\begin{cor}\label{arr1} Under the conditions of Theorem \ref{main} we have
\begin{align}
\sum_{k,l=1}^p|a_{kl}^{(\mathbb J,j,\nu)}|^2\le \left(\im\left( m_n^{(\mathbb J,j)}(z)- \frac{p-n+|\mathbb J|}{nz}\right)v^{-1}\right)^{2^{\nu}}
\end{align}
and 
 \begin{align}
\sum_{l=1}^p |a_{kl}^{(\mathbb J,j,\nu)}|^2\le \left(\im\left( m_n^{(\mathbb J,j)}(z)-
\frac{p-n+|\mathbb J|}{nz}\right)v^{-1}\right)^{2^{\nu}-1}n^{-1}v^{-1}\im R_{k+n,k+n}^{(\mathbb J,j)}.\notag
\end{align} 
\end{cor}
\begin{proof}
 By definition of $a_{kl}^{(\mathbb J,j,0)}$, see  \eqref{defaq}, applying Lemma \ref{resol00}, equality \eqref{res1}, in the Appendix, we get
 \begin{equation}\notag
  \sum_{l=1}^p|a_{kl}^{(\mathbb J,j,0)}|^2\le \frac1p\sum_{l=1}^p|R_{k+n,l+n}^{(\mathbb J,j)}|^2
  \le n^{-1}v^{-1}\left(\im R_{k+n,k+n}^{(\mathbb J,j)}\right),
 \end{equation}

The general case follows now by induction in $\nu$, Lemma \ref{arr0}, and inequality \eqref{res1}, Lemma \ref{resol00} in the Appendix.
\end{proof}
\begin{cor}\label{arr2}
 Under the conditions of Theorem
  \ref{main} we have
 \begin{align}\notag
  a_{ll}^{(\mathbb J,j,\nu+1)}\le (nv)^{-1}\left(v^{-1}\im \left(m_n^{(\mathbb J,j)}(z)-\frac{p-n+|\mathbb J|}{nz}\right)\right)^{2^{\nu}-1}\im R_{l+n,l+n}^{(\mathbb J,j)}.
 \end{align}

\end{cor}
\begin{proof}
 The result immediately follows from the definition of $a_{ll}^{(\mathbb J,j,\nu)}$ and Corollary \ref{arr1}.
\end{proof}
\begin{cor}\label{new1}
 Under the conditions of Theorem
  \ref{main} we have
 \begin{align}\notag
  \frac1{n^2}\sum_{l,k=1}^p|a_{lk}^{(\mathbb J,\nu+1)}|^q\le (nv)^{-q}\left(v^{-1}\im \left(m_n^{(\mathbb J,j)}(z)-\frac{p-n+|\mathbb J|}{nz}\right)\right)^{(2^{\nu}-1)q}
 (\frac1n\sum_{k=1}^p|R_{k+n,k+n}^{(\mathbb J,j)}|^{\frac q2})^2
 \end{align}
and
\begin{align}\notag
  \frac1{n}\sum_{l=1}^p(\sum_{k=1}^p|a_{lk}^{(\mathbb J,\nu+1)}|^2)^{\frac q2}\le (nv)^{-\frac q2}\left(v^{-1}\im \left(m_n^{(\mathbb J,j)}(z)-\frac{p-n+|\mathbb J|}{nz}\right)\right)^{(2^{\nu}-1)\frac q2}
 \frac1n\sum_{k=1}^p|R_{k+n,k+n}^{(\mathbb J,j)}|^{\frac q2}
 \end{align}
\end{cor}
\begin{proof}
 By definition of $a_{lk}^{(\mathbb J,\nu+1)}$ and Cauchy's inequality, we have
 \begin{equation}
  |a_{lk}^{(\mathbb J,\nu+1)}|^q\le (\sum_{r=1}^p|a_{lr}^{(\mathbb J,\nu)}|^2)^{\frac q2}(\sum_{r=1}^p|a_{rk}^{(\mathbb J,\nu)}|^2)^{\frac q2}.
 \end{equation}
Using Corollary \ref{arr1} and summing in $l,k$, we get
\begin{align}
 \frac1{n^2}\sum_{l,k=1}^p|a_{lk}^{(\mathbb J,\nu+1)}|^q\le n^{-q}v^{-(2^{\nu}-1)q}\im^{(2^{\nu}-1)q}\left( m_n^{(\mathbb J)}(z)-
\frac{p-n+|\mathbb J|}{nz}\right)(\frac1n\sum_{k=1}^p|R_{k+n,k+n}^{(\mathbb J,j)}|^{\frac q2})^2.
\end{align}

\end{proof}

In what follows we shall use the notations
\begin{align}\label{not1}
 \Psi^{(\mathbb J)}&:=\im \left(m_n^{(\mathbb J)}(z)-\frac{p-n+|\mathbb J|}{nz}\right)+\frac1{nv},
 \quad (A_{\nu,q}^{(\mathbb J)})^2=\E(\Psi^{(\mathbb J)})^{(2^{\nu}-1)2 q}, \notag\\
 T_{\nu,q}^{(\mathbb J,j)}&:=\E|Q^{(\mathbb J,j)}_{\nu}|^q,\quad 
A_q^{(\mathbb J)} :=1+\E^{\frac14}|\Psi^{(\mathbb J)}|^{4q}.
\end{align}
Note that
\begin{equation}
 \im \left(m_n^{(\mathbb J,j)}(z)-\frac{p-n+|\mathbb J|}{nz}\right)\le\Psi^{(\mathbb J)}
\end{equation}

Let $s_0$ denote some fixed number (for instance $s_0=2^4$). Let $A_1$ be a constant (to be chosen later) and $0<v_1\le 4$
a constant such that $v_0=A_0n^{-1}\le v_1$ for all $n\ge 1$.
\begin{lem}\label{bp1} Assuming the conditions of Theorem \ref{main} and for $q\le A_1(nv)^{\frac14}$
\begin{equation}\label{cond03}
\E|R_{l+n,l+n}^{(\mathbb J)}|^q\le C_0^q,\text{ for }v\ge v_1,\text{ for all }l=1,\ldots,p,
\end{equation} 
we have for $v\ge v_1/s_0$  and  $q\le A_1(nv)^{\frac14}$, and $j\in\mathbb T_{\mathbb J}$

\begin{align}\label{q1}
 \E (Q^{(\mathbb J,j)}_0)^q\le6(\frac {C_3q}{\sqrt2})^{2q}v^{-q}A_q^{(\mathbb J)},
\end{align}
where $C_3$ is a  constant
depending on $C_0$. 
\end{lem}

\begin{proof}

Using the representation \eqref{r1} and the triangle inequality, we get
\begin{equation}\label{po0}
\E| Q^{(\mathbb J,j)}_{\nu}|^q\le 3^q\Big(\E|Q^{(\mathbb J,j)}_{\nu1}|^q+\E|Q^{(\mathbb J,j)}_{\nu2}|^q+\E|Q^{(\mathbb J,j)}_{\nu3}|^q\Big).
\end{equation}
Let $\mathfrak M^{(\mathbb A)}$ denote the $\sigma$-algebra generated by r.v.'s 
$X_{j,l}$ for $j\in\mathbb T_{\mathbb A}$, $l=1,\ldots,p$, for any set $\mathbb A$.
Conditioning on $\mathfrak M^{(\mathbb J,j)}$ ($\mathbb A=\mathbb J\cup \{j\}$) and applying Rosenthal's inequality (see Lemma \ref{Rosent}), we get
\begin{align}
 \E|Q^{(\mathbb J,j)}_{\nu2}|^q\le C_1^qq^q\bigg(\E\Big(\sum_{l=1}^p|a_{ll}^{(\mathbb J,j,\nu+1)}|^2\Big)^{\frac q2}
 +\sum_{l=1}^p\E|a_{ll}^{(\mathbb J,j,\nu+1)}|^{q}\E|X_{jl}|^{2q}\bigg),\notag
\end{align}
where $C_1$ denotes the  absolute constant in Rosenthal's inequality.
By Remark \ref{trunc00}, we get
\begin{align}\label{r01}
 \E|Q^{(\mathbb J,j)}_{\nu2}|^q&\le C_1^qq^q\Big(\E(\sum_{l=1}^p|a_{ll}^{(\mathbb J,j,\nu+1)}|^2)^{\frac q2}+p^{\frac q2}\frac1{p}
 \sum_{l=1}^p\E|a_{ll}^{(\mathbb J,j,\nu+1)}|^{q}\Big).
\end{align}
Analogously conditioning on $\mathfrak M^{(\mathbb J,j)}$ and applying Burkholder's inequality (see Lemma \ref{burkh1}), we get
\begin{align}\label{r3}
\E| Q^{(\mathbb J,j)}_{\nu3}|^q&\le C_2^qq^q\bigg(\E\Big(\sum_{l=2}^p|\sum_{k=1}^{l-1}X_{jk}
a^{(\mathbb J,j,\nu+1)}_{kl}|^2\Big)^{\frac q2} 
\notag\\&\qquad\qquad\qquad\qquad+\sum_{l=1}^{p-1}\E\big|\sum_{k=1}^{l-1}X_{jk}a^{(\mathbb J,j,\nu+1)}_{lk}\big|^q\E|X_{jl}|^q\bigg),
\end{align}
where $C_2$ denotes the absolute constant in Burkholder's inequality.
Conditioning again on $\mathfrak M^{(\mathbb J,j)}$  and applying Rosenthal's inequality, we obtain
\begin{align}\label{r4}
 \E|\sum_{l=1}^pX_{jl}a^{(\mathbb J,j,\nu+1)}_{kl}|^q&\le C_1^qq^q\Big(\E(\sum_{l=1}^{p}|a^{(\mathbb J,j,\nu+1)}_{lk}|^2)^{\frac q2}
 \notag\\&\qquad\qquad+\sum_{l=1}^p\E|a^{(\mathbb J,j,\nu+1)}_{lk}|^q\E|X_{jl}|^q\Big).
\end{align}
Combining inequalities \eqref{r3} and \eqref{r4}, we get
\begin{align}
 \E| Q^{(\mathbb J,j)}_{\nu3}|^q&\le C_2^qq^q\E|Q^{(\mathbb J,j)}_{\nu+1}|^{\frac q2}+C_1^qC_2^qq^{2q}\sum_{l=1}^p
 \E(\sum_{k=1}^p|a^{(\mathbb J,j,\nu+1)}_{kl}|^2)^{\frac q2}\E|X_{jk}|^q\notag\\&\qquad\qquad\qquad+
C_1^qC_2^q q^{2q}\sum_{l=1}^p\sum_{k=1}^p\E|a^{(\mathbb J,j,\nu+1)}_{kl}|^q\E|X_{jk}|^q\E|X_{jl}|^q.\notag
\end{align}
Using Remark \ref{trunc00}, this implies
\begin{align}\label{r5}
 \E| Q^{(\mathbb J,j)}_{\nu3}|^q&\le C_2^qq^q\E|Q^{(\mathbb J,j)}_{\nu+1}|^{\frac q2}+C_1^qC_2^qq^{2q}n^{\frac q4}
 \frac1p\sum_{l=1}^p\E(\sum_{k=1}^p|a^{(\mathbb J,j,\nu+1)}_{lk}|^2)^{\frac q2}\notag\\&\qquad\qquad\qquad+
 C_1^qC_2^qq^{2q}n^{\frac q2}\frac1{n^2}\sum_{l=1}^p\sum_{k=1}^p\E|a^{(\mathbb J,j,\nu+1)}_{lk}|^q.
\end{align}
Using the definition \eqref{qf} of $Q^{(\mathbb J,j)}_{\nu1}$ and the definition \eqref{defaq} of coefficients $ a^{(\mathbb J,j,\nu+1)}_{ll}$, 
it is straightforward to check that
\begin{equation}\label{r6}
 \E|Q^{(\mathbb J,l)}_{\nu1}|^q\le \E\Big[\sum_{k,l=1}^p|a_{kl}^{(\mathbb J,j,\nu)}|^2\Big]^q.
\end{equation}
Combining \eqref{r01}, \eqref{r5} and \eqref{r6}, we get by \eqref{po0}
\begin{align}
 \E|Q^{(\mathbb J,j)}_{\nu}|^q&\le C_2^qq^q\E|Q^{(\mathbb J,q)}_{\nu+1}|^{\frac q2}
 +C_1^qC_2^qq^{2q}n^{\frac q4}\frac1n\sum_{l=1}^p\E\Big(\sum_{k=1}^p|a^{(\mathbb J,j,\nu+1)}_{kl}|^2\Big)^{\frac q2}\notag\\&+
 C_1^qC_2^qq^{2q}n^{\frac q2}\frac1{n^2}\sum_{l,k=1}^p\E|a^{(\mathbb J,j,\nu+1)}_{kl}|^q\notag\\&
 +C_1^qC_2^q\E[\sum_{q,r\in\mathbb T_{\mathbb J,k}}|a_{qr}^{(\mathbb J,j,\nu)}|^2]^q\notag\\&+
 C_1^pC_2^qq^q\Big(\E\Big(\sum_{l=1}^p|a_{ll}^{(\mathbb J,j,\nu+1)}|^2\Big)^{\frac q2}
 +n^{\frac q2}\frac1{n}\sum_{l=1}^p\E|a_{ll}^{(\mathbb J,j,\nu+1)}|^{q}\Big).\notag
\end{align}

Applying now Lemma \ref{arr0} and Corollary \ref{new1}, we obtain 
\begin{align}
 \E|Q^{(\mathbb J,j)}_{\nu}|^q&\le C_3^qq^q\E|Q^{(\mathbb J,k)}_{\nu+1}|^{\frac q2}+
C_3^q\E(\Psi^{(\mathbb J)})^{(2^{\nu}-1)q}\;v^{-(2^{\nu}-1) q}\notag\\&+C_3^qq^{2q}\;v^{-2^{\nu-1}q}n^{-\frac q4}\E(\Psi^{(\mathbb J)})^{(2^{\nu}-1)\frac q2} 
 \Big(\frac1n\sum_{l=1}^p|R_{l+n,l+n}^{(\mathbb J,j)}|^{\frac q2}\Big)\notag\\
&+C_3^qq^{2q}n^{-\frac q2}\;v^{-2^{\nu}q}\E(\Psi^{(\mathbb J)})^{(2^{\nu}-1)q} 
 \Big(\frac1n\sum_{l=1}^p|R_{ll}^{(\mathbb J,j)}|^{\frac q2}\Big)^2,
\end{align}
where $C_3=3C_1C_2$.
Applying Cauchy--Schwartz and Jensen inequalities, we may rewrite the last inequality in the form
\begin{align}\label{last1}
 \E|Q&^{(\mathbb J,j)}_{\nu}|^q\le C_3^qq^q\E|Q^{(\mathbb J,j)}_{\nu+1}|^{\frac q2}+
 C_3^q\E(\Psi^{\mathbb J)})^{(2^{\nu}-1) q}\,v^{-(2^{\nu}-1) q}\notag\\
 &+C_3^qq^{2q}\;v^{-2^{\nu-1}q}n^{-\frac q4}\E^{\frac12}(\Psi^{(\mathbb J)})^{(2^{\nu}-1)q}\E^{\frac12}(\frac1n\sum_{l=1}^p|R_{l+n,l+n}^{(\mathbb J,j)}|^{q})\notag\\
 &+C_3^qq^{2q}n^{-\frac q2}\;v^{-2^{\nu}q}\E^{\frac12}(\Psi^{\mathbb J)})^{(2^{\nu}-1)2q}
 \E^{\frac12}(\frac1n\sum_{l=1}^p|R_{l+n,l+n}^{(\mathbb J,j)}|^{2q}).
\end{align}

 Introduce the notation
 \begin{align}
 \Gamma_q(z)&:=\E^{\frac12}\left(\frac1n\sum_{l=1}^p|R_{l+n,l+n}^{(\mathbb J,k)}|^{2q}\right).\notag
 \end{align}

We rewrite the inequality \eqref{last1} using $\Gamma_q(z)$ and the  notations of \eqref{not1} as follows
\begin{align}\label{rewrite}
 T_{\nu,q}^{(\mathbb J,j)}&\le (C_3q)^qT_{\nu+1,q/2}^{(\mathbb J,j)}+C_3^qA^{(\mathbb J)}_{\nu,q}v^{-(2^{\nu}-1)q}\notag\\&+
 (C_3q^2)^{q}\Big(v^{-2^{\nu}q}n^{-\frac q4}(A^{(\mathbb J)}_{\nu,\frac q2})^{\frac12}\Gamma_q^{\frac12}(z)+v^{-2^{\nu}q}n^{-\frac q2}A^{(\mathbb J)}_{\nu,q}
 \Gamma_q(z)\Big).
\end{align}

Note that
 \begin{equation}\notag
 A_{0,q}^{(\mathbb J)}=1,\quad
  A_{\nu,q/2^{\nu}}^{(\mathbb J)}\le \sqrt{1+\E(\Psi^{(\mathbb J)})^{2q}}\le 1+\E^{\frac14}(\Psi^{(\mathbb J)})^{4q},
 \end{equation}
 where $\Psi^{(\mathbb J)}=\im (m_n^{(\mathbb J)}(z)-\frac{p-n+|\mathbb J|}{nz})+\frac{1}{nv}$.
Furthermore,
\begin{equation}\notag
 \Gamma_{2q/2^{\nu}}\le \Gamma_{2q}^{\frac1{2^{\nu}}}.
\end{equation}

Without loss of generality we may assume $q=2^L$ and $\nu=0,\ldots,L$.
We may write
\begin{align}
 T_{0,q}^{(\mathbb J,j)}&\le (C_3q)^qT_{1,q/2}^{(\mathbb J,k)}+C_3^q+
 (C_3q^2)^{q}v^{- q}\Big(n^{-\frac q4}\Gamma_q^{\frac12}(z)+
 n^{-\frac q2}\Gamma_q(z)\Big).\notag
\end{align}
By induction we get
\begin{align}\label{t01}
 T_{0,q}^{(\mathbb J,j)}&\le \prod_{\nu=0}^L(C_3q/2^{\nu})^{q/2^{\nu}}T_{L,1}^{(\mathbb J,j)}+A_{q}^{(\mathbb J)}\sum_{l=1}^L
 \prod_{\nu=0}^{l-1}(C_3q/2^{\nu})^{q/2^{\nu}}v^{-(2^{l}-1)q/2^{l}}\notag\\&\qquad\qquad+
 A_{q}^{(\mathbb J)}v^{-q}\sum_{l=1}^L\Big(\prod_{\nu=0}^{l-1}(C_3q/2^{\nu})^{q/2^{\nu}}\Big)(n^{-q}\Gamma_q^2)^{\frac1{2^{l+1}}}
 \notag\\&\qquad\qquad + A_{q}^{(\mathbb J)}\sum_{l=1}^L\Big(\prod_{\nu=0}^{l-1}(C_3q/2^{\nu})^{q/2^{\nu}}\Big)(n^{-q}{\Gamma_q^2})^{\frac1{2^{l}}}.
\end{align}
It is straightforward to check that
\begin{align}\notag
 \sum_{\nu=1}^{l-1}\frac{\nu}{2^{\nu}}=2(1-\frac{2l+1}{2^{l}}).
\end{align}

Note that, for $l\ge 1$,
\begin{equation}\label{rel00}
\prod_{\nu=0}^{l-1}\big(C_3(q/2^{\nu})\big)^{q/2^{\nu}}=
\frac{(C_3q)^{2q(1-2^{-l})}}{2^{2q(1-\frac{2l+1}{2^{l}})}}=2^{4q\frac l{2^{l}}}
\Big(\frac{C_3q}{2}\Big)^{2q(1-2^{-l})}.
\end{equation}
Applying this relation, we get
\begin{align}
A_{q}^{(\mathbb J)}\sum_{l=0}^L\Big(\prod_{\nu=0}^{l-1}(C_3q/2^{\nu})^{q/2^{\nu}}\Big)v^{-(2^{l}-1)q/2^{l}}&\le A_q^{(\mathbb J)}(\frac{C_3q}2)^{2q}v^{-q}
\sum_{l=0}^{L-1}2^{\frac{4ql}{2^l}}\Big(\frac{4v}{C_3^2q^2}\Big)^{\frac q{2^l}}.\notag
\end{align}
Note that for $l\ge 0$, $\frac l{2^l}\le \frac12$ and recall that $q=2^L$. Using this observation, we get
\begin{align}
A_{q}^{(\mathbb J)}\sum_{l=0}^L\Big(\prod_{\nu=0}^{l-1}(C_3q/2^{\nu})^{q/2^{\nu}}\Big)v^{-(2^{l}-1)q/2^{l}}&\le A_q^{(\mathbb J)}(\frac{C_3q}2)^{2q}v^{-q}2^{2q}
\sum_{l=0}^{L-1}\Big(\frac{4v}{C_3^2q^2}\Big)^{2^{L-l}}.\notag
\end{align}
This implies that for 
$\frac {4v}{C_3^2q^2}\le \frac12$,
\begin{align}
A_{q}^{(\mathbb J)}\sum_{l=1}^L\Big(\prod_{\nu=0}^{l-1}(C_3q/2^{\nu})^{q/2^{\nu}}\Big)v^{-(2^{l}-1)q/2^{l}}\le
(C_3q)^{2q}A_q^{(\mathbb J)}v^{-q}.\notag
\end{align}
Furthermore, by definition of $T_{\nu,q}$, we have
\begin{equation}\notag
 T_{L,1}^{(\mathbb J,j)}=\E Q^{(\mathbb J,j)}_L\le \E\sum_{l,k=1}^p(a_{kl}^{(\mathbb J,j,L)})^2.
\end{equation}
Applying Corollary \ref{arr1} and H\"older's inequality, we get
\begin{align}\label{t02-}
 T_{L,1}^{(\mathbb J,j)}\le E(v^{-1}\Psi^{(\mathbb J)})^q\le v^{-q}A_q^{(\mathbb J)}.
\end{align}
By condition \eqref{cond03}, we have
\begin{equation}\notag
\Gamma_q:=\Gamma_q(u+iv)\le s_0^{2q}C_0^{2q}.
\end{equation}
Using this inequality, we get,
\begin{align}\label{t02+}
 A_{q}^{(\mathbb J)}v^{-q}\sum_{l=1}^L\Big(\prod_{\nu=0}^{l-1}&(C_3q/2^{\nu})^{q/2^{\nu}}\Big)n^{-\frac q{2^{l+1}}}\Gamma_q^{\frac1{2^{l}}}
 \notag\\
 &\le  A_{q}^{(\mathbb J)}v^{-q}\sum_{l=0}^L\Big(\prod_{\nu=0}^{l-1}(C_3q/2^{\nu})^{q/2^{\nu}}\Big)
 (s_0^4C_0^4n^{-1})^{\frac q{2^{l+1}}}.
\end{align}
Applying relation \eqref{rel00}, we obtain
\begin{align}
 A_{q}^{(\mathbb J)}v^{-q}\sum_{l=1}^L\Big(\prod_{\nu=0}^{l-1}&(C_3q/2^{\nu})^{q/2^{\nu}}\Big)n^{-\frac q{2^{l+2}}}\Gamma_q^{\frac1{2^{l+2}}}
 \notag\\
 &\le \Big(\frac{C_3q}{2}\Big)^{2q}A_{q}^{(\mathbb J)}v^{-q}\sum_{l=1}^L2^{2q\frac l{2^l}}
\Big(\frac{C_3q}{2}\Big)^{-\frac{2q}{2^{l}}}
 (s_0^2C_0^2n^{-1})^{\frac q{2^{l+2}}}\notag\\&=\Big(\frac{C_3q}{2}\Big)^{2q}A_{q}^{(\mathbb J)}v^{-q}\sum_{l=1}^L2^{q\frac l{2^{l-1}}}
\Big(\frac{C_3q}{2}\Big)^{-\frac{2q}{2^{l-1}}}
 ((s_0^4C_0^4n^{-1})^{\frac14})^{\frac q{2^{l-1}}}\notag\\&=\Big(\frac{C_3q}{\sqrt 2}\Big)^{2q}A_{q}^{(\mathbb J)}v^{-q}\sum_{l=1}^L
\Big(
 \frac{s_0C_0}{C_3qn^{\frac14}}\Big)^{{2^{L-l+1}}}.\notag
\end{align}
Without loss of generality we may assume that $C_3\ge 2(C_0s_0)$. Then we get 
\begin{align}
 A_{q}^{(\mathbb J)}v^{-q}\sum_{l=0}^L\Big(\prod_{\nu=0}^{l-1}&(C_3q/2^{\nu})^{q/2^{\nu}}\Big)n^{-\frac q{2^{l+1}}}\Gamma_q^{\frac1{2^{l}}}\le 
 (C_3q)^{2q}A_{q}^{(\mathbb J)}v^{-q}.\notag
\end{align}

Analogously we get
\begin{align}\label{t02*}
A_{q}^{(\mathbb J)}v^{-q}\sum_{l=1}^L\Big(\prod_{\nu=0}^{l-1}(C_3p/2^{\nu})^{q/2^{\nu}}\Big)n^{-\frac q{2^{l}}}\Gamma_q^{\frac1{2^{l-1}}}\le 
({C_3q})^{2q}v^{-q}A_q^{(\mathbb J)}.
\end{align}
Combining  inequalities \eqref{t01}, \eqref{t02-}, \eqref{t02+}, \eqref{t02*}, we finally arrive at
\begin{equation}\label{t0}
T_{0,q}^{(\mathbb J,j)}\le 6(C_3q)^{2q}v^{-q}A_q^{(\mathbb J)}.
\end{equation}

 Thus, Lemma \ref{bp1} is proved.
\end{proof}
\end{Proof of}
\subsection{Diagonal Entries of the Resolvent Matrix}\label{diag}

Recall that
\begin{equation}\label{repr001}
 R_{jj}=-\frac1{z+ym_n(z)+\frac{y-1}z}+\frac1{z+ym_n(z)\frac{y-1}z}\varepsilon_jR_{jj},
\end{equation}
or
\be\label{repr01*}
 R_{jj}=-\frac 1{z+ys_y(z)+\frac{y-1}z}+\frac{y\Lambda_nR_{jj}}{(z+ys_y(z)+\frac{y-1}z)}+
\frac1{z+ys_y(z)+\frac{y-1}z}\varepsilon_jR_{jj}, \en
where
$\varepsilon_j:=\varepsilon_{j1}+\varepsilon_{j2}+\varepsilon_{j3}$  with
\begin{align}
\varepsilon_{j1}&:=-\frac1p{\sum_{k\ne l
\in\mathbb T_j}}X_{jk}X_{jl}R^{(j)}_{kl},\quad
\varepsilon_{j2}:=-\frac1p{\sum_{k\in\mathbb T_j}}(X_{jk}^2-1)R^{(j)}_{kk},\notag\\
\varepsilon_{j3}&:=\frac1p(\sum_{l=1}^p R^{(j)}_{l+n,l+n}- \sum_{l=1}^pR_{l+n,l+n}),\notag\\
\Lambda_n&:=m_n(z)-s_y(z)=\frac1{2n}\Tr\mathbf R-s_y(z)+\frac{1-y}{2yz}.\label{epsjn}
\end{align}
Using equation \eqref{stmarpas}, we may rewrite representation \eqref{repr01*} as follows
\begin{equation}\label{repr01**}
  R_{jj}=s_y(z)+y\Lambda_nR_{jj}s_y(z)+
s_y(z)\varepsilon_jR_{jj}.
\end{equation}
Similar we have, for $j=1,\ldots,p$
\begin{align}\label{res22}
 R_{j+n,j+n}=-\frac1{z+ym_n(z)}+\frac1{z+ym_n(z)}\widehat\varepsilon_jR_{jj},
\end{align}
where
\begin{align}
\widehat\varepsilon_{j1}&:=-\frac1p{\sum_{k=1}^n}(X_{kj}^2-1)R^{(j+n)}_{kk},\quad
\widehat\varepsilon_{j2}:=-\frac1p{\sum_{1\le k\ne l\le n
}}X_{kj}X_{lj}R^{(j+n)}_{kl},\notag\\
\widehat\varepsilon_{j3}&:=\frac1p(\sum_{l=1}^n R^{(j+n)}_{l,l}- \sum_{l=1}^nR_{l,l})
.\label{epsjn10}
\end{align}
We may rewrite \eqref{res22} as
\begin{align}
 R_{j+n,j+n}&=-\frac1{z+ys_y(z)}+\frac{y\Lambda_nR_{j+n,j+n}}{z+ys_y(z)}+\frac1{z+ys_y(z)}
 \widehat{\varepsilon}_jR_{j+n,j+n}.\notag
\end{align}
where
\begin{align}
\Lambda_n&:=m_n(z)-s_y(z)=\frac1{2n}\Tr\mathbf R-s_y(z)+\frac{1-y}{2yz}.\label{epsjn1}
\end{align}
We shall consider the representation \eqref{repr01**} only, because \tc{arguments for the representation \eqref{res22} is similar. In the latter case
\eqref{res22} we need to use Lemma
\ref{resol1*} instead of the bound} $|s_y(z)|\le \frac1{\sqrt y}$.
Since $|s_y(z)|\le1/\sqrt y$, the representation \eqref{repr01**} yields, for any $q\ge1$,
\begin{equation}\label{in1}
|R_{jj}^{(\mathbb J)}|^q\le 3^qy^{-\frac q2}+3^qy^{-\frac q2}|\varepsilon_j^{(\mathbb J)}|^q|R_{jj}^{(\mathbb J)}|^q
+3^py^{-\frac q2}|\Lambda_n^{(\mathbb J)}|^q|R_{jj}^{(\mathbb J)}|^q.
\end{equation}
We shall use the equality
\begin{equation}\label{shur}
 \Tr\mathbf R-\Tr\mathbf R^{(j)}=(1+\frac1p\sum_{k,l=1}^pX_{jk}X_{jl}[(\mathbf R^{(j)})^2]_{k+n,l+n})R_{jj}.
 \end{equation}
See for instance  \cite{GT:2009}, Lemma  3.1.

Applying the Cauchy -- Schwartz inequality, we get
\begin{equation}\label{ineq10}
3^{-q}y^{\frac q2}\E|R_{jj}^{(\mathbb J)}|^q\le 1+\E^{\frac12}|\varepsilon_j^{(\mathbb J)}|^{2q}\E^{\frac12}|R_{jj}^{(\mathbb J)}|^{2q}+\E^{\frac12}|\Lambda_n^{(\mathbb J)}|^{2q}
\E^{\frac12}|R_{jj}^{(\mathbb J)}|^{2q}.
\end{equation}

We shall investigate now the behavior of $\E|\varepsilon_j^{(\mathbb J)}|^{2q}$ and $\E|\Lambda_n^{(\mathbb J)}|^{2q} $.
First we note,
\begin{equation}\notag
\E|\varepsilon_j^{(\mathbb J)}|^{2q}\le 3^{2q}\sum_{\nu=1}^3\E|\varepsilon_{j\nu}^{(\mathbb J)}|^{2q}.
\end{equation}

\begin{lem}\label{lem2*}
Assuming the conditions of Theorem \ref{stiltjesmain} we have, for any $q\ge1$, and for any $z=u+iv\in\mathbb C_+$,
\begin{equation}\notag
\E|\varepsilon_{j3}^{(\mathbb J)}|^{2q}\le \frac1{n^{2q}v^{2q}}.
\end{equation}
\end{lem}
\begin{proof}For a proof of this Lemma see Lemma 3.2 in \cite{GT:2004}.
\end{proof}
Let $A_1>0$ and $0\le v_1\le 4\sqrt y$ be a fixed.
\begin{lem}\label{lem5} Assuming the  conditions of Theorem \ref{main}, and assuming for all $\mathbb J\subset T$ with $|\mathbb J|\le L$ and 
all $l\in\{n+1,\ldots,n+p\}$
\begin{equation}\label{cond3}
\E|R_{ll}^{(\mathbb J)}|^{q}\le C_0^{q}, \text{ for } 1\le q\le A_1(nv)^{\frac14}\text{  and for }v\ge v_1,
\end{equation} 
and 
for all $\widehat{\mathbb J}\subset\widehat{\mathbb T}$ with $|\widehat{\mathbb J}|\le L$ and 
all $l\in\{n+1,\ldots,n+p\}$
\begin{equation}\label{cond31}
\E|R_{ll}^{(\mathbb J)}|^{q}\le C_0^{q}, \text{ for } 1\le q\le A_1(nv)^{\frac14}\text{  and for }v\ge v_1,
\end{equation} 
we have, for all $v\ge v_1/s_0$, and for all $\mathbb J\subset\mathbb T$ with $|\mathbb J|\le L-1$, 
\begin{equation}\notag
\max\{\E|\varepsilon_{j1}^{(\mathbb J)}|^{2q},\E|\widehat\varepsilon_{j1}^{(\mathbb J)}|^{2q}\}\le(C_1q)^{2q}n^{-q}s_0^{2q}C_0^{4q},\text{ for } 1\le q\le A_1(nv)^{\frac14}.
\end{equation}
\end{lem}

\begin{proof}Recall that $s_0=2^4$ and  note that if $q\le  A_1(nv)^{\frac14}$ for $v\ge v_1/s_0$ then $q'=2q\le A_1(nv)^{\frac14}$ for $v\ge v_1$.
 Let $v':=vs_0$. If $v\ge v_1/s_0$ then $v'\ge v_1$. We have
 \begin{equation}\label{rem1}
  q'=2q\le 2A_1(nv)^{\frac14}= 2A_1(nv's_0^{-1})^{\frac14}=A_1(nv')^{\frac14}.
 \end{equation}
 We apply now Rosenthal's inequality for the moments of sums of independent random variables and get
\begin{align}
\E|\varepsilon_{j1}^{(\mathbb J)}|^{2q}\le (C_1q)^{2q}n^{-2q}\Big(\E\big(\sum_{l=1^p}|R_{l+n,l+n}^{(\mathbb J,j)}|^{2}\big)^{q}
+\E|X_{jl}|^{4q}\sum_{l=1}^p\E|R^{(\mathbb J,j)}_{l+n,l+n}|^{2q}\Big).\notag
\end{align}
According to inequality \eqref{rem1}  we may apply Lemma \ref{schlein} and  
 condition \eqref{cond3} for $q'=2q$. We get, for $v\ge v_1/s_0$,
\begin{align}
\E|\varepsilon_{j1}^{(\mathbb J)}|^{2q}\le(C_1q)^{2q}n^{-q}s_0^{2q}C_0^{2q}.\notag
\end{align}
We use as well that by the conditions of Theorem \ref{main},
$ \E|X_{jl}|^{4q}\le D^{4q-4}n^{q-1}\mu_4$,
and by Jensen's inequality,
$(\frac1n\sum_{l=1}^p|R_{l+n,l+n}^{(\mathbb J,j)}|^{2})^{q}\le \frac1n\sum_{l=1}^p|R_{l+n,l+n}^{(\mathbb J,j)}|^{2q}$.
Similar we get he estimation for $\E|\widehat\varepsilon_{j1}^{(\mathbb J)}|^{2q}$. 
Thus, Lemma \ref{lem5} is proved.
\end{proof}
\begin{lem}\label{lem6}Assuming the  conditions of Theorem \ref{main}, 
condition \eqref{cond3}, 
for $v\ge v_1$ and $q\le A_1(nv)^{\frac14}$,
we have, for any $v\ge v_1/s_0$ and $q\le A_1(nv)^{\frac14}$,
\begin{align}
\max\{\E|\varepsilon_{j2}^{(\mathbb J)}|^{2q},\E|\widehat\varepsilon_{j2}^{(\widehat{\mathbb J})}|^{2q}\}\le 6(C_3q)^{4q}n^{-q}v^{-q}A_q^{(\mathbb J)}+
2(C_3q)^{4q}n^{-q}v^{-q}(C_0s_0)^q.\notag
\end{align}

\end{lem}
\begin{proof}We consider the quantity $\E|\varepsilon_{j2}^{(\mathbb J)}|^{2q}$ only. The other \tc{one}  is similar.
We apply Burkholder's inequality for  quadratic forms. See Lemma \ref{burkh1} in the Appendix. We obtain
\begin{align}
\E|\varepsilon_{j2}^{(\mathbb J)}|^{2q}&\le (C_1q)^{2q}n^{-2q}\Bigg(\E\big(\sum_{l=1}^p\big|
\sum_{k=1}^{l-1}X_{jk}
R^{(\mathbb J,j)}_{l+n,k+n}\big|^2\big)^q\notag\\&+\max_{j,l}\E|X_{jl}|^{2q}
\sum_{l=1}^p
\E\big|\sum_{k=1}^{l-1}X_{jk}
R^{(\mathbb J,j)}_{l+n,k+n}\big|^{2q}\Bigg).\notag
\end{align}
Using now the quantity $Q_0^{(\mathbb J,j)}$ for the first term and Rosenthal's inequality and condition \eqref{moment} for the second term, we obtain
with Lemma \ref{resol00}, inequality \eqref{res2}, in the Appendix and $C_3=C_1C_2$
\begin{align}
\E|\varepsilon_{j2}^{(\mathbb J)}|^{2q}&
\le (C_2q)^{2q}n^{-q}\E|Q^{(\mathbb J,j)}_{0}|^{q}\notag\\&\qquad\qquad\qquad+(C_3q)^{4q}n^{-\frac{3q}2}\frac1n\sum_{l=1}^p\E\big(
\sum_{k=1}^{l-1}|
R^{(\mathbb J,j)}_{l+n,k+n}|^2\big)^{q}\notag\\&\qquad\qquad\qquad+(C_3q)^{4q}n^{-q}\frac1{n^2}
\sum_{l=1}^p\sum_{k=1}^{l-1}\E|R^{(\mathbb J,j)}_{l+n,k+n}|^{2q}\notag\\&\qquad
\le 
(C_2q)^{2q}n^{-q}\E|Q^{(\mathbb J,j)}_{0}|^{q}+(C_3q)^{4q}n^{-\frac{3q}2}v^{-q}\frac1n\sum_{l=1}^p\E|R_{l+n,l+n}^{(\mathbb J,j)}|^q\notag\\&
\qquad\qquad\qquad
+(C_3q)^{4q}n^{-q}v^{-q}\frac1{n^2}
\sum_{l=1}^p\E|R_{l+n,l+n}^{(\mathbb J,j)}|^q.\notag
\end{align}
By Lemma \ref{schlein} and condition \eqref{cond3}, we get
\begin{align}
\E|\varepsilon_{j2}^{(\mathbb J)}|^{2q}&\le 
(C_3q)^{2q}n^{-q}\E|Q^{(\mathbb J,j)}_{0}|^{q}
+2(C_3q)^{3q}n^{-q}v^{-q}(C_0s_0)^q.\notag
\end{align}
Applying now Lemma \ref{bp1}, we get the claim.
Thus, Lemma \ref{lem6} is proved.
\end{proof}
Recall that
\begin{equation}\notag
\Lambda_n^{(\mathbb J)}=\frac1n\sum_{j\in \mathbb T_{\mathbb J}} R^{(\mathbb J)}_{jj}-s_y(z),\quad\text{and}
\quad T_n^{(\mathbb J)}(z)=\frac1n\sum_{j\in\mathbb T_{\mathbb J}}\varepsilon_j^{(\mathbb J)}R_{jj}^{(\mathbb J)}.
\end{equation}
and
\begin{equation}\notag
 \Lambda_n^{(\widehat{\mathbb J})}=\frac1n\sum_{j=1}^n R^{(\mathbb J)}_{jj}-s_y(z)+\frac{(p-|J|-n)_+}{nz},\quad\text{and}
\quad T_n^{(\mathbb J)}(z)=\frac1n\sum_{j=1}^n\varepsilon_j^{(\widehat{\mathbb J})}R_{jj}^{(\mathbb J)}.
\end{equation}

\begin{lem}\label{lem7}Assuming the conditions of Theorem \ref{main}, we have
\begin{equation}\notag
|\Lambda_n^{(\mathbb J)}|\le C(\sqrt{|T_n^{(\mathbb J)}(z)|}+\frac{\sqrt{|\mathbb J|}}{\sqrt n}).
\end{equation}

\end{lem}
\begin{proof}It is similar to the  inequality (2.10) in \cite{SchleinMaltseva:2013}. For completeness we include short proof here. Obviously
\begin{align}\label{lambda00}
\Lambda_n^{(\mathbb J)}(z)&=
\frac{(T_n^{(\mathbb J)}(z)-\frac{|\mathbb J|}n)}{z+2ys_y(z)+\frac{y-1}z-y\Lambda_n^{(\mathbb J)}(z)}.
\end{align}
Note that
\begin{equation}\notag
 z+2ys_y(z)+\frac{y-1}z=\sqrt{(z+\frac{y-1}z)^2-4y}.
\end{equation}
Solving this equation \eqref{lambda00} w.r.t. $\Lambda_n^{(\mathbb J)}(z)$, we get
\begin{equation}
\Lambda_n^{(\mathbb J)}(z)=\frac{-\sqrt{(z+\frac{y-1}z)^2-4y}+\sqrt{(z+\frac{y-1}z)^2-4y-4y\widetilde T_n^{(\mathbb J)}(z)}}2,
\end{equation}
where
\begin{equation}\notag
 \widetilde T_n(z)=s_y(z)(T_n^{(\mathbb J)}(z)-\frac{|\mathbb J|}n).
\end{equation}
We take here the branch of $\sqrt z$ such that $\im \sqrt z\ge 0$.
Since for any $a,b\in\mathbb C$
$|\sqrt{a+b}-\sqrt a|\le C\frac{|b|}{\sqrt{|a|+|b|}}$,
we get
\begin{equation}\notag
|\Lambda_n^{(\mathbb J)}(z)|\le C(\sqrt{|T_n^{(\mathbb J)}(z)|}+\frac{\sqrt{|\mathbb J|}}{\sqrt n}).
\end{equation}
Thus, Lemma \ref{lem7} is proved.
\end{proof}
\begin{lem}\label{lem8}Assuming the  conditions of Theorem \ref{main} and 
condition \eqref{cond3},  we obtain, for $|\mathbb J|\le Cn^{\frac12}$
\begin{align}
\E|\Lambda_n^{(\mathbb J)}|^{2q}&\le\frac{C^q}{n^{\frac q4}}+ \Big(\frac{\mu_4}
{n^{\frac q2+1}}+\frac1{n^{2q}v^{2q}}+(C_1q)^{2q}n^{-q}s_0^{2q}C_0^{4q}\notag\\&\quad+6(C_3q)^{4q}n^{-q}v^{-q}A_{q}^{(\mathbb J)}+
2(C_3q)^{4q}n^{-q}v^{-q}(C_0s_0)^q\Big)^{\frac12}(C_0s_0)^q. \notag
\end{align}
\end{lem}
\begin{proof}
By Lemma \ref{lem6}, we have
\begin{equation}\notag
\E|\Lambda_n^{(\mathbb J)}|^{2q}\le C^p\E|T_n^{(\mathbb J)}(z)|^q+\frac {|\mathbb J|^{\frac q2}}{n^{\frac q2}}\le C^q\E|T_n^{(\mathbb J)}(z)|^q+\frac {C}{n^{\frac q4}}.
\end{equation}
Furthermore,
\begin{align}
\E|T_n^{(\mathbb J)}(z)|^q\le\Big(\frac1n\sum_{j\in\mathbb T_{\mathbb J}}\E|\varepsilon_{j}^{(\mathbb J)}|^{2q}\Big)^{\frac12}
\Big(\frac1n\sum_{j\in\mathbb T_{\mathbb J}}\E|R_{jj}^{(\mathbb J)}|^{2q}\Big)^{\frac12}.\notag
\end{align}
Lemmas \ref{lem2*} -- \ref{lem6} together with Lemma \ref{schlein} imply
\begin{align}\label{ineq1}
\frac1n\sum_{j\in\mathbb T_{\mathbb J}}\E|\varepsilon_{j}^{(\mathbb J)}|^{2q} \le  & 4^{2q-1}\Big( \frac{\mu_4}
{n^{\frac q2+1}}+\frac1{n^{2q}v^{2q}}+(C_1 q)^{2q}n^{-q}s_0^{2q}C_0^{4q}\notag\\&
+6(C_3q)^{4q}n^{-q}v^{-q}A_{q}^{(\mathbb J)}+2(C_3q)^{4q}n^{-q}v^{-q}(C_0s_0)^q\Big).
\end{align}

Thus, Lemma \ref{lem8} is proved.
\end{proof}
\begin{lem}\label{lem01}Assuming the  conditions of Theorem \ref{main} and 
condition \eqref{cond3}, there exists an absolute constant $C_4$ such that, for
$q\le A_1(nv)^{\frac14}$ and $v\ge v_1/s_0$, we have, uniformly in  $\mathbb J\subset \mathbb T$ such that $|\mathbb J|\le Cn^{\frac12}$, and for $z\in\mathbb G$,
\begin{align}\notag
A_q^{(\mathbb J)}\le  C_4^q.
\end{align}
\end{lem}
\begin{proof}

We start from the obvious inequality, using $|s_y(z)|\le 1/\sqrt y$,
\begin{align}\notag
\E|\Psi^{(\mathbb J)}|^{2q}\le 3^{2q}(\frac1{y^q}+\left(\frac{v|\mathbb J|}{n|z|^2}\right)^{2q}+\left(\frac{v(1-y)}{y|z|^2}\right)^{2q}+(nv)^{-2q}+\E|\Lambda_n^{(\mathbb J)}(z)|^{2q}).
\end{align}
Note that, for $z\in\mathbb G$,
\begin{equation}
 \frac{v|\mathbb J|}{n|z|^2}\le \frac{|\mathbb J|}{n|z|}\le \frac C{\sqrt n|z|}\le \frac C{\sqrt{nv}}\le \frac1{2\sqrt y}.
\end{equation}
It is straightforward to check that, for $z\in\mathbb G$,
\begin{equation}
 \frac{v(1-y)}{y|z|^2}\le \frac{1-y}{y|z|}\le \frac{1+\sqrt y}y.
\end{equation}
Furthermore, applying Lemma \ref{lem8}, we get
\begin{align}\label{a10}
\E|\Psi^{(\mathbb J)}|^{2q}&\le 3^{2q}\Bigg(\left(\frac 2{y}\right)^{2q}+(nv)^{-2q}+\Big(\frac{\mu_4}
{n^{\frac q2+1}}+\frac1{n^{2q}v^{2q}}+(C_1q)^{2q}n^{-q}s_0^{2q}C_0^{4q}\notag\\&+6(C_3q)^{4q}n^{-q}v^{-q}A_{q}^{(\mathbb J)}+
2(C_3q)^{4q}n^{-q}v^{-q}(C_0s_0)^q\Big)^{\frac12}(C_0s_0)^q\Bigg).
\end{align}
By definition,
\begin{equation}\label{a1}
A_q^{(\mathbb J)}\le 1+\E^{\frac12}(\Psi^{(\mathbb J)})^{2q}.
\end{equation}
Inequalities \eqref{a1} and \eqref{a10} together imply
\begin{align}
A_q^{(\mathbb J)}&\le1+3^{q}\Bigg(1+(nv)^{-\frac{q}2}+(C_0s_0)^{\frac q2}\Big(\mu_4^{\frac14}n^{-\frac q8}+\frac1{n^{\frac q2}v^{\frac q2}}
+(C_1q)^{\frac q2}n^{-\frac q4}s_0^{\frac q2}C_0^{q}\notag\\&+3(C_3q)^{q}n^{-\frac q4}v^{
-\frac q4}(A_q^{(\mathbb J)})^{\frac14}+
2(C_3q)^{q}n^{-\frac q4}v^{-\frac q4}(C_0s_0)^{\frac q4}\Big)\Bigg).\notag
\end{align}
Let $C'=s_0\max\{9,C_3^{\frac12}, C_0^{\frac32}C_1^{\frac12}, 3C_3C_0^{\frac12}, C_3C_0^{\frac34}\}$.
The last inequality implies that
\begin{align}
A_q^{(\mathbb J)}&\le {C'}^q\Big(1+(nv)^{-\frac{q}2}+\mu_4^{\frac12}n^{-\frac q8}+\frac1{n^{\frac q2}v^{\frac q2}}\notag\\&\qquad\qquad\qquad
+q^{\frac q2}n^{-\frac q4}+q^{q}n^{-\frac q4}v^{-\frac q4}+
q^{\frac{4q}3}n^{-\frac q3}v^{-\frac q3}\Big).\notag
\end{align}
For $q\le A_1(nv)^{\frac14}$, we get, for $z\in\mathbb G$,
\begin{equation}\notag
A_q^{(\mathbb J)}\le C_4^q,
\end{equation}
where $C_4$ is some absolute constant. We may take $C_4=2C'$.
\end{proof}
\begin{cor}\label{cor4.8}Assuming the  conditions of Theorem \ref{main} and  
condition \eqref{cond3}, we have , for $v\ge v_1/s_0$, and for any $\mathbb J\subset \mathbb T$ such that $|\mathbb J|\le \sqrt n$
\begin{align}\label{lll}
\E|\Lambda_n^{(\mathbb J)}|^{2q}&\le  C_0^{2q}\Big(\frac{4^{\frac q4}\mu_4^{\frac12}s_0^{q}}{n^{\frac q4}v^{\frac q4}}+\frac{s_0^{\frac q2}}{n^qv^q}
+\frac{C_5^qq^{2q}}{n^{\frac q2}v^{\frac q2}}\Big),
\end{align}
where
\begin{align}\notag
 C_5:=4C_1^2s_0^{4}+6^{\frac1q}C_3^4C_4+2^{\frac1q}C_3^4s_0^3.
\end{align}

\end{cor}
\begin{proof}
 Without loss of generality we may assume that $C_0>1$. The bound \eqref{lll} follows now from Lemmas \ref{lem8} and \ref{lem01}.
\end{proof}

\begin{lem}\label{lem9}Assuming the  conditions of Theorem \ref{main} and
condition \eqref{cond3} for $\mathbb J\subset \mathbb T$ such that $|\mathbb J|\le L\le \sqrt n$, there exist positive constant $A_0, C_0, A_1$ depending on 
$\mu_4, D$ only 
, such that we have, for $q\le A_1(nv)^{\frac14}$ and $v\ge v_1/s_0$  uniformly in $\mathbb J$ and $v_1$, for $j\in\mathbb T_{\mathbb J}\cup\{n+1,\ldots,n+p\}$
\begin{align}\notag
\E|R_{jj}^{(\mathbb J)}|^q\le C_0^q
\end{align}
and 
for $\widehat{\mathbb J}\subset \widehat{\mathbb T}$ such that $|\widehat{\mathbb J}|\le L\le \sqrt n$, there exist positive constant $A_0, C_0, A_1$ 
depending on $\mu_4, D$ only 
, such that we have, for $q\le A_1(nv)^{\frac14}$ and $v\ge v_1/s_0$  uniformly in $\widehat {\mathbb J}$ and $v_1$, for any 
$j\in\{1,\ldots,n\}\cup\widehat{\mathbb T}_{\widehat{\mathbb J}}$,
\begin{align}\notag
\E|R_{jj}^{(\widehat{\mathbb J})}|^q\le C_0^q
\end{align}
with $|\widehat{\mathbb J}|\le L-1$.
\end{lem}
\begin{proof}According to inequality \eqref{ineq10}, we have
\begin{align}
\E|R_{jj}^{(\mathbb J)}|^q&\le4^q(1+(\E^{\frac12}|\Lambda_n^{(\mathbb J)}|^{2q}+\E^{\frac12}|\varepsilon_j^{(\mathbb J)}|^{2q})
\E^{\frac12}|R_{jj}^{(\mathbb J)}|^{2q}).
\end{align}
Applying  condition \eqref{cond3}, we get
\begin{align}
\E|R_{jj}^{(\mathbb J)}|^q&\le 4^q(1+(\E^{\frac12}|\Lambda_n^{(\mathbb J)}|^{2q}+\E^{\frac12}|\varepsilon_{j1}^{(\mathbb J)}|^{2q}
+\cdots+\E^{\frac12}|\varepsilon_{j4}^{(\mathbb J)}|^{2q})s_0^qC_0^q).\notag
\end{align} 
Combining results of Lemmas \ref{lem2*} -- \ref{lem6} and Corollary \ref{cor4.8},  we obtain
\begin{align}
\E|&R_{jj}^{(\mathbb J)}|^q\le5^q\Bigg(1+s_0^qC_0^{2q}\Big(\frac{4^{\frac q4}\mu_4^{\frac14}s_0^{q}}{n^{\frac q4}v^{\frac q4}}+\frac{s_0^{q}}{n^qv^q}
+\frac{C_5^qq^{2q}}{n^{\frac q2}v^{\frac q2}}\Big)^{\frac12}\notag\\&\qquad\qquad\qquad+
s_0^qC_0^{3q}\bigg(\frac{4^{\frac q4}\mu_4^{\frac14}}{n^{\frac q4}v^{\frac q4}}
+\frac{s_0^{q}}{n^qv^q}
+\frac{C_5^qq^{2q}}{n^{\frac q2}v^{\frac q2}}
 \bigg)\Bigg).
\end{align}
We may rewrite the last inequality as follows
\begin{align}
\E|R_{jj}^{(\mathbb J)}|^q\le C_0^q\Big(\frac{5^q}{C_0^q}+\frac{{\widehat C}_1^{\frac q8}}{(nv)^{\frac q8}} +\frac {({\widehat C}_2q^4)^{\frac q4}}{(nv)^{\frac q4}}+
\frac {({\widehat C}_3q^4)^{\frac q2}}{(nv)^{\frac q2}}+
\frac {{\widehat C}_4^{ q}}{(nv)^{q}}\Big),\notag
\end{align}
where
\begin{align}
 {\widehat C}_1&=5^8s_0^{12}C_0^8\mu_4^{\frac1p},\notag\\
 {\widehat C}_2&=5^4s_0^4C_0^4C_5^2(1+2C_0^4\mu_4^{\frac2p}),\notag\\
 {\widehat C}_3&=5^2s_0^2C_0^2(s_0+C_5^2),\notag\\
 {\widehat C}_4&=5C_0^2s_0^2.\notag
\end{align}
Note that for 
\begin{equation}\label{a0}
 A_0\ge 2^8A_1^4\max\{{\widehat C}_1,\ldots,{\widehat C}_4\}
\end{equation}
 and $C_0\ge 25$, we obtain
 that  
\begin{equation}\notag
\E|R_{jj}^{(\mathbb J)}|^q\le C_0^q.
\end{equation}
Thus Lemma \ref{lem9} is proved.
\end{proof}

\begin{cor}\label{cor8}Assuming the  conditions of Theorem \ref{main}, 
  we have, for $q\le 8$ and $v\ge v_0= A_0n^{-1}$
there exist a  constant $C_0>0$ depending on $\mu_4$ and $D$ only such that for all $1\le j\le n+p$ and all $z\in\mathbb G$
\begin{equation}\label{r11}
\E|R_{jj}|^q\le C_0^q,
\end{equation}
and
\begin{equation}\label{r12}
 \E\frac1{|z+ym_n(z)+\frac{y-1}z|^q}\le C_0^q.
\end{equation}

\end{cor}
\begin{proof}Let $L=[-\log_{s_0}v_0]+1$. Note that
$s_0^{-1}v_0\le s_0^{-L}\le v_0$ and $A_1 {n^{\frac14}}{s_0^{-\frac L4}}\ge s_0^{-\frac14}A_1(nv_0)^{\frac14}$.
We may choose $C_0=25$ and $A_0, A_1$ such that \eqref{a0} holds and
\begin{equation}
A_1(nv)^{\frac14}\ge 8.
\end{equation}
 Then, 
for $v=1$, and for any $q\ge 1$, for any set $\mathbb J\subset\mathbb T$ such that $|\mathbb J|\le L$
\begin{equation}\label{in10}
\E|R_{jj}^{(\mathbb J)}|^q\le C_0^q.
\end{equation}
By Lemma \ref{lem9}, inequality \eqref{in10} holds for $v\ge 1/s_0$ and for
$q\le A_1n^{\frac14}/s_0^{\frac14}$ and for $\mathbb J\subset\mathbb T$ such that $|\mathbb J|\le L-1$. After repeated
application of Lemma  \ref{lem9} (with  \eqref{in10} as assumption valid for $v \ge 1/s_0$) 
we arrive at the conclusion that the
inequality \eqref{in10} holds for $v\ge 1/s_0^{2}$,
$p\le A_1n^{\frac14}/s_0^{\frac12}$ and all $\mathbb J\subset\mathbb T$ such that $|\mathbb J|\le L-2$. 
Continuing this iteration inequality\tc{, ienequlity}  \eqref{in10} finally  holds
for $v\ge A_0n^{-1}$, $q\le 8$  and $\mathbb J=\emptyset$.\\
The proof of inequality of \eqref{r12} is similar.  We have by \eqref{stmarpas} ,
\begin{align}\label{recip}
 \frac1{|z+ym_n(z)+\frac{y-1}z|}&\le \frac1{|ys(z)+z+\frac{y-1}z|}+\frac{|\Lambda_n|}{|z+ym_n(z)+\frac{y-1}z||z+ys(z)+\frac{y-1}z|}\notag\\&
 \le |s(z)|(1+\frac{|\Lambda_n|}{|z+ym_n(z)+\frac{y-1}z|}).
\end{align}
Furthermore, using that $|m_n'(z)|\le \frac1n\sum_{j=1}^n\E|R_{jj}|^2$ and 
$z+ym_n(z)+\frac{y-1}z|\ge v+y\im m_n(z)+\frac{v(1-y)}{|z|^2}$
, we get
\begin{align}
 |\frac d{dz}\log(z+ym_n(z)+\frac{y-1}z)|&\le \frac{|1+ym_n'(z)-\frac{y-1}{z^2}|}{|z+ym_n(z)+\frac{y-1}z|}\notag\\&
 \le \frac1v\frac{v+y\im m_n(z)+\frac{(1-y)v}{|z|^2}}{|v+y\im m_n(z)+\frac{v(1-y)}{|z|^2}|}\le \frac1v.\notag
\end{align}
By integration, this implies that (see the proof of Lemma \ref{schlein})
\begin{equation}\label{inverse}
\frac1{|(u+iv/s_0)+ym_n(u+iv/s_0)+\frac{y-1}{u+iv/s_0}|}\le \frac {s_0}{|(u+iv)+ym_n(u+iv)+\frac{y-1}{u+iv}|}.
\end{equation}
Inequality \eqref{recip} and the Cauchy--Schwartz inequality together imply 
\begin{align}\notag
 \E\frac1{|z+ym_n(z)+\frac{y-1}z|^q}\le 2^q|s(z)|^q(1+\E^{\frac12}|\Lambda_n|^{2q}\E^{\frac12}\frac1{|z+ym_n(z)+\frac{y-1}z|^{2q}}).
\end{align}
Applying inequality \eqref{inverse}, we obtain
\begin{align}\notag
 \E\frac1{|z+ym_n(z)+\frac{y-1}z|^q}&\le 2^q|s(z)|^q(1+\E^{\frac12}|\Lambda_n|^{2q}s_0^qC_0^q).
\end{align}
Using Corollary \ref{cor4.8}, we get, for $v\ge 1/s_0$
\begin{align}
 \E\frac1{|z+ym_n(z)+\frac{y-1}z|^q}&\le 2^q|s(z)|^q\Big(1+\Big(\frac{4^{\frac q8}\mu_4^{\frac14}s_0^{\frac q2}}{n^{\frac q8}v^{\frac q8}}
 +\frac{s_0^{\frac q4}}{n^{\frac q2}v^{\frac q2}}
+\frac{C_5^qq^{q}}{n^{\frac q4}v^{\frac q4}}\Big)s_0^qC_0^{2q}\Big).\notag
\end{align}
Thus inequality \eqref{r12} holds for $v\ge 1/s_0$ as well. Repeating this argument inductively with $A_0,A_1,C_)$ satisfying \eqref{a0}
for the regions $v \ge s_0^{-\nu}$, for $\nu=1,\ldots,L$ and  $z \in \mathbb G$, we get the claim.
Thus, Corollary \ref{cor8} is proved.

\end{proof}


\section{{Proof} of Theorem \ref{stiltjesmain}}\label{expect}
We return now to the representation \eqref{repr01} which implies that
\begin{align}\label{lambda}
s_n(z)&=\frac1n\sum_{j=1}^n\E R_{jj}=s_y(z)+\E\Lambda_n=s_y(z)+\E\frac{T_n(z)}{z+y(s(z)+m_n(z))+\frac{y-1}z}.
\end{align}
We may continue the last equality as follows
\begin{align}\label{eq00}
s_n(z)=s_y(z)+\E\frac{\frac1n\sum_{j=1}^n\varepsilon_{j3}R_{jj}}{z+y(s_y(z)+m_n(z))+\frac{y-1}z}+
\E\frac{\widehat T_n(z)}{z+y(s(z)+m_n(z))+\frac{y-1}z},
\end{align}
where 
$$
\widehat T_n=\frac1n\sum_{j=1}^n(\varepsilon_{j1}+\varepsilon_{j2})R_{jj}.
$$


Note that the definition of $\varepsilon_{j3}$ in \eqref{repr01} and equality \eqref{rrr7} in the Appendix together imply 
\begin{equation}\label{7.3}
\frac1n\sum_{j=1}^n\varepsilon_{j3}R_{jj}=\frac y{2n}(-m_n'(z)+\frac{m_n(z)}z).
\end{equation}

Thus we may rewrite \eqref{eq00} as
\begin{align}\label{eq01}
s_n(z)=s_y(z)&-\frac y{2n}\E\frac{ m_n'(z)-\frac{m_n(z)}z}{z+y(s_y(z)+m_n(z))+\frac{y-1}z}\notag\\&+
\E\frac{\widehat T_n(z)}{z+y(s_y(z)+m_n(z))+\frac{y-1}z}.
\end{align}

Denote by
\begin{equation}\mathfrak T=\E\frac{\widehat T_n(z)}{z+y(s_y(z)+m_n(z))+\frac{y-1}z}.\notag
\end{equation}

\subsection{Estimation of  $\mathfrak T$} We represent $\mathfrak T$ 
\begin{align}
\mathfrak T=\mathfrak T_{1}+\mathfrak T_{2},\notag
\end{align}
where
\begin{align}
\mathfrak T_{1}&=-\frac1n\sum_{j=1}^n\E\frac{(\varepsilon_{j1}+\varepsilon_{j2})
\frac1{z+ym_n^{(j)}(z)+\frac{y-1}z}}{z+y(m_n(z)+s_y(z))+\frac{y-1}z},\notag\\
\mathfrak T_{2}&=\frac1n\sum_{j=1}^n
\E\frac{(\varepsilon_{j1}+\varepsilon_{j2})(R_{jj}+\frac1{z+ym_n^{(j)}(z)+\frac{y-1}z})}{z+y(m_n(z)+s_y(z))+\frac{y-1}z}.\notag
\end{align}
\subsubsection{Estimation  of $\mathfrak T_{1}$}
We may decompose  $\mathfrak T_{1}$ as
\begin{align}\label{tii}
\mathfrak T_{1}=\mathfrak T_{11}+\mathfrak T_{12},
\end{align} 
where
\begin{align}
\mathfrak T_{11}&=-\frac1n\sum_{j=1}^n\E\frac{(\varepsilon_{j1}+\varepsilon_{j2})
\frac1{z+ym_n^{(j)}(z)+\frac{y-1}z}}{z+y(m_n^{(j)}(z)+s_y(z))+\frac{y-1}z},\notag\\
\mathfrak T_{12}&=-\frac1n\sum_{j=1}^n
\E\frac{(\varepsilon_{j1}+\varepsilon_{j2})\widetilde\varepsilon_{j3}
\frac1{z+ym_n^{(j)}(z)+\frac{y-1}z}}{(z+y(m_n^{(j)}(z)+s_y(z))+\frac{y-1}z)(z+y(m_n(z)+s_y(z))+\frac{y-1}z)},\notag
\end{align}
where
\begin{equation}\notag
\widetilde\varepsilon_{j3}=y\varepsilon_{j3}-\frac1{nz}.
\end{equation}

It is easy to see that,  by conditional expectation
\begin{equation}\label{fin104}
\mathfrak T_{11}=0.
\end{equation}
Applying the Cauchy--Schwartz inequality, we get, for $\nu=1,2$,
\begin{align}\label{fin100}
\Bigg|\E&\frac{\varepsilon_{j\nu}\widetilde\varepsilon_{j3}
\frac1{z+ym_n^{(j)}(z)+\frac{y-1}z}}{(z+y(m_n^{(j)}(z)+s_y(z))+\frac{y-1}z)(z+y(m_n(z)+s_y(z))+\frac{y-1}z)}\Bigg|\notag\\
&\le 
\E^{\frac12}\Bigg|\frac{\varepsilon_{j\nu}}{(z+ym_n^{(j)}(z)+\frac{y-1}z)(z+y(m_n^{(j)}(z)+s_y(z))+\frac{y-1}z)}\Bigg|^2\\
&\qquad\qquad\qquad\times\E^{\frac12}\Bigg|\frac{\widetilde\varepsilon_{j3}}{z+y(m_n(z)+s(z))+\frac{y-1}z}\Bigg|^2.
\end{align}
Applying the Cauchy -- Schwartz inequality again, we get
\begin{align}\label{fin200}
 \E^{\frac12}\Bigg|&\frac{\varepsilon_{j\nu}}{(z+ym_n^{(j)}(z)+\frac{y-1}z)(z+y(m_n^{(j)}(z)+s_y(z))+\frac{y-1}z)}\Bigg|^2
 \notag\\&\le 
 \E^{\frac14}\frac{|\varepsilon_{j\nu}|^4}{|z+y(m_n^{(j)}(z)+s_y(z))+\frac{y-1}z|^4}\E^{\frac14}\frac1{|z+ym_n^{(j)}(z)+\frac{y-1}z|^4}.
\end{align}
Inequalities \eqref{fin100}, \eqref{fin200}, Corollary  \ref{cor8} together imply
\begin{align}\label{f19}
 \Bigg|\E&\frac{\varepsilon_{j\nu}\widetilde\varepsilon_{j3}
\frac1{z+ym_n^{(j)}(z)+\frac{y-1}z}}{(z+y(m_n^{(j)}(z)+s(z))+\frac{y-1}z)(z+y(m_n(z)+s(z))+\frac{y-1}z)}\Bigg|\notag\\
&\qquad\qquad\qquad\qquad\qquad\qquad\le \frac C{nv}\E^{\frac14}\frac{|\varepsilon_{j\nu}|^4}{|z+ym_n^{(j)}(z)+\frac{y-1}z+s(z)|^4}.
\end{align}

By Corollary \ref{corgot}, inequality \eqref{ingot} with $\alpha=0$ and $\beta=4$ in the Appendix we have for $\nu=1,2$
\begin{align}\label{f21}
 \E^{\frac14}\frac{|\varepsilon_{j\nu}|^4}{|z+y(m_n^{(j)}(z)+s(z))+\frac{y-1}z|^4}\le \frac {C}{\sqrt {nv}|(z+\frac{y-1}z)^2-4y|^{\frac14}}
\end{align}
with some constant $C>0$ depending on $\mu_4$ and $D$ only.
We get from \eqref{f19},  and  \eqref{f21} that
for $z\in\mathbb G$, 
\begin{align}\label{finisch1}
|\mathfrak T_{1}|\le \frac C{(nv)^{\frac32}|(z+\frac{y-1}z)^2-4y|^{\frac14}}
\end{align}
with some constant $C>0$ depending on $\mu_4$ and $D$ only.
\subsubsection{Estimation of $\mathfrak T_{2}$}Using the representation \eqref{repr01}, we rewrite
 $\mathfrak T_2$ in the form
 \begin{equation}\notag
 \mathfrak T_2=\frac1n\sum_{j=1}^n\E\frac{(\varepsilon_{j1}+\varepsilon_{j2})^2R_{jj}}
 {(z+ym_n^{(j)}(z)+\frac{y-1}z)(z+y(s_y(z)+m_n(z))+\frac{y-1}z)}.
 \end{equation}
 We decompose $\mathfrak T_2$ as follows
\begin{equation}
 \mathfrak T_{2}=\mathfrak T_{21}+\mathfrak T_{22},
\end{equation}
where
\begin{align}
\mathfrak T_{21}&=\frac1n\sum_{j=1}^n\E\frac{\varepsilon_{j2}^2R_{jj}}{(z+ym^{(j)}(z)+\frac{y-1}z)(z+y(s_y(z)+m_n(z))+\frac{y-1}z)},\notag\\
 \mathfrak T_{22}&=\frac1n\sum_{j=1}^n
 \E\frac{(\varepsilon_{j1}^2+2\varepsilon_{j1}\varepsilon_{j2})R_{jj}}{(z+ym^{(j)}(z)+\frac{y-1}z)(z+y(s_y(z)+m_n(z))+\frac{y-1}z)}.\notag
\end{align}
Applying the Cauchy -- Schwartz inequality and inequality \eqref{lar1} in the Appendix, we obtain, for $z\in\mathbb G$,
\begin{align}\label{lang2}
|\mathfrak T_{22}|&\le \frac Cn\sum_{j=1}^n\E^{\frac12}\frac{|{\varepsilon}_{j1}|^2|{\varepsilon}_{j1}+2\varepsilon_{j2}|^2}{|z+ym^{(j)}(z)+\frac{y-1}z|^2
|z+y(s_y(z)+m_n^{(j)}(z))+\frac{y-1}z|^2}\E^{\frac12}|R_{jj}|^2.
\end{align}
We note that
\begin{align}
 \E\{|{\varepsilon}_{j1}|^2{|\varepsilon}_{j1}+2\varepsilon_{j2}|^2\Big|\mathfrak M^{(j)}\}&\le C(\E^{\frac12}\{|\varepsilon_{j1}|^4\Big|\mathfrak M^{(j)}\}
 (\E^{\frac12}\{|\varepsilon_{j1}|^4\Big|\mathfrak M^{(j)}\}+\E^{\frac12}\{|\varepsilon_{j2}|^4\Big|\mathfrak M^{(j)}\}).\notag
\end{align}
Using Lemmas  \ref{basic2}, \ref{basic5},  we get
\begin{align} 
\E\{|{\varepsilon}_{j1}|^2{|\varepsilon}_{j1}+2\varepsilon_{j2}|^2\Big|\mathfrak M^{(j)}\}
\le \frac C{n^2}+\frac C{n^2v}(\im m_n^{(j)}(z)+\frac{(1-y)v}{|z|^2}).\notag
\end{align}
This implies that
\begin{align}
|\mathfrak T_{22}|&\le (\frac C{n|(z^2+\frac{y-1}z)^2-4y|^{\frac12}}+\frac C{n\sqrt v|(z^2+\frac{y-1}z)^2-4y|^{\frac14}})\notag\\&\qquad\qquad\qquad\qquad\times\frac1n\sum_{j=1}^n\E^{\frac12}\frac1{|z+ym_n^{(j)}(z)+\frac{y-1}z|^2}\E|R_{jj}|^2.\notag
\end{align}
Applying  Corollary \ref{cor8}, we get
\begin{equation}\label{finish2}
|\mathfrak T_{22}|\le \frac C{n v^{\frac34}}.
\end{equation}


We continue now with $\mathfrak T_{21}$. We represent it in the form
\begin{equation}\label{t21ii}
\mathfrak T_{21}=H_1+H_2,
\end{equation}
where
\begin{align}
H_1&=-\frac1n\sum_{j=1}^n\E\frac{{\varepsilon}_{j2}^2}{(z+ym^{(j)}(z)+\frac{y-1}z)^2(z+y(s(z)+m_n(z))+\frac{y-1}z)},\notag\\
H_2&=\frac1n\sum_{j=1}^n\E\frac{{\varepsilon}_{j2}^2(R_{jj}+\frac1{z+ym_n^{(j)}(z)+\frac{y-1}z})}{(z+ym^{(j)}(z)+\frac{y-1}z)(z+y(s(z)+m_n(z))+\frac{y-1}z)}.\notag
\end{align}
Furthermore, using the representation
\begin{equation}\notag
 R_{jj}=-\frac1{z+ym_n^{(j)}(z)+\frac{y-1}z}+\frac1{z+ym_n^{(j)}(z)+\frac{y-1}z}(\varepsilon_{j1}+\varepsilon_{j2})R_{jj}
\end{equation}
(compare with  \eqref{repr001}), 
we bound $H_2$ in the following way
\begin{align}\notag
|H_2|\le H_{21}+H_{22},
\end{align}
where
\begin{align}
H_{21}&=\frac1n\sum_{j=1}^n\E\frac{4|{\varepsilon}_{j1}|^3|R_{jj}|}{|z+ym^{(j)}(z)+\frac{y-1}z|^2|z+y(s(z)+m_n(z))+\frac{y-1}z|},\notag\\
H_{22}&=\frac1n\sum_{j=1}^n\E\frac{2|{\varepsilon}_{j2}|^3|R_{jj}|}{|z+ym^{(j)}(z)+\frac{y-1}z|^2|z+y(s(z)+m_n(z))+\frac{y-1}z|}.\notag
\end{align}
Using inequality \eqref{lar1} in the Appendix and H\"older  inequality, we get, for $\nu=1,2$
\begin{align}\label{inr3}
H_{2\nu}&\le \frac1n\sum_{j=1}^n\E\frac{4|{\varepsilon}_{j\nu}|^3|R_{jj}|}{|z+ym^{(j)}(z)+\frac{y-1}z|^2|z+y(s(z)+m_n^{(j)}(z))+\frac{y-1}z|}\notag\\&
+\frac1n\sum_{j=1}^n\E\frac{4|{\varepsilon}_{j\nu}|^3|R_{jj}||\varepsilon_{j3}|}{|z+ym^{(j)}(z)+\frac{y-1}z|^2|z+y(s(z)+m_n^{(j)}(z))+\frac{y-1}z|}\notag\\&
\qquad\qquad\qquad\times\frac1{
|z+y(s(z)+m_n(z))+\frac{y-1}z|}\notag\\&
\le \frac Cn\sum_{j=1}^n\E^{\frac34}\frac{|{\varepsilon}_{j\nu}|^4}{|z+ym^{(j)}(z)+\frac{y-1}z|^{\frac83}|z+y(s(z)+m_n^{(j)}(z))+\frac{y-1}z|^{\frac43}}
\E^{\frac14}|R_{jj}|^4.
\end{align}
Applying  Corollary \ref{corgot} with $\beta=\frac43$ and $\alpha=\frac83$, we obtain,  for $z\in\mathbb G$, and for $\nu=1,2$
\begin{align}
\E^{\frac34}\frac{|{\varepsilon}_{j2}|^4}{|z+ym^{(j)}(z)+\frac{y-1}z|^{\frac83}|z+y(s(z)+m_n^{(j)}(z))+\frac{y-1}z|^{\frac43}}\le
\frac C{(nv)^{\frac32}}.\notag
\end{align}
This yields together with Corollary \ref{cor8} and inequality \eqref{inr3}
\begin{equation}\label{finisch4}
 H_2\le \frac C{(nv)^{\frac32}}.
\end{equation}

Consider now $H_1$. Using the equality
\begin{align}
 &\frac1{z+y(m_n(z)+s(z))+\frac{y-1}z}=\frac1{z+2s_y(z)+\frac{y-1}z}\notag\\&\qquad\qquad\qquad\qquad-\frac{\Lambda_n(z)}{(z+2s_y(z)+\frac{y-1}z)(z+y(m_n(z)+s(z))+\frac{y-1}z)}\notag
\end{align}
and
\begin{equation}\label{lepsj4}
 \Lambda_n=\Lambda_n^{(j)}+\widetilde\varepsilon_{j3},
\end{equation}
we represent it in the form
\begin{align}\label{h1}
H_1=H_{11}+H_{12}+H_{13},
\end{align}
where
\begin{align}
H_{11}&=-\frac1{(z+ys(z)+\frac{y-1}z)^2}\frac1n\sum_{j=1}^n\E\frac{\varepsilon_{j2}^2}{z+y(s(z)+m_n(z))+\frac{y-1}z}\notag\\&
=-s_y^2(z)\frac1n\sum_{j=1}^n\E\frac{\varepsilon_{j2}^2}{z+y(s_y(z)+m_n(z))+\frac{y-1}z},\notag\\
H_{12}&=-\frac1{(z+ys_y(z)+\frac{y-1}z)}\notag\\&\times\frac1n\sum_{j=1}^n\E\frac{\varepsilon_{j2}^2\Lambda_n^{(j)}
}{(z+ym_n^{(j)}(z)+\frac{y-1}z)^2(z+y(s_y(z)+m_n(z))+\frac{y-1}z)},\notag\\
H_{13}&=-\frac1{(z+ys_y(z)+\frac{y-1}z)^2}\notag\\&\times\frac1n\sum_{j=1}^n\E\frac{\varepsilon_{j2}^2
\Lambda_n^{(j)}
}{(z+ym_n^{(j)}(z)+\frac{y-1}z)(z+y(s_y(z)+m_n(z))+\frac{y-1}z)}.\notag
\end{align}
In order to apply conditional independence, we write 
\begin{align}\notag
H_{11}=H_{111}+H_{112},
\end{align}
where
\begin{align}
H_{111}&=-s_y^2(z)\frac1n\sum_{j=1}^n\E\frac{\varepsilon_{j2}^2}{z+y(m_n^{(j)}(z)+s(z))+\frac{y-1}z},\notag\\
H_{112}&=\frac{s_y^2(z)}n\sum_{j=1}^n\E\frac{\varepsilon_{j2}^2\varepsilon_{j3}}{(z+y(s_y(z)+m_n(z))+\frac{y-1}z)(z+y(m_n^{(j)}(z)+s(z))+\frac{y-1}z)}.\notag
\end{align}
It is straightforward to check that
\begin{align}
\E\{\varepsilon_{j2}^2|\mathfrak M^{(j)}\}=\frac1{n^2}\sum_{l,k=1}^p(R^{(j)}_{l+n,k+n})^2-\frac1{n^2}\sum_{l=1}^p(R^{(j)}_{l+n,l+n})^2.\notag
\end{align}
By representation \eqref{matrshur}, we have 
\begin{align}
 \E\{\varepsilon_{j2}^2|\mathfrak M^{(j)}\}&=\frac{y^2z^2}{n^2}\Tr(\mathbf X^*\mathbf X-z^2\mathbf I)^{-2}-\frac {y^2}{n^2}\sum_{l=1}^p(R^{(j)}_{l+n,l+n})^2\notag\\
 &=\frac{y^2z^2}{n^2}\sum_{k\in\mathbb T_j}\frac1{((s_k^{(j)})^2-z^2)^2}-\frac {y^2}{n^2}\sum_{l=1}^p(R^{(j)}_{l+n,l+n})^2.
  \end{align}
Applying equality \eqref{maj}, we get
\begin{equation}\label{majderiv}
 (m_n{(j)}(z))'=\frac 1{n}\sum_{k\in\mathbb T_j}\frac1{((s_k^{(j)})^2-z^2)}-\frac{2z^2}{n}\sum_{k\in\mathbb T_j}\frac1{((s_k^{(j)})^2-z^2)^2}.
\end{equation}
Combining the last two equalities, we arrive
\begin{align}
\E\{\varepsilon_{j2}^2|\mathfrak M^{(j)}\}&=-\frac {y^2}{2n}(m_n^{(j)}(z))'+ \frac1{nz}m_n^{(j)}(z)-\frac{ y^2}{n^2}\sum_{l=1}^p(R^{(j)}_{l+n,l+n})^2.\notag
\end{align}

Using equality \eqref{7.3} for $m_n'(z)$ and the corresponding relation for ${m_n^{(j)}}'(z)$, we may write
\begin{equation}\notag
H_{111}=L_1+\cdots+L_5,
\end{equation}
where
\begin{align}
L_1&=y^2s^2(z)\frac1{2n}\E\frac{m_n'(z)-\frac{m_n(z)}z}{z+y(m_n(z)+s(z))+\frac{y-1}z}\notag\\
L_2&=-y^2s^2(z)\frac{m_n^{(j)}(z)-m_n(z)}{nz(z+y(m_n^{(j)}(z)+s_y(z))+\frac{y-1}z)},\notag\\
L_3&=y^2s^2(z)\frac1n\sum_{j=1}^n\E\frac{\frac1{n^2}\sum_{l=1}^p
(R^{(j)}_{l+n,l+n})^2}{z+y(m_n^{(j)}(z)+s(z))+\frac{y-1}z},\notag\\
L_4&=y^2s^2(z)\frac1n\sum_{j=1}^n\E\frac{\frac1{n}((m_n^{(j)}(z))'-m_n'(z))}{z+y(m_n^{(j)}(z)+s(z))+\frac{y-1}z},\notag\\
L_5&=y^2s^2(z)\frac1n\sum_{j=1}^n\E\frac{\frac1{n}((m_n^{(j)}(z))'-m_n'(z))\widetilde{\varepsilon}_{j3}}
{(z+y(m_n(z)+s(z))+\frac{y-1}z)(z+y(m_n^{(j)}(z)+s(z))+\frac{y-1}z)},\notag\\
L_6&=-y^2s^2(z)\frac1n\sum_{j=1}^n\E\frac{\frac1{n}(m_n^{(j)}(z)-m_n(z))\widetilde{\varepsilon}_{j3}}
{z(z+y(m_n(z)+s(z))+\frac{y-1}z)(z+y(m_n^{(j)}(z)+s(z))+\frac{y-1}z)}.\notag
\end{align}
First we note that
\begin{equation}\notag
 |m_n(z)-m_n^{(j)}(z)|\le \frac C{nv}.
\end{equation}
This inequality together with Lemma \ref{lem00}, inequality \eqref{lem00.2}, imply that
\begin{align}\notag
 |L_2|\le \frac C{n^2v^2\sqrt{|(z+\frac{y-1}z)^2-4y|}}
\end{align}

Using Lemma \ref{lem00}, inequality \eqref{lem00.2}, \ref{basic8},  and Corollary \ref{cor8}, it is 
straightforward to check that
\begin{align}
|L_3|&\le \frac C{n\sqrt{|(z+\frac{y-1}z)^2-4y|}},\notag\\
|L_4|&\le \frac C{n^2v^2{\sqrt{|(z+\frac{y-1}z)^2-4y|}}},\notag\\
|L_5|&\le \frac C{n^3v^3{|(z+\frac{y-1}z)^2-4y|}},\notag\\
|L_6|&\le \frac C{n^3v^3{|(z+\frac{y-1}z)^2-4y|}}.\notag
\end{align}
Applying inequality \eqref{lar1}, we may write
\begin{align}
 |H_{12}|\le \frac Cn\sum_{j=1}^n \E\frac{|\varepsilon_{j2}|^2|\Lambda_n^{(j)}|}
 {|z+ym_n^{(j)}(z)+\frac{y-1}z||z+y(m_n^{(j)}(z)+s(z))+\frac{y-1}z|}.\notag
\end{align}
Conditioning on $\mathfrak M^{(j)}$ and applying Lemma \ref{basic2}, Lemma \ref{lem00}, 
inequality \eqref{lem00.2}, Corollary \ref{cor8} 
and equality \eqref{lepsj4}, we get
\begin{align}
 |H_{12}|&\le \frac C{{nv}}\frac 1n\sum_{j=1}^n \E\frac{|\Lambda_n^{(j)}|}{|z+ym_n^{(j)}(z)+\frac{y-1}z|}\notag\\&
 \le\frac C{{nv}}\frac 1n\sum_{j=1}^n (\E\frac{|\Lambda_n|}{|z+ym_n^{(j)}(z)+\frac{y-1}z|}+ \E\frac{|\Lambda_n-\Lambda_n^{(j)}|}{|z+ym_n^{(j)}(z)+\frac{y-1}z|}\notag\\&
 \le 
 \frac C{nv}\E^{\frac12}|\Lambda_n|^2+\frac C{n^2v^2}.\notag
\end{align}
By Lemma \ref{lam1*}, we get
\begin{equation}\label{h12}
 |H_{12}|\le\frac C{n^2v^2}.
\end{equation}
Similar we get
\begin{equation}\label{h13}
 |H_{13}|\le \frac C{n^2v^2}.
\end{equation}

We rewrite now the equations \eqref{eq00}  and \eqref{eq01} as follows,
\begin{align}\label{final00}
\E\Lambda_n(z)=\E m_n(z)-s(z)=-\frac{y(1-ys_y^2(z))}{2n}\E\frac{m_n'(z)-\frac{m_n(z)}z}{z+y(m_n(z)+s(z))+\frac{y-1}z}+\mathfrak T_3,
\end{align}
where
\begin{align}
|\mathfrak T_3|\le \frac C{n\sqrt v\sqrt{|(z+\frac{y-1}z)^2-4y|}}+\frac C{n^{\frac32}v^{\frac32}|(z+\frac{y-1}z)^2-4y|^{\frac14}}.\notag
\end{align}
We use here  inequalities \eqref{finisch1}, \eqref{finish2}, \eqref{finisch4}, \eqref{h12}, \eqref{h13} to bound $|\mathfrak T_3|$.
Note that
\begin{equation}\notag
1-ys_y^2(z)=-s_y(z)\sqrt{(z+\frac{y-1}z)^2-4y}.
\end{equation}
In \eqref{final00} we estimate now the remaining quantity 
\begin{equation}\notag
\mathfrak T_4=-\frac{ys_y(z)\sqrt{(z+\frac{y-1}z)^2-4y}}{2n}\E\frac{m_n'(z)-\frac{m_n(z)}z}{z+y(m_n(z)+s(z))+\frac{y-1}z}.
\end{equation}

\subsection{Estimation of $\mathfrak T_4$}\label{better}
Using that $\Lambda_n=m_n(z)-s(z)$ we rewrite $\mathfrak T_4$ as
\begin{equation}\notag
\mathfrak T_4=\mathfrak T_{41}+\cdots+\mathfrak T_{45},
\end{equation}
where
\begin{align}
\mathfrak T_{41}&=-\frac{ys_y(z)(s_y'(z)-\frac{s_y(z)}z)}{2n},\notag\\
\mathfrak T_{42}&=\frac{ys_y(z)\sqrt{(z+\frac{y-1}z)^2-4y}}{2n}\E\frac{m_n'(z)-s_y'(z)}{z+y(m_n(z)+s_y(z))+\frac{y-1}z},\notag\\
\mathfrak T_{43}&=\frac{ys_y(z)}{2n}\E\frac{(m_n'(z)-s_y'(z))\Lambda_n}{z+y(m_n(z)+s_y(z))+\frac{y-1}z},\notag\\
\mathfrak T_{44}&=\frac{ys_y(z)\sqrt{(z+\frac{y-1}z)^2-4y}}{2n}\E\frac{m_n(z)-s_y(z)}{z+y(m_n(z)+s_y(z))+\frac{y-1}z},\notag\\
\mathfrak T_{45}&=\frac{ys_y(z)}{2n}\E\frac{(m_n(z)-s_y(z))\Lambda_n}{z+y(m_n(z)+s_y(z))+\frac{y-1}z}.\notag
\end{align}

\subsubsection{Estimation of  $\mathfrak T_{42}$}
First we investigate $m_n'(z)$. The following equality holds
\begin{equation}\label{mn'}
m_n'(z)=-\frac1n\sum_{j=1}^n R_{jj}^2=\frac1n\sum_{j=1}^n (\Tr \mathbf R-\Tr\mathbf R^{(j)})R_{jj}=
\frac{s_y(z)}n\sum_{j=1}^n(\Tr \mathbf R-\Tr\mathbf R^{(j)}) +D_1,
\end{equation}
where
\begin{align}\label{d1}
D_1&=\frac1n\sum_{j=1}^n (\Tr \mathbf R-\Tr\mathbf R^{(j)})(R_{jj}-s(z))\notag\\&=
\sum_{j=1}^n(\varepsilon_{j3}+\frac1{pz})(R_{jj}-s_y(z)).
\end{align}
Using equality \eqref{shur}, we may write
\begin{align}
m_n'(z)&=\frac{s_y(z)}n\sum_{j=1}^n\E (1+\frac1p\sum_{l,k=1}^pX_{jl}X_{jk}[(\mathbf R^{(j)})^2]_{l+n,k+n})R_{jj}+D_1\notag\\&=
\frac{s_y^2(z)}n\sum_{j=1}^n \E(1+\frac1p\sum_{l,k=1}^pX_{jl}X_{jk}[(\mathbf R^{(j)})^2]_{l+n,k+n})+D_1+D_2,\notag
\end{align}
where
\begin{equation}\notag
 D_2=\frac{s_y(z)}n\sum_{j=1}^n\E (1+\frac1p\sum_{l,k=1}^pX_{jl}X_{jk}[(\mathbf R^{(j)})^2]_{l+n,k+n})(R_{jj}-s_y(z))
\end{equation}

Denote by
\begin{align}\beta_{j1}&=\frac1p\sum_{l=1}^p\E[(R^{(j)})^2]_{l+n,l+n}-\frac1p\sum_{l=1}^p[(R)^2]_{l+n,l+n}
=\frac1p\sum_{l\in\mathbb T_j}\E[(R^{(j)})^2]_{ll}-ym_n'(z)\notag\\&=\frac1p\frac d{dz}(\Tr \mathbf R-\Tr\mathbf R^{(j)}),\notag\\
\beta_{j2}&=\frac1p\sum_{l=1}^p\E(X^2_{jl}-1)[(R^{(j)})^2]_{l+n,l+n},\notag\\
\beta_{j3}&=\frac1p\sum_{1\le l\ne k\le p}\E X_{jl}X_{jk}[(R^{(j)})^2]_{l+n,k+n}.\notag
\end{align}
Using  these notation we may write
\begin{align}\notag
m_n'(z)=s_y^2(z)(1+ym_n'(z)-\frac{1-y}{z^2})+\frac{s_y^2(z)}n\sum_{j=1}^n (\beta_{j1}+\beta_{j2}+\beta_{j3})+D_1+D_2.
\end{align}
Solving this equation with respect to $m_n'(z)$ we obtain
\begin{equation}\label{semi1}
m_n'(z)=\frac{s_y^2(z)(1-\frac{1-y}{z^2})}{1-ys_y^2(z)}+\frac{1}{1-s_y^2(z)}(D_3+D_1+D_2),
\end{equation}
where
\begin{align}\notag
D_3=\frac{s_y^2(z)}n\sum_{j=1}^n (\beta_{j1}+\beta_{j2}+\beta_{j3}).
\end{align}
Note that for the Marchenko -- Pastur  law
\begin{equation}\notag
\frac{s_y^2(z)(1-\frac{1-y}{z^2})}{1-ys_y^2(z)}=-\frac{s_y(z)(1-\frac{1-y}{z^2})}{z+2ys_y(z)+\frac{y-1}z}=
s_y'(z).
\end{equation}
Applying this relation we rewrite equality \eqref{semi1} as
\begin{equation}\label{mnderiv}
m_n'(z)-s'(z)=-\frac{1}{s_y(z)(z+2s_y(z)+\frac{y-1}z)}(D_1+D_2+D_3).
\end{equation}
Using the last equality, we may represent $\mathfrak T_{42}$ now as follows
\begin{equation}\notag
\mathfrak T_{42}=\mathfrak T_{421}+\mathfrak T_{422}+\mathfrak T_{423},
\end{equation}
where
\begin{align}
\mathfrak T_{421}&=\frac1n\E \frac{D_1}{z+y(m_n(z)+s_y(z))+\frac{y-1}z},\notag\\
\mathfrak T_{422}&=\frac1n\E \frac{D_2}{z+y(m_n(z)+s_y(z))+\frac{y-1}z},\notag\\
\mathfrak T_{423}&=\frac1n\E \frac{D_3}{z+y(m_n(z)+s_y(z))+\frac{y-1}z}.\notag
\end{align}



Recall that, by \eqref{d1},
\begin{align}\mathfrak T_{421}&=\frac1{n}\sum_{j=1}^n
\E\frac{\varepsilon_{j3}(R_{jj}-s(z))}{(z+y(s(z)+m_n(z))+\frac{y-1}z)}.
\end{align}
Applying the Cauchy -- Schwartz inequality, we get for $z\in\mathbb G$,
\begin{equation}\notag
|\mathfrak T_{421}|\le \frac1{n}\sum_{j=1}^n\E^{\frac12}|R_{jj}-s_y(z)|^2
\E^{\frac12}\frac{|\varepsilon_{j3}|^2}{|z+y(s(z)+m_n(z))+\frac{y-1}z|^2}.
\end{equation}
Using Corollary \ref{corgot}, inequality \eqref{raz*} and Corollary \ref{cor8}, we get
\begin{align}\label{finisch7*}
|\mathfrak T_{421}|\le\frac C{n^{\frac32}v^{\frac32}}.
\end{align}
\subsubsection{Estimation of $\mathfrak T_{423}$}
We represent now $\mathfrak T_{423}$ in the form
\begin{align}\label{t51}
\mathfrak T_{423}=\mathfrak T_{51}+\mathfrak T_{52}+\mathfrak T_{53},
\end{align}
where
\begin{align}\notag
\mathfrak T_{5\nu}&=\frac1{n^2}\sum_{j=1}^n\E\frac{\beta_{j\nu}}{z+y(m_n(z)+s(z))+\frac{y-1}z},\quad\text{for}\quad \nu=1,2,3.
\end{align}

We consider the quantity $\mathfrak T_{5\nu}$, for $\nu=1,2,3$.
Applying the Cauchy-Schwartz inequality and inequality \eqref{lar1} in the Appendix as well, 
we get
\begin{align}\notag
|\mathfrak T_{5\nu}|\le \frac C{n^2}\sum_{j=1}^n\E^{\frac12}\frac{|\beta_{j\nu}|^2}{|z+y(m_n^{(j)}(z)+s_y(z))+\frac{y-1}z|^2}.
\end{align}
By Lemma \ref{bet1*} together with Lemma \ref{lem00} in the Appendix, we obtain
\begin{align}\notag
\E^{\frac12}\frac{|\beta_{j\nu}|^2}{|z+y(m_n^{(j)}(z)+s(z))+\frac{y-1}z|^2}\le \frac C{n^{\frac12}v^{\frac32}|(z+\frac{y-1}z)^2-4y|^{\frac14}}.
\end{align}
This implies that
\begin{align}\label{finisch7}
|\mathfrak T_{5\nu}|\le \frac C{n^{\frac32}v^{\frac32}|(z+\frac{y-1}z)^2-4y|^{\frac14}}.
\end{align}
Equality \eqref{t51} and inequality \eqref{finisch7} yield
\begin{equation}\label{finisch7+}
 |\mathfrak T_{423}|\le \frac C{n^{\frac32}v^{\frac32}|(z+\frac{y-1}z)^2-4y|^{\frac14}}.
\end{equation}

\subsubsection{Estimation of $\mathfrak T_{422}$} By the definitions of $\mathfrak T_{423}$ and $D_2$, we have
\begin{align}
 \mathfrak T_{422}&=\frac{s_y(z)}{n^2}\sum_{j=1}^n\E \frac{ (1+\frac1p\sum_{l,k=1}^pX_{jl}X_{jk}[(\mathbf R^{(j)})^2]_{l+n,k+n})(R_{jj}-s_y(z))}
 {z+y(m_n(z)+s_y(z))+\frac{y-1}z}.\notag
\end{align}
We have
\begin{equation}\notag
 1+\frac1p\sum_{l,k=1}^pX_{jl}X_{jk}[(\mathbf R^{(j)})^2]_{l+n,k+n}=1+ym_n'(z)+\beta_{j1}+\beta_{j2}+\beta_{j3}.
\end{equation}
This implies
\begin{equation}\notag
 \mathfrak T_{422}=\mathfrak T_{60}+\cdots+\mathfrak T_{63},
\end{equation}
where
\begin{align}
 \mathfrak T_{60}&=\frac{s_y(z)}{n}\E\frac{\Lambda_n(1+ym_n'(z))}{z+y(m_n(z)+s_y(z))+\frac{y-1}z},\notag\\
 \mathfrak T_{6\nu}&=\frac{s_y(z)}{n^2}\sum_{j=1}^n\E\frac{\beta_{j\nu}(R_{jj}-s_y(z))}{z+y(m_n(z)+s_y(z))+\frac{y-1}z}, \text{ for }\nu=1,2,3.\notag
\end{align}
Applying \tc{the Cauchy -- Schwartz} inequality, we get
\begin{equation}\notag
 |\mathfrak T_{60}|\le \frac C{n\sqrt y}\E^{\frac12}|\Lambda_n|^2\E^{\frac12}\frac{|1+m_n'(z)|^2}{|z+y(m_n(z)+s_y(z))+\frac{y-1}z|^2}
\end{equation}
By inequality $|m_n'(z)|\le v^{-1}\im m_n(z)+\frac{1-y}{|z|^2}$ and Lemma \ref{lem00}, we get
\begin{equation}\notag
 |\mathfrak T_{60}|\le \frac C{nv\sqrt y}\E^{\frac12}|\Lambda_n|^2.
\end{equation}
Applying Lemma \ref{lam1*} below, we get
\begin{equation}\notag
 |\mathfrak T_{60}|\le \frac C{n^{\frac32}v^{\frac32}}.
\end{equation}
Similar to inequality \eqref{finisch7+}, we get, for $\nu=1,2,3$,
\begin{equation}\label{finisch7++}
 |\mathfrak T_{6\nu}|\le \frac C{n^2v^2|(z+\frac{y-1}z)^2-4y|^{\frac14}}.
\end{equation}

Combining \eqref{finisch7*}, \eqref{finisch7+} and \eqref{finisch7++}, we get, for $z\in\mathbb G$,
\begin{align}\label{finisch7^}
 |\mathfrak T_{42}|\le \frac C{n^{\frac32}v^{\frac32}|(z+\frac{y-1}z)^2-4y|^{\frac14}}.
\end{align}




\subsubsection{Estimation of $\mathfrak T_{43}$} Recall that
\begin{equation}\notag
\mathfrak T_{43}=\frac{s(z)}n\E\frac{(m_n'(z)-s_y'(z))\Lambda_n}{z+y(m_n(z)+s_y(z))+\frac{y-1}z}.
\end{equation}
Applying equality \eqref{mnderiv}, we obtain
\begin{align}\notag
 \mathfrak T_{43}=\mathfrak T_{431}+\mathfrak T_{432}+\mathfrak T_{433},
\end{align}
where
\begin{align}
 \mathfrak T_{431}&=\frac{1}{2n(z+2s_y(z)+\frac{y-1}z)}\E\frac{D_1\Lambda_n}{z+y(m_n(z)+s_y(z))+\frac{y-1}z},\notag\\
 \mathfrak T_{432}&=\frac{1}{2n(z+2s_y(z)+\frac{y-1}z)}\E\frac{D_2\Lambda_n}{z+y(m_n(z)+s_y(z))+\frac{y-1}z},\notag\\
 \mathfrak T_{433}&=\frac{1}{2n(z+2s(z)+\frac{y-1}z)}\E\frac{D_3\Lambda_n}{z+y(m_n(z)+s_y(z))+\frac{y-1}z}.\notag
\end{align}
Applying the Cauchy -- Schwartz inequality, we get
\begin{align}\notag
|\mathfrak T_{431}|\le \frac{1}{n(z+2s_y(z)+\frac{y-1}z)}\E^{\frac12}\frac{|D_1|^2}{|z+y(m_n(z)+s_y(z))+\frac{y-1}z|^2}\E^{\frac12}|\Lambda_n|^2.
\end{align}
By definition of $D_1$ and Lemma \ref{lam1*} , we get
\begin{align}
|\mathfrak T_{431}|
&\le
\frac{C}{n^{2} v|(z+\frac{y-1}z)^2-4y|^{\frac12}}\sum_{j=1}^n\E^{\frac14}\frac{|\varepsilon_{j3}+\frac1{pz}|^4}{|z+y(m_n(z)+s(z))+\frac{y-1}z|^4}\notag\\
&\qquad\qquad\qquad\times\E^{\frac14}|R_{jj}-s(z)|^4.\notag
\end{align}
Applying now Corollary \ref{cor8} and Lemma \ref{lem14}, we get
\begin{align}
|\mathfrak T_{431}|
 \le
\frac{4}{n^{2} v^{2}|(z+\frac{y-1}z)^2-4y|^{\frac12}}.\notag
\end{align}
For $z\in\mathbb G$ this yields
\begin{equation}\notag
|\mathfrak T_{431}|
 \le
\frac{4}{n^{\frac32} v^{\frac32}|(z+\frac{y-1}z)^2-4y|^{\frac14}}.
\end{equation}
Applying again the Cauchy -- Schwartz inequality, we get for $\mathfrak T_{432}$ accordingly
\begin{align}\notag
 |\mathfrak T_{432}|\le \frac C{n|(z+\frac{y-1}z)^2-4y|^{\frac12}}\E^{\frac12}|D_2|^2\E^{\frac12}|\Lambda_n|^2.
\end{align}
By Lemma \ref{lam1*}, we have
\begin{equation}\label{krak}
 |\mathfrak T_{432}|\le \frac C{n^2v|(z+\frac{y-1}z)^2-4y|^{\frac12}}\E^{\frac12}|D_2|^2.
\end{equation}
By definition of $D_2$,
\begin{equation}\notag
 \E|D_2|^2\le \frac1n\sum_{j=1}^n(\E|\beta_{j1}|^2+\E|\beta_{j2}|^2+\E|\beta_{j3}|^2).
\end{equation}
Applying Lemmas \ref{bet1*} with $\nu=2,3$, and \ref{beta}, we get
\begin{align}\label{krak1}
 \E|D_2|^2\le \frac{C}{n^2v^4}+\frac C{nv^3}.
\end{align}
Inequalities \eqref{krak} and \eqref{krak1} together imply, for $z\in\mathbb G$,
\begin{align}
 |\mathfrak T_{432}|&\le \frac C{n^3v^3|(z+\frac{y-1}z)^2-4y|^{\frac12}}+\frac C{n^{\frac52}v^{\frac32}|(z+\frac{y-1}z)^2-4y|^{\frac12}}\notag\\
 &\le \frac C{n^{\frac32}v^{\frac32}|(z+\frac{y-1}z)^2-4y|^{\frac14}}.\notag
\end{align}
\subsubsection{Estimation of $\mathfrak T_{44}, \mathfrak T_{45}$}
Note that
\begin{align}
\mathfrak T_{44}&=\frac{ys_y(z)\sqrt{(z+\frac{y-1}z)^2-4}}{2nz}\E\frac{\Lambda_n}{z+y(m_n(z)+s_y(z))+\frac{y-1}z},\notag\\
\mathfrak T_{45}&=\frac{ys_y(z)}{2nz}\E\frac{\Lambda_n^2}{z+y(m_n(z)+s_y(z))+\frac{y-1}z}.\notag
\end{align}
Applying Lemma \ref{lem00}, inequality \eqref{lem00.2}, we get
\begin{align}\notag
 |\mathfrak T_{44}|\le \frac{C}{2n|z|}\E|\Lambda_n|.
\end{align}
Therefore, by Lemma \ref{lam1*}, we have
\begin{equation}\notag
 |\mathfrak T_{44}|\le \frac{C}{2n^2v^2}.
\end{equation}
Using Lemma \ref{lem00}, inequality \eqref{lem00.2},  Lemma \ref{lam1*},  below, we get
\begin{equation}\notag
 |\mathfrak T_{45}|\le \frac{C}{n^3v^2|z\sqrt{(z+\frac{y-1}z)^2-4y}|}.
\end{equation}
Applying now inequality \eqref{cy} below, we get
\begin{equation}\notag
 |\mathfrak T_{45}|\le \frac{C(y)}{n^3v^{\frac52}}.
\end{equation}
\subsubsection{Estimation of $\mathfrak T_{41}$}

Finally we observe that
\begin{equation}\notag
 s_y'(z)=-\frac{s_y(z)(1-\frac{y-1}{z^2})}{\sqrt{(z+\frac{y-1}z)^2-4y}}
\end{equation}
and
\begin{align}
 s'_y(z)-\frac{s_y(z)}z=-\frac{2s_y(z)(z+ys_y(z))}{z\sqrt{(z+\frac{y-1}z)^2-4y}}.\notag
\end{align}
Therefore
\begin{equation}\notag
 |\mathfrak T_{41}|\le \frac C{n|(z^2+y-1)^2-4yz^2|^{\frac12}}. 
\end{equation}
  We have
\begin{align}\notag
 (z^2+y-1)^2-4yz^2=(z+\sqrt y-1)(z+1\sqrt y-1)(z-\sqrt y+1)(z-\sqrt y-1).
\end{align}
For $z\in\mathbb G$ we get
\begin{equation}\label{cy}
 |(z^2+y-1)^2-4yz^2|^{\frac12}\ge C\sqrt{1-\sqrt y}\sqrt v.
\end{equation}

We may rewrite now
\begin{equation}\label{final001}
 |\mathfrak T_{41}|\le \frac {C(y)}{n\sqrt v},
\end{equation}
where
\begin{equation}\notag
 C(y)=\begin{cases}&C, \text{ if }y=1,\\&\frac{C}{\sqrt{1-\sqrt y}}, \text{ if }y<1.\end{cases}
\end{equation}

Combining now relations \eqref{final00}, \eqref{h1}, \eqref{finisch4}, \eqref{t51}, \eqref{finisch7^}, \eqref{final001}, we get for $z\in\mathbb G$,
\begin{equation}
 |\E\Lambda_n|\le \frac C{n v^{\frac34}}+\frac C{n^{\frac32}v^{\frac32}|(z+\frac{y-1}z)^2-4y|^{\frac14}}.
\end{equation}
The last inequality completes the proof of Theorem \ref{stiltjesmain}.


\section{Appendix}
\subsection{Rosenthal's and Burkholder's Inequalities}
In this subsection we state the Rosenthal and Burkholder inequalities starting with Rosenthal's inequality.
Let $\xi_1,\ldots,\xi_n$ be independent random variables with $\E\xi_j=0$, $\E\xi_j^2=1$ and for $p\ge 1$ $\E|\xi_j|^p\le \mu_p$ for $j=1,\ldots,n$.
\begin{lem}\label{Rosent}{\rm (Rosenthal's inequality)}

 There exists an absolute constant $C_1$ such that
 \begin{equation}\notag
 \E|\sum_{j=1}^na_j\xi_j|^q\le C_1^qq^q\Big(\big(\sum_{j=1}^n|a_j|^2\big)^{\frac q2}+\mu_p\sum_{j=1}^n|a_j|^q\Big)
 \end{equation}
\end{lem}
\begin{proof}
 For the proof of this inequality see \cite{Rosenthal:1970} and \cite{Johnson:1985}.
\end{proof}
Let $\xi_1,\ldots\xi_n$ be martingale-difference with respect to $\sigma$-algebras $\mathfrak M_j=\sigma(\xi_1,\ldots,\xi_{j-1})$.
Assume that  $\E\xi_j^2=1$ and $\E|\xi_j|^q<\infty$, for $q\ge2$.
\begin{lem}\label{burkh}{\rm (Burkholder's inequality)}
 There exist an absolute constant $C_2$ such that
 \begin{equation}\notag
  \E|\sum_{j=1}^n\xi_j|^q\le C_2^qq^q\Big(\Big(\E(\sum_{k=1}^n\E\{\xi_k^2|\mathfrak M_{k-1}\}\Big)^{\frac q2}+\sum_{k=1}^p\E|\xi_k|^q\Big).
 \end{equation}

\end{lem}
\begin{proof}
 For the proof of this inequality see \cite{Burkholder:1973} and \cite{Hitczenko:1990}.
\end{proof}
We rewrite the Burkholder inequality for quadratic forms in independent random variables.
Let $\zeta_1,\ldots,\zeta_n$ be independent random variables such that $\E\zeta_j=0$, $\E|\eta_j|^2=1$ and $\E|\zeta_j|^q\le \mu_q$. 
Let $a_{ij}=a_{ji}$ for all $i,j=1,\ldots n$.
Consider the quadratic form
\begin{equation}\notag
 Q=\sum_{1\le j\ne k\le n}a_{jk}\zeta_j\zeta_k.
\end{equation}
\begin{lem}\label{burkh1}
 There exists an absolute constant $C_2$ such that
 \begin{equation}\notag
  \E|Q|^q\le C_2^q\Big(\E\big(\sum_{j=2}^{n}(\sum_{k=1}^{j-1}a_{jk}\zeta_k)^2\big)^{\frac q2}+\mu_q\sum_{j=2}^n\E|\sum_{k=1}^{j-1}a_{jk}\zeta_k|^q\Big).
 \end{equation}

\end{lem}
\begin{proof}
 Introduce the random variables 
 \begin{equation}\notag
  \xi_j=\zeta_j\sum_{k=1}^{j-1}a_{jk}\zeta_k, \q j=2,\ldots,n.
 \end{equation}
 It is straightforward to check that
 \begin{equation}\notag
  \E\{\xi_j|\mathfrak M_{j-1}\}=0,
 \end{equation}
 and that $\xi_j$ are $\mathfrak M_{j}$ measurable.
 Hence $\xi_1,\ldots,\xi_n$ are martingale-differences.
 We may write
 \begin{equation}\notag
  Q=2\sum_{j=2}^n\xi_j
 \end{equation}
Applying now Lemma \ref{burkh} and using
\begin{align}
 \E\{|\xi_j|^2|\mathfrak M_{j-1}\}&=(\sum_{k=1}^{j-1}a_{jk}\eta_k)^2\E\zeta_j^2,\notag\\
 \E|\xi_j|^q&=\E|\eta_j|^q\E|\sum_{k=1}^{j-1}a_{jk}\zeta_j|^q,\notag
\end{align}
we get the claim.
Thus, Lemma \ref{burkh1} is proved.

\end{proof}

\subsection{Auxiliary Inequalities for Resolvent Matrices} We shall use the following relation between resolvent matrices. Let $\mathbb A$ and $\mathbb B$ be two Hermitian matrices and 
let $\mathbf R_{\mathbf A}=(\mathbb A-z\mathbf I)^{-1}$ and $\mathbf R_{\mathbf B}=(\mathbb B-z\mathbf I)^{-1}$ denote their resolvent matrices.
Recall the resolvent equality
\begin{equation}\label{reseq}
 \mathbf R_{\mathbf A}-\mathbf R_{\mathbf B}=\mathbf R_{\mathbf A}(\mathbf B-\mathbf A)\mathbf R_{\mathbf B}=
 -\mathbf R_{\mathbf B}(\mathbf B-\mathbf A)\mathbf R_{\mathbf A}.
\end{equation}
Recall the equation, for $j\in\mathbb T_{\mathbb J}$, and $\mathbb J\subset\mathbb T$ (compare with \eqref{repr001})
\begin{equation}
 R_{jj}^{(\mathbb J)}=-\frac1{z+ym_n^{(\mathbb J)}(z)+\frac{n-p+|\mathbb J|}{nz}}+\frac1{z+ym_n^{(\mathbb J)}(z)+\frac{n-p+|\mathbb J|}{nz}}
 \varepsilon_j^{(\mathbb J)}R^{(\mathbb J)}_{jj},
\end{equation}
where $\varepsilon_j^{(\mathbb J)}=\varepsilon_{j1}^{(\mathbb J)}+\varepsilon_{j2}^{(\mathbb J)}+\varepsilon_{j3}^{(\mathbb J)}$ and 
\begin{align}
\varepsilon_{j1}^{(\mathbb J)}&=\frac1n\sum_{l=1}^p(X_{jl}^2-1)R^{(\mathbb J,j)}_{l+n,l+n},\notag\\
\varepsilon_{j2}^{(\mathbb J)}&=\frac1n\sum_{1\le l\ne k\le p}X_{jl}X_{jk}R^{(\mathbb J,j)}_{k+n,l+n},\notag\\
\varepsilon_{j3}^{(\mathbb J)}&=m_n^{(\mathbb J)}(z)-m_n^{(\mathbb J,j)}(z).\notag
\end{align}
Summing these equations for $j\in\mathbb T_{\mathbb J}$, we get
\begin{equation}\label{main1}
m_n^{(\mathbb J)}(z)=-\frac{n-|\mathbb J|}{n(z+ym_n^{(\mathbb J))}(z)+\frac{n-p+|\mathbb J|}{pz})}
+\frac{T_n^{(\mathbb J)}}{z+ym_n^{(\mathbb J)}(z)+\frac{n-p+|\mathbb J|}{pz}},
\end{equation}
where 
\begin{equation}\notag
 T_n^{(\mathbb J)}=\frac1n\sum_{j\in\mathbb T_{\mathbb J}}\varepsilon_j^{(\mathbb J)}R_{jj}^{(\mathbb J)}.
\end{equation}
Note that
\begin{align}\label{main2}
 \frac1{z+ym_n^{(\mathbb J)}(z)+\frac{n-p+|\mathbb J|}{pz}}&=\frac1{z+ys_y(z)+\frac{p-n+|\mathbb z}{pz}}\notag\\&
 -\frac{m_n^{(\mathbb J)}(z)-s_y(z)}{(z+ys_y(z)+\frac{n-p+|\mathbb J|}{pz})(z+ym_n^{(\mathbb J)}(z)+\frac{n-p+|\mathbb J|}{pz})}\notag\\&=
 -s_y(z)+\frac{s_y(z)\Lambda_n^{(\mathbb J)}(z)}{z+ym_n^{(\mathbb J)}(z)+\frac{n-p+|\mathbb J|}{pz}},
\end{align}
where 
\begin{equation}\notag
 \Lambda_n^{(\mathbb J)}=\Lambda_n^{(\mathbb J)}(z)=m_n^{(\mathbb J)}(z)-s_y(z).
\end{equation}
Equalities \eqref{main1} and \eqref{main2} together imply 
\begin{align}
 \Lambda_n^{(\mathbb J)}&=-\frac{s_y(z)\Lambda_n^{(\mathbb J)}}{z+ym_n^{(\mathbb J)}(z)+\frac{n-p+|\mathbb J|}{pz}}
\notag\\& +\frac{T_n^{(\mathbb J)}}{z+ym_n^{(\mathbb J)}(z)+\frac{n-p+|\mathbb J|}{pz}}+\frac{|\mathbb J|}{n(z+ym_n^{(\mathbb J)}(z)+\frac{n-p+|\mathbb J|}{pz})}.
\notag
\end{align}
Solving this with respect to $\Lambda_n^{(\mathbb J)}$, we get
\begin{equation}\label{main3}
 \Lambda_n^{(\mathbb J)}=\frac{T_n^{(\mathbb J)}}{z+y(m_n^{(\mathbb J)}(z)+s_y(z))+\frac{n-p+|\mathbb J|}{pz}}+\frac{|\mathbb J|}{n(z+y(m_n^{(\mathbb J)}(z)+s_y(z))
 +\frac{n-p+|\mathbb J|}{pz})}.
\end{equation}

\begin{lem}\label{resol00}
For any $z=u+iv$ with $v>0$ and for any $\mathbb J\subset \mathbb T$, we have
\begin{align}\label{res1}
 \frac1n\sum_{l=1}^p\sum_{k=1}^p|R^{(\mathbb J)}_{l+n,k+n}|^2\le v^{-1}\im m_n^{(\mathbb J)}(z)+\frac{1-y}{|z|^2},
\end{align}
and, for any $\mathbb J\subset \mathbb L=\{1,\ldots,p\}$
\begin{align}\label{res1a}
 \frac1n\sum_{l=1}^n\sum_{k=1}^n|R^{(\mathbb J)}_{l,k}|^2\le v^{-1}\im m_n^{(\mathbb J)}(z)+\frac{1-y}{|z|^2},
\end{align}
For any  $\mathbb J\subset\mathbb T$, $l=1,\ldots,p$,
\begin{equation}\label{res2}
 \sum_{k=1}^p|R^{(\mathbb J)}_{l+n,k+n}|^2\le v^{-1}\im R^{(\mathbb J)}_{l+n,l+n}.
\end{equation}
and, for any $\mathbb J\subset \mathbb L=\{1,\ldots,p\}$, $l=1,\ldots,n$,
\begin{equation}\label{res2a}
 \sum_{k=1}^n|R^{(\mathbb J)}_{l,k}|^2\le v^{-1}\im R^{(\mathbb J)}_{l,l}.
\end{equation}
For any  $\mathbb J\subset\mathbb T$, $l=1,\ldots,p$,
\begin{equation}\label{res20}
 \sum_{k=1}^p|[(\mathbf R^{(\mathbb J)}])^2_{l+n,k+n}|^2\le Cv^{-3}\im R^{(\mathbb J)}_{l+n,l+n}.
\end{equation}
and, for any $\mathbb J\subset \mathbb L=\{1,\ldots,p\}$, $l=1,\ldots,n$,
\begin{equation}\label{res20a}
 \sum_{k=1}^n|[(\mathbf R^{(\mathbb J)})^2]_{l,k}|^2\le Cv^{-3}\im R^{(\mathbb J)}_{l,l}.
\end{equation}
Moreover, for any $\mathbb J\subset T$ and for any $l\in\mathbb T_{\mathbb J}$ we have
\begin{align}\label{res3}
\frac1n\sum_{l=1}^p|[(\mathbf R^{(\mathbb J)})^2]_{l+n,l+n}|^2\le v^{-3}\im m_n^{(\mathbb J)}(z),
\end{align}
and, for any $\mathbb J\subset \mathbb L=\{1,\ldots,p\}$,
\begin{align}\label{res3a}
\frac1n\sum_{l=1}^n|[(R^{(\mathbb J)})^2]_{l,l}|^2\le v^{-3}\im m_n^{(\mathbb J)}(z).
\end{align}

For any $q\ge1$,
\begin{align}\label{res4}
\frac1n\sum_{l=1}^p|[(R^{(\mathbb J)})^2]_{l+n,l+n}|^q\le v^{-q}\frac1n\sum_{l=1}^p\im^qR^{(\mathbb J)}_{l+n,l+n},
\end{align}
and
\begin{align}\label{res4a}
\frac1n\sum_{l=1}^n|[(R^{(\mathbb J)})^2]_{l,l}|^q\le v^{-q}\frac1n\sum_{l=1}^n\im^qR^{(\mathbb J)}_{l,l}.
\end{align}

Finally,
\begin{align}\label{res5}
\frac1n\sum_{l,k=1}^p|[(R^{(\mathbb J)})^2]_{l+n,k+n}|^2\le v^{-3}\im m_n^{(\mathbb J)}(z),
\end{align}
and
\begin{align}\label{res6}
\frac1n\sum_{l,k=1}^p|[(R^{(\mathbb J)})^2]_{l+n,k+n}|^{2q}\le v^{-3q}\frac1n\sum_{l=1}^p\im^qR^{(\mathbb J)}_{l+n,l+n},
\end{align}
We have as well
\begin{align}\label{res7}
\frac1{n^2}\sum_{l,k=1}^p|[(R^{(\mathbb J)})^2]_{l+n,k+n}|^{2p}\le v^{-2q}(\frac1n\sum_{l\in\mathbb T_{\mathbb J}}\im^pR^{(\mathbb J)}_{l+n,l+n})^2.
\end{align}
\end{lem}
\begin{proof}Firstly we note that
\begin{align} 
 \max\{\frac1n\sum_{l=1}^p\sum_{k=1}^p|R^{(\mathbb J)}_{l+n,k+n}|^2,\frac1n\sum_{l=1}^n\sum_{k=1}^n|R^{(\mathbb J)}_{l,k}|^2\}\le \frac1n\Tr|\mathbf R ^{(\mathbb J)}|^2.
\end{align}
Furthermore,
\begin{equation}\notag
 \frac1n\Tr|R^{(\mathbb J)}|^2=v^{-1}\frac1n\im\Tr \mathbf R^{(\mathbf J)}=2v^{-1}\im m_n^{(\mathbb J)}(z)+\frac{1-y-\frac{|\mathbb J|}n}{|z|^2}
\end{equation}
These relations imply inequalities \eqref{res1} and \eqref{res1a}.
Note that
\begin{align}
 \max\{\sum_{k=1}^p|R^{(\mathbb J)}_{l+n,k+n}|^2,\sum_{k=1}^n|R^{(\mathbb J)}_{l,k}|^2\}\le {\sum}^*_{k}|R^{(\mathbb J)}_{l,k}|^2.\notag
\end{align}
Here we denote by ${\sum^*}_{k}$ the sum over $k\in\mathbb T_{\mathbb J}\cup\{1,\ldots,p\}$ if $\mathbb J\subset \mathbb T$, and 
the sum over $k\in\{1,\ldots,n\}\cup\{\mathbb L\setminus\mathbb J\}$.
Let us denote by $\lambda^{(\mathbb J)}_k$  the eigenvalues of the matrix $\mathbf W^{(\mathbb J)}$.
Let denote now by $\mathbf u_k^{(\mathbb J)}=(u^{(\mathbb J)}_{kl})$ the eigenvector of the matrix $\mathbf W^{(\mathbb J)}$ 
 corresponding to the eigenvalue $\lambda^{(\mathbb J)}_k$.
Using  this notation we may write
 \begin{equation}\label{orth1}
  R^{(\mathbb J)}_{lk}={\sum}^*_{q}\frac1{\lambda_q^{(\mathbb J)}-z}u^{(\mathbb J)}_{lq}u^{(\mathbb J)}_{kq}.
 \end{equation}
It is straightforward to check that the following inequality holds
\begin{align}
 {\sum}^*_{k}|R^{(\mathbb J)}_{kl}|^2&\le{\sum}^*_{q}\frac1{|\lambda^{(\mathbb J)}_q-z|^2}|u^{(\mathbb J)}_{lq}|^2\notag\\&=
 v^{-1}\im\Big({\sum}^*_{q}\frac1{\lambda^{(\mathbb J)}_q-z}|u^{(\mathbb J)}_{lq}|^2\Big)=v^{-1}\im R_{ll}^{(\mathbb J)}.\notag
\end{align}
Thus, inequalities \eqref{res2} and \eqref{res2a} are proved.
Similarly we get
\begin{equation}\notag
 {\sum}^*_{k}|[(R^{(\mathbb J)})^2]_{kl}|^2\le{\sum}^*_{}\frac1{|\lambda^{(\mathbb J)}_q-z|^4}|u^{(\mathbb J)}_{lq}|^2\le
 v^{-3}\im R^{(\mathbb J)}_{ll}.
\end{equation}
This implies inequalities \eqref{res20} and \eqref{res20a}.   
To prove  inequality \eqref{res3} and \eqref{res3a} we observe that
\begin{equation}\label{resol1} 
 |[(\mathbf R^{(\mathbb J)})^2]_{ll}|\le {\sum}^*_{k}|R^{(\mathbb J)}_{lk}|^2.
\end{equation}
This inequality implies
\begin{equation}\notag
 \frac1n{\sum}^*_{l}|[(\mathbf R^{(\mathbb J)})^2]_{ll}|^2\le \frac1n{\sum}^*_{l}
 ({\sum}^*_{k}|R^{(\mathbb J)}_{lk}|^2)^2.
\end{equation}
Applying now inequality \eqref{res2}, we get
\begin{align}
\frac1n{\sum}^*_{l}|[(\mathbf R^{(\mathbb J)})^2]_{ll}|^2\le v^{-2}\frac1n{\sum}^*_{l}\im^2R^{(\mathbb J)}_{ll}.\notag
\end{align}
Using  $|R^{(\mathbb J)}_{ll}|\le v^{-1}$ this leads to the following bound
\begin{align}
\frac1n{\sum}^*_{l}|[(\mathbf R^{(\mathbb J)})^2]_{ll}|^2\le v^{-3}\frac1n{\sum}^*_{l}\im R^{(\mathbb J)}_{ll}=
v^{-3}\im m_n^{(\mathbb J)}(z).
\end{align}
Thus inequalities \eqref{res3} and \eqref{res3a} are proved.
Furthermore, applying inequality \eqref{resol1}, we may write
\begin{align}
\frac1n{\sum}^*_{l}|[(\mathbf R^{(\mathbb J)})^2]_{ll}|^4\le\frac1n{\sum}^*_{l}
 ({\sum}^*_{k}|R^{(\mathbb J)}_{lk}|^2)^4.\notag
\end{align}
Applying \eqref{res2}, this inequality yields
\begin{align}
\frac1n\sum_{l}|[(R^{(\mathbb J)})^2]_{ll}|^4\le v^{-4}\frac1n{\sum}^*_{l}
 \im^4R^{(\mathbb J)}_{ll}.\notag
\end{align}
The last inequality proves  inequality \eqref{res4}. Similarly we get inequality \eqref{res4a}.
Note that
\begin{align}
\frac1n{\sum}^*_{l,k}|[(\mathbf R^{(\mathbb J)})^2]_{lk}|^2&\le \frac1n\Tr|\mathbf R^{(\mathbb J)}|^4=\frac1n{\sum}^*_{l}
\frac1{|\lambda^{(\mathbb J)}_l-z|^4}\notag\\&\le v^{-3}\im\frac1n{\sum}^*_{l}\frac1{\lambda^{(\mathbb J)}_l-z}\le
v^{-3}(\im m_n^{(\mathbb J)}(z)+\frac{1-y-\frac{v|\mathbb J|}n}{|z|^2}).\notag
\end{align}
Thus, inequality \eqref{res5} is proved.
To finish we note that
\begin{align}
 \frac1n{\sum}^*_{l,k}|[(\mathbf R^{(\mathbb J)})^2]_{lk}|^4\le \frac1n{\sum}^*_{l}
 ({\sum}^*_{k}|[(\mathbf R^{(\mathbb J)})^2]_{lk}|^2)^2.\notag
\end{align}
Applying inequality \eqref{res20}, we get
\begin{align}
 \frac1n{\sum}^*_{l,k}|[(\mathbf R^{(\mathbb J)})^2]_{lk}|^4\le v^{-6}\frac1n{\sum}^*_{l}
 (\im R^{(\mathbb J)}_{ll})^2.\notag
\end{align}
To prove inequality \eqref{res7}, we note 
\begin{equation}\notag
|[(\mathbf R^{(\mathbb J)})^2]_{lk}|^2\le({\sum}^*_{q}|R^{(\mathbb J)}_{lq}|^2)({\sum}^*_{q}|R^{(\mathbb J)}_{kq}|^2). 
\end{equation}
This inequality implies
\begin{align}
 \frac1{n^2}{\sum}^*_{l,k}|[(\mathbf R^{(\mathbb J)})^2]_{lk}|^{2p}\le 
 (\frac1n{\sum}^*_{l,k\in}({\sum}^*_{q}|R^{(\mathbb J)}_{lq}|^2)^p)^2
 (\im R^{(\mathbb J)}_{ll})^2.\notag
\end{align}
Applying inequality \eqref{res1}, we get the claim.
Thus, Lemma \ref{resol00} is proved.
\end{proof}
\begin{lem}\label{schlein}For any $s\ge 1$, and for any $z=u+iv$ and for any $\mathbb J\subset \mathbb T$, and $l=1,\ldots,p$,
\begin{equation}\label{ddd1}
 |R_{jj}^{(\mathbb J)}(u+iv/s)|\le s|R_{l+n,l+n}^{(\mathbb J)}(u+iv)|.
\end{equation}
and 
\begin{equation}\label{ddd2}
 |\frac1{u+iv/s+ym_n(u+iv/s)+\frac{y-1}{u+iv/s}}|\le s|\frac1{u+iv+ym_n(u+iv)+\frac{y-1}{u+iv}}|.
\end{equation}
\end{lem}
\begin{proof} See  Lemma 3.4 in \cite{SchleinMaltseva:2013}.
 For the readers convenience we include the short argument here.
 Note that, for any $l=1,\ldots,p$,
 \begin{equation}\notag
  |\frac d{dv}\log R^{(\mathbb J)}_{l+n,l+n}(u+iv)|\le \frac1{|R_{l+n,l+n}^{(\mathbb J)}(u+iv)|}|\frac d{dv} R_{l+n,l+n}^{(\mathbb J)}(u+iv)|.
 \end{equation}
Furthermore,
\begin{equation}\notag
 \frac d{dv} R_{l+n,l+n}^{(\mathbb J)}(u+iv)=[(\mathbf R^{(\mathbb J)})^2]_{l+n,l+n}(u+iv)
\end{equation}
and
\begin{equation}\notag
 |[(\mathbf R^{(\mathbb J)})^2]_{l+n,l+n}(u+iv)|\le v^{-1}\im R^{(\mathbb J)}_{l+n,l+n}.
\end{equation}
From here it follows that
\begin{equation}\notag
  |\frac d{dv}\log R_{l+n,l+n}^{(\mathbb J)}(u+iv)|\le v^{-1}.
 \end{equation}
 We may write now
 \begin{equation}\notag
  |\log R_{l+n,l+n}^{(\mathbb J)}(u+iv)-\log R_{l+n,
  l+n}^{(\mathbb J)}(u+iv/s)|\le \int_{v/s}^v\frac{du}u=\log s.
 \end{equation}
The last inequality yields \eqref{ddd1}. Similarly we have
\begin{align}\notag
 |-\frac d{dz}\log\{z+ym_n(z)+\frac{y-1}z\}|\le \frac{|1+y\frac d{dz}m_n(z)-\frac{1-y}{z^2}|}{|z+ym_n(z)+\frac{y-1}z|}.
\end{align}
Using that
\begin{align}
 |z+ym_n(z)+\frac{y-1}z|&\ge |\im\{z+ym_n(z)+\frac{y-1}z\}|=v(1+v^{-1}y\im m_n(z)+\frac{(1-y)v}{|z|^2}),\notag\\
 |1+y\frac d{dz}m_n(z)-\frac{1-y}{z^2}|&\le 1+y|\frac d{dz}m_n(z)|+\frac{1-y}{|z^2|}\le 1+yv^{-1}\im m_n(z)+\frac{(1-y)v}{|z|^2}.\notag
\end{align}
The last inequalities together imply
\begin{align}
 |-\frac d{dz}\log\{z+ym_n(z)+\frac{y-1}z\}|\le v^{-1}.\notag
\end{align}
From here \tc{\eqref{ddd2} follows}.
Thus Lemma \ref{schlein} is proved.
\end{proof}

\subsection{Some Auxiliary Bounds for Resolvent Matrices for $z=u+iV$ with $V=4\sqrt y$}
We shall use the bound for the $\varepsilon_{j\nu}$ for $V=\sqrt y$.
\begin{lem}\label{eps2}Assuming the conditions of Theorem \ref{main}, we get
\begin{align}\notag
 \E|\varepsilon_{j2}|^q\le \frac {C^qq}{n^{\frac q2}}.
\end{align}

\end{lem}
\begin{proof}
 Conditioning on $\mathfrak M^{(j)}$  and applying Burkholder's inequality \newline(see Lemma \ref{burkh1}), we get
 \begin{align}
  \E|\varepsilon_{j2}|^q\le C_2^qq^qn^{-q}(\E|\sum_{k=2}^p(\sum_{l=1}^{k-1}R^{(j)}_{k+n,l+n}X_{jk})^2|^{\frac q2}+
  \mu_q\sum_{k=2}^p\E|\sum_{l=1}^{k-1}R^{(j)}_{k+n,l+n}X_{jl}|^q).\notag
 \end{align}
Applying now Corollary \ref{q1*} and Rosenthal's inequality, we get
\begin{align}
 \E|\varepsilon_{j2}|^q\le C_2^qq^{2q}n^{-\frac q2}+\mu_qn^{-q}q^{2q}\sum_{l\in\mathbb T_j}\E(\sum_{k\in\mathbb T_j}|R^{(j)}_{kl}|^2)^{\frac q2}+
 \mu_q^2n^{-q}q^{2q}\sum_{k,l=1}^p\E |R^{(j)}_{k+n,l+n}|^q.\notag
\end{align}
Using that $|R^{(j)}_{k+n,l+n}|\le\frac14$ and $\sum_{l=1}^p|R^{(j)}_{kl}|^2\le \frac1{16}$ and $\mu_q\le D^{\frac q4}n^{\frac q4-1}\mu_4$, we get
\begin{align}
 \E|\varepsilon_{j2}|^q\le C_2^qq^{2q}n^{-\frac q2}.\notag
\end{align}
Thus Lemma \ref{eps2} is proved.
\end{proof}
\begin{lem}\label{eps3}Assuming the conditions of Theorem \ref{main}, we get
\begin{align}
 \E|\varepsilon_{j1}|^q\le \frac {C^qq}{n^{\frac q2}}.\notag
\end{align}

\end{lem}

\begin{proof}
 Conditioning and applying  Rosenthal's inequality, we obtain
 \begin{equation}\notag
  \E|\varepsilon_{j1}|^q\le C^qq^qn^{-q}(\mu_4^{\frac q2}\E(\sum_{l=1}^p|R^{(j)}_{l+n,l+n}|^2)^{\frac q2}+\mu_{2q}\sum_{l=1}^p\E|R^{(j)}_{l+n,l+n}|^q).
 \end{equation}
Using that $|R^{(j)}_{l+n,l+n}|\le \frac14$ and $\mu_{2q}\le D^{2q-4}n^{\frac q2-1}\mu_4$, we get
\begin{equation}\notag
 \E|\varepsilon_{j1}|^q\le C^qq^qn^{-\frac q2}.
\end{equation}
Thus Lemma \ref{eps3} is proved.
 
\end{proof}

\begin{lem}\label{lem2} Assuming the conditions of Theorem \ref{main}, we get, for any $q\ge1$,
\begin{align}\notag
 |\varepsilon_{j3}|^q\le \frac {C^q}{n^{ q}}.
\end{align}

\end{lem}
\begin{proof}Recall that
\begin{align}
 \varepsilon_{j3}=\frac1n\sum_{l=1}^pR_{l+n,l+n}-\frac1n\sum_{l=1}^pR^{(j)}_{l+n,l+n}.\notag
\end{align}
It is straightforward to check that
\begin{equation}\notag
 \frac1n\sum_{l=1}^pR_{l+n,l+n}-\frac1n\sum_{l=1}^pR^{(j)}_{l+n,l+n}=\frac1{2n}\Tr\mathbf R-\frac1{2n}\Tr\mathbf R^{(j)}+\frac1{2nz}
\end{equation}
From here follows immediately  the bound
 \begin{equation}\notag
  |\varepsilon_{j3}|\le \frac1{nv}, \text{ a. s.}
 \end{equation}
See for instance \cite{GT:2003},  Lemma 3.3.
The last bound implies the claim.
\end{proof}

\subsection{Some Auxiliary Bounds for Resolvent Matrices for $z\in\mathbb G$}We start from the simple lemma on the behavior of Stieltjes transform of symmetrizing
Marchenko -- Pastur distribution with parameter $y$.
\begin{lem}\label{resol1*}For all $z=u+iv\in\mathbb C_+$ with $1-\sqrt y\le |u|\le 1+\sqrt y$ and for nay $0<y\le1$, we have
\begin{equation}\notag
 |z+ys_y(z)|\ge \frac1{1+\sqrt y}.
\end{equation}

\end{lem}
\begin{proof}We consider representation
\begin{align}\notag
 s_y(z)=\frac{-(z+\frac{y-1}z)+\sqrt{(z+\frac{y-1}z)^2-4y}}{2y}.
\end{align}
We may rewrite this equality as follows
\begin{equation}\notag
 s_y(z)=\frac{-(z+\frac{y-1}z)+\sqrt{(z+\frac{1-y}z)^2-4}}{2y}.
\end{equation}
Introduce the notation
\begin{equation}\notag
 w=z+\frac{1-y}z.
\end{equation}
In these notation we may write
\begin{equation}\notag
 z+ys_y(z)=w+s(w),
\end{equation}
where $s(w)=\frac{-w+\sqrt{w^2-4}}2$ denotes the Stieltjes transform of the semi-circular law. Here we choose the branch of the root such that $\im\sqrt\{\cdot\}\ge0$
Furthermore, we note 
\begin{equation}\notag
 \im w\q\begin{cases}\ge 0, \text{ if } |z|\ge\sqrt{1-y},\\
                     <0,\text{ if } |z|\le \sqrt{1-y}\end{cases}.
                     \end{equation}
Therefore, since  for  Stieltjes transform of the semi-circular law $|w+s(w)|\ge 1$ for any $w\in\mathbb C_+$, we get
\begin{equation}\notag
 |z+ys_y(z)|\ge1, \text{ for all } |z|\ge\sqrt{1-y}.
\end{equation}
Consider now the function $f(z):=z+ys_y(z)$ in the domain \tc{$\mathbb K:=\{z\in\mathbb C_+:\q|z|\le \sqrt{1-y} \}$} . 
 This function is analytic \tc{in the open domain $\mathbb K$ and continuous in the domain $\mathbb K$.  By the Maximum Principle, the 
 minimum of modulus of the function $f(z)$ in the domain $\mathbb K$ is attained  on its boundary.} 
 It is straightforward to check that $|f(z)|=1$, for all $z\in\mathbb C_+:\q|z|=\sqrt{1-y}$. Moreover,
 $|f(z)|=1$, for $z=u$ with $-\sqrt{1-y}\le u\le -1+\sqrt y$ or $1-\sqrt y\le u\le \sqrt{1-y}$.
 \tc{Let $z=\pm(1-\sqrt y)+iv$ with $|z|\le \sqrt{1-y}$ and $v>0$. We consider the case $z=1-\sqrt y+iv$ only.} It is straightforward to check that
 \begin{equation}\notag
  \re(w^2-4)\le 0,\quad \im(w^2-4)\le 0.
 \end{equation}
Moreover,
\begin{align}\notag
 |w|\le 2, \text{ and } |\sqrt{w^2-4}|\le 2\sqrt y
\end{align}

 We may write
 \begin{equation}\notag
  f(z)=\frac{2}{w-\sqrt{w^2-4}}.
 \end{equation}
These relations together imply that, for $z=1-\sqrt y+iv$ with $|z|\le \sqrt{1-y}$ and $v>0$
\begin{equation}\notag
 |f(z)|\ge \frac1{1+\sqrt y}.
\end{equation}

This inequality proves the Lemma.
\end{proof}

Recall that
\begin{align}\label{region}
 \mathbb G &:=\{z=u+iv\in\mathbb C^+:\,u\in\mathbb J_{\varepsilon}, v\ge v_0/\sqrt{\gamma}\},\;\; \text{where} \; v_0= A_0 n^{-1},\\
  \mathbb J_{\varepsilon} &=\{u: (1-\sqrt y)+\varepsilon\le |u|\le 1+\sqrt y-\varepsilon\}, \quad \varepsilon:=c_1n^{-\frac23}, \notag\\ \gamma&=\gamma(u)=
  \min\{|u|-1+\sqrt y,1+\sqrt y-|u| \}.\notag
\end{align}
In the next lemma we prove some simple inequalities for the region $\mathbb G$.

\begin{lem}\label{lemG}
 For any $z\in\mathbb G$ we have
 \begin{align}
  |(z+\frac{y-1}z)^2-4y|&\ge \frac{4y}{5}\max\{\gamma,v\},\notag\\
  {nv}\sqrt{|(z+\frac{y-1}z)^2-4y|}&\ge A_0\frac{4y}{5}.
 \end{align}

\end{lem}
\begin{proof}For $z\in\mathbb G$ we have
\begin{equation}
 |z|\le ((1+\sqrt y)^2+v^2)^{\frac12}\le 5/\sqrt y.
\end{equation}
Furthermore, we observe that
 \begin{align}
  |(z+\frac{y-1}z)^2-4y|&
 =|(z+\frac{y-1}z-2\sqrt y)(z+\frac{y-1}z+2\sqrt y)|\notag\\&=
 \frac{|z-(1+\sqrt y))(z+(1+\sqrt y))(z-(1-\sqrt y))(z+(1-\sqrt y))|}{|z|^2}.
 \end{align}
 Assume that $1-\sqrt y\le u\le 1$. Then, for $z\in\mathbb G$,
 \begin{align}
 |(z+\frac{y-1}z)^2-4y|&\ge v \frac{|u-(1+\sqrt y))(u+(1+\sqrt y))||(z+(1-\sqrt y))|}{|z|^2}\notag\\&\ge4v\sqrt y\frac1{|z|}\ge 
 v\frac{4y}{5},\notag
 \end{align}
 and
 \begin{align}
 |(z+\frac{y-1}z)^2-4y|&\ge \gamma \frac{|u-(1+\sqrt y))(u+(1+\sqrt y))(u+(1-\sqrt y))|}{|z|^2}\notag\\&\ge
 \gamma\frac{4 y}{5}.\notag
 \end{align}
Similarly  we get the lower bound for $-1-\sqrt y\le u\le -1$, $-1\le u\le -1+\sqrt y$ and for  $1\le u\le 1+\sqrt y$.
This inequality proves the Lemma.
\end{proof}

\begin{lem}\label{lem00}Assuming the conditions  of Theorem \ref{main}, there exists an absolute constant $c_0>0$ such that for any $\mathbb J\subset \mathbb T$,
\begin{equation}\label{lem00.1}
 |z+ym_n^{(\mathbb J)}(z)+ys_y(z)+\frac{y-1}z|\ge y\im m_n^{(\mathbb J)}(z)+\im\{\frac{y-1}z\},
\end{equation}
moreover, for $z\in\mathbb G$,
\begin{equation}\label{lem00.2}
 |z+ym_n^{(\mathbb J)}(z)+ys_y(z)+\frac{y-1}z|\ge c_0\sqrt{|(z+\frac{y-1}z)^2-4y|}.
\end{equation}

\end{lem}

\begin{proof}
 Firstly we note
 \begin{align}
 |z+y(m_n^{(\mathbb J)}(z)+s_y(z))+\frac{y-1}z|&\ge\im\{ys_y(z)+z+\frac{y-1}z\}\notag\\&
 \ge\frac12\im\sqrt{(z+\frac{y-1}z)^2-4y}.\notag
 \end{align}
Furthermore, it is simple to check that, for $z\in \mathbb G$
\begin{equation}\notag
 \re\{(z+\frac{y-1}z)^2-4y\}\le 0.
\end{equation}
This implies that
\begin{equation}\notag
 \im\sqrt{(z+\frac{y-1}z)^2-4}\}\ge \frac{\sqrt2}2\sqrt{|(z+\frac{y-1}z)^2-4y|}.
\end{equation}
Thus Lemma \ref{lem00} is proved.
\end{proof}



\begin{lem}\label{basic2}Assuming the conditions of Theorem \ref{main}, there exists an absolute constant $C>0$ such that for any $j=1,\ldots,n$, 
 \begin{equation}\label{basic3}
  \E\{|\varepsilon_{j2}|^2\big|\mathfrak M^{(j)}\}\le \frac C{n}(v^{-1}\im m_n^{(j)}(z)+\frac{1-y}{|z|^2}),
 \end{equation}
and
\begin{equation}\label{basic4}
 \E\{|\varepsilon_{j2}|^4\big|\mathfrak M^{(j)}\}\le \frac {C\mu_4^2}{n^2}(v^{-1}\im m_n^{(j)}(z)+\frac{1-y}{|z|^2})^2.
\end{equation}

\end{lem}
\begin{proof}
 Note that r.v.'s $X_{jl}$, for $l=1,\ldots,p$ are independent of $\mathfrak M^{(j)}$ and that for $l,k=1,\ldots,p$,  $R^{(j)}_{l+n,k+n}$
 are measurable with respect to $\mathfrak M^{(j)}$.
 This implies that $\varepsilon_{j2}$ is a quadratic form with coefficients $R^{(j)}_{l+n,l+n}$ independent of $X_{jl}$. Thus
 its variance and fourth moment are easily available.
 \begin{equation}\notag
  \E\{|\varepsilon_{j2}|^2\big|\mathfrak M^{(j)}\}=\frac1{n^2}\sum_{l,k=1}^p|R^{(j)}_{l+n,k+n}|^2\le \frac1{n^2}\Tr\mathbb |\mathbf R^{(j)}|^2,
 \end{equation}
Here we use the notation $|\mathbf A|^2=\mathbf A\mathbf A^*$ for any matrix $\mathbf A$.
Applying Lemma \ref{resol00}, inequality \eqref{res1},
we get equality \eqref{basic3}.

Furthermore, direct calculations show that
\begin{align}
 \E\{|\varepsilon_{j2}|^4\big|\mathfrak M^{(j)}\}&\le \frac C{n^2}(\frac1n\sum_{1\le l\ne k\le p}|R^{(j)}_{l+n,k+n}|^2)^2
 +\frac {C\mu_4^2}{n^2}\frac1{n^2}\sum_{l,k=1}^p|R^{(j)}_{l+n,k+n}|^4\notag\\&\le \frac {C\mu_4^2}{n^2}(\frac1n\sum_{1\le l\ne k\le p}|R^{(j)}_{l+n,k+n}|^2)^2
 \le\frac{C\mu_4^2}{n^2}(v^{-1}\im m_n^{(j)}(z)+\frac{1-y}{|z|^2})^2.\notag
\end{align}
Here again we used Lemma \ref{resol00}, inequality \eqref{res1}.
Thus Lemma \ref{basic2} is proved.
\end{proof}
\begin{lem}\label{basic5}Assuming the conditions of Theorem \ref{main}, there exists an absolute constant $C>0$ such that for any $j=1,\ldots,n$,
 \begin{align}\label{basic6}
  \E\{|\varepsilon_{j1}|^2\big|\mathfrak M^{(j)}\}\le \frac {C\mu_4}{n}\frac1n\sum_{l=1}^p|R^{(j)}_{l+n,l+n}|^2,
 \end{align}
 and
 \begin{align}\label{basic7}
  \E\{|\varepsilon_{j1}|^4\big|\mathfrak M^{(j)}\}\le \frac {C\mu_4}{n^2}\frac1n\sum_{l=1}^p|R^{(j)}_{l+n,l+n}|^4.
 \end{align}

\end{lem}
\begin{proof}The first inequality is obvious. To prove the second inequality, we apply
 Rosenthal's inequality. We obtain
\begin{align}
\E\{|\varepsilon_{j1}|^4\big|\mathfrak M^{(j)}\}\le \frac {C\mu_4}{n^2}(\frac1n\sum_{l=1}^p|R^{(j)}_{l+n,l+n}|^2)^2
+\frac {C\mu_8}{n^3}\frac1n\sum_{l=1}^p|R^{(j)}_{l+n,l+n}|^4. \notag
\end{align}
Using  $|X_{jl}|\le Cn^{\frac 14}$ we get 
$\mu_8\le Cn\mu_4$ and the claim.
Thus Lemma \ref{basic5} is proved.
\end{proof}
\begin{cor}\label{corgot}Assuming the conditions of Theorem \ref{main}, there exists an absolute constant $C>0$, depending on $\mu_4$ and $D$ only, 
such that for any $j=1,\ldots,n$,  $\nu=1,2$,
$z\in\mathbb G$, and $1\le \alpha\le \frac12A_1(nv)^{\frac14}$,
\begin{align}\label{ingot}
 \E\frac{|\varepsilon_{j\nu}|^2}{|z+y(m_n^{(j)}(z)+s_y(z))+\frac{y-1}z||z+ym_n^{(j)}(z)+\frac{y-1}z|^{\alpha}}\le \frac C{nv}
\end{align}
and 
 \begin{align}\label{corgot1}
 \E\frac{|\varepsilon_{j\nu}|^4}{|z+y(m_n^{(j)}(z)+s_y(z))+\frac{y-1}z|^2|z+ym_n^{(j)}(z)+\frac{y-1}z|^{\alpha}}\le \frac C{n^2v^2}.
\end{align}
\end{cor}
\begin{proof}
 For $\nu=1$, by Lemma \ref{lem00}, we have
 \begin{align}\label{rrr6}
  \E&\frac{|\varepsilon_{j1}|^2}{|z+y(m_n^{(j)}(z)+s_y(z))+\frac{1-y}{z}||z+ym_n^{(j)}(z)+\frac{1-y}{z}|^{\alpha}}\notag\\&\qquad\le
  \E\frac{y\im m_n^{(j)}+\frac {C(1-y)v}{|z|}}{nv|z+y(m_n^{(j)}(z)+s(z))+\frac{1-y}z||z+ym_n^{(j)}(z)+\frac{1-y}z|^{\alpha}}.
 \end{align}
 Note that, for $z\in\mathbb G$,
 \begin{equation}\label{uuu1}
  y\im m_n^{(j)}+\frac {C(1-y)v}{|z|}\le |z+y(m_n^{(j)}(z)+s(z))+\frac{1-y}z|.
 \end{equation}
This inequality and inequality \eqref{rrr6} together imply
\begin{align}
 \E&\frac{|\varepsilon_{j1}|^2}{|z+ym_n^{(j)}(z)+\frac{y-1}z+s_y(z)||z+ym_n^{(j)}(z)+\frac{y-1}z|^{\alpha}}\le \frac C{nv}\E\frac1{|z+ym_n^{(j)}(z)+\frac{1-y}z|^{\alpha}}.\notag
\end{align}

Applying now Corollary \ref{cor8}, we get the claim. The proof of the second inequality for $\nu=1$ is similar.
For $\nu=2$ we apply Lemma \ref{basic2}, inequality \eqref{basic3} and obtain, using \eqref{uuu1},
\begin{align}\notag
  \E&\frac{|\varepsilon_{j2}|^2}{|z+y(m_n^{(j)}(z)+s(z))+\frac{y-1}z||z+ym_n^{(j)}(z)+\frac{y-1}z|^{\alpha}}\le \frac C{nv}.
 \end{align}
Similarly, using Lemma \ref{basic2}, inequality \eqref{basic4}, we get
 \begin{align}
  \E&\frac{|\varepsilon_{j2}|^4}{|z+y(m_n^{(j)}(z)+s_y(z))+\frac{y-1}z|^2|z+ym_n^{(j)}(z)+\frac{y-1}z|^{\alpha}}\notag\\&\le 
  \frac C{n^2v^2}\E\frac{(\im m_n^{(j)}(z)+\frac{(1-y)v}{|z|^2})^2}{|z+y(m_n^{(j)}(z)+s_y(z))+\frac{y-1}z|^2|z
  +ym_n^{(j)}(z)+\frac{y-1}z|^{\alpha}}\notag\\&\qquad\qquad\le \frac C{n^2v^2}\E\frac{1}{|z+ym_n^{(j)}(z)+\frac{y-1}z|^{\alpha}}.\notag
 \end{align}
 Applying Corollary \ref{cor8}, we get the claim.
 
\end{proof}


\begin{lem}\label{basic8}Assuming the conditions of Theorem \ref{main}, there exists an absolute constant $C>0$ such that for any $j=1,\ldots,n$,
 \begin{align}
  |\varepsilon_{j3}|\le \frac C{nv}\quad\text{a.s.}
 \end{align}
 \end{lem}
\begin{proof}Firstly we represent
\begin{equation}\label{tracej}
 \varepsilon_{j3}=\frac1{2n}\Tr \mathbf R-\frac1{2n}\Tr\mathbf R^{(j)}+\frac{1-y}{2z}.
\end{equation}
This equality implies
\begin{equation}\notag
 |\varepsilon_{j3}|\le \frac1{nv}.
\end{equation}

\end{proof}
\begin{lem}\label{lam1**}Assuming the conditions of Theorem \ref{main}, we have, for  $z\in\mathbb G$,
\begin{align}
\E|\Lambda_n|^2\le \frac C{nv|(z+\frac{y-1}z)^2-4|^{\frac12}}.\notag
\end{align}
\end{lem}
\begin{proof}We write
\begin{align}
 \E|\Lambda_n|^2=\E\Lambda_n\overline \Lambda_n&=\E\frac{T_n}{z+y(m_n(z)+s_y(z))+\frac{y-1}z}\overline\Lambda_n\notag\\&=
 \sum_{\nu=1}^3\E\frac{T_{n\nu}}{z+y(m_n(z)+s(z))+\frac{y-1}z}\overline\Lambda_n,\notag
\end{align}
where
\begin{align}
 T_{n\nu}:=\frac1n\sum_{j=1}^n\varepsilon_{j\nu}R_{jj}, \text{ for }\nu=1,\ldots,3.\notag
\end{align}
Applying Cauchy -- Schwartz inequality , we get
\begin{equation}\label{ll2}
 \E^{\frac12}|\Lambda_n|^2\le \sum_{\nu=1}^3\E^{\frac12}\frac{|T_{n\nu}|^2}{|z+y(m_n(z)+s_y(z))+\frac{y-1}z|^2}.
\end{equation}

First we observe that 
\begin{align}
 T_{n3}&=\frac1{2n^2}\sum_{j=1}^n(\Tr \mathbf R-\Tr \mathbf R^{(j)})R_{jj}+\frac1{2n^2z}\sum_{j=1}^nR_{jj}\notag\\&=
 -\frac1{2n}\frac d{dz}m_n(z)+\frac{1-y}{2nz}m_n(z).\notag
\end{align}
Therefore,
\begin{align}\notag
 |T_{n3}|=\frac1n|m_n'(z)|+\frac{1-y}{|z|}|m_n(z)|\le \frac1{nv}\im m_n(z)+\frac{1-y}{n|z|}|m_n(z)|.
\end{align}
 Hence $|z+y(m_n(z)+s_y(z))+\frac{1-y}z|\ge \im m_n(z)+\frac{(1-y)v}{|z|^2}$ and Jensen's inequality yields
\begin{equation}\label{tn4}
 \E\frac{|T_{n3}|^2}{|z+y(m_n(z)+s_y(z))+\frac{1-y}z|^2}\le \frac C{n^2v^2}(1+|z|^2)\le \frac {4C}{n^2v^2}.
\end{equation}
Furthermore, we observe that, 
\begin{align}
\frac1{|z+y(s_y(z)+m_n(z))+\frac{y-1}z|}&\le\frac1{|z+y(s_y(z)+m_n^{(j)}(z))+\frac{y-1}z|}\notag\\&\qquad\qquad\times(1+
\frac{|\varepsilon_{j3}|}{|z+y(s_y(z)+m_n(z))+\frac{y-1}z|} ).\notag
\end{align}
Therefore, by Lemmas \ref{basic8} and  \ref{lemG}, for $z\in\mathbb G$,
\begin{align}\label{lar1}
\frac1{|z+y(s_y(z)+m_n(z))+\frac{y-1}z|}&\le\frac C{|z+y(s_y(z)+m_n^{(j)}(z))+\frac{y-1}z|}.
\end{align}
Applying inequality \eqref{lar1}, we may write, for $\nu=1,2$
\begin{align}
 \E&\frac{|T_{n\nu}|^2}{|z+y(m_n(z)+s_y(z))+\frac{y-1}z|^2}\notag\\&\qquad\qquad\le \frac1n\sum_{j=1}^n
 \E\frac{|\varepsilon_{j\nu}|^2|R_{jj}|^2}{|z+y(s_y(z)+m_n^{(j)}(z))+\frac{y-1}z|^2}.\notag
\end{align}
Applying Cauchy -- Schwartz inequality and Lemma \ref{lem00}, we get
\begin{align}
 \E&\frac{|T_{n\nu}|^2}{|z+y(m_n(z)+s_y(z))+\frac{y-1}z|^2}\notag\\&\le \frac C{n|(z+\frac{y-1}z)^2-4|^{\frac12}}\sum_{j=1}^n
 \E^{\frac12}\frac{|\varepsilon_{j\nu}|^4}{|z+y(m_n^{(j)}(z)+s_y(z))+\frac{y-1}z|^2}
 \E^{\frac12}|R_{jj}|^4.\notag
\end{align}
Using now Corollary \ref{corgot}, inequality \eqref{corgot1} and Corollary \ref{cor8}, we get for $\nu=1,2,3$
\begin{align}\label{lll2}
 \E\frac{|T_{n\nu}|^2}{|z+y(m_n(z)+s_y(z))+\frac{y-1}z|^2}\le \frac C{nv|(z+\frac{y-1}z)^2-4|^{\frac12}}.
\end{align}
Inequalities \eqref{ll2}, \eqref{tn4} and \eqref{lll2} together complete the proof.
Thus Lemma \ref{lam1**} is proved.
\end{proof}

\begin{lem}\label{lam1*}Assuming the conditions of Theorem \ref{main}, we have, for  $z\in\mathbb G$,
\begin{align}\notag
\E|\Lambda_n|^2\le \frac C{n^2v^2}.
\end{align}
\end{lem}
\begin{proof}We write
\begin{align}
 \E|\Lambda_n|^2&=\E\Lambda_n\overline \Lambda_n=\E\frac{T_n}{z+y(m_n(z)+s(z))+\frac{y-1}z}\overline\Lambda_n\notag\\&=
 \sum_{\nu=1}^4\E\frac{T_{n\nu}}{z+y(m_n(z)+s(z))+\frac{y-1}z}\overline\Lambda_n,\notag
\end{align}
where
\begin{align}
 T_{n\nu}:=\frac1n\sum_{j=1}^n\varepsilon_{j\nu}R_{jj}, \text{ for }\nu=1,2,3.\notag
\end{align}
First we observe that by \eqref{shur}
\begin{align}
 |T_{n3}|=\frac1n|m_n'(z)|\le \frac1{nv}(\im m_n(z)+\frac{(1-y)v}{|z|^2}).\notag
\end{align}
 Hence $|z+y(m_n^{(j)}(z)+s(z))+\frac{y-1}z|\ge \im m_n^{(j)}(z)+\frac{(1-y)v}{|z|^2}$ and Jensen's inequality yields
\begin{equation}\label{tn4**}
 |\E\frac{T_{n3}}{z+y(m_n(z)+s_y(z))+\frac{y-1}z}\overline\Lambda_n|\le \frac 1{nv}\E^{\frac12}|\Lambda_n|^2.
\end{equation}

Consider now the quantity
\begin{equation}\notag
 Y_{\nu}:=\E\frac{T_{n\nu}}{z+y(m_n(z)+s_y(z))+\frac{y-1}z}\overline\Lambda_n,
\end{equation}
for $\nu=1,2$.
We represent it as follows
\begin{equation}\notag
 Y_{\nu}=Y_{\nu1}+Y_{\nu2},
\end{equation}
where
\begin{align}
 Y_{\nu1}&=-\frac1n\sum_{j=1}^n\E\frac{\varepsilon_{j\nu}\overline\Lambda_n}{(z+ym_n^{(j)}(z)+\frac{y-1}z)(z+y(m_n(z)+s_y(z))+\frac{y-1}z)},\notag\\
 Y_{\nu2}&=\frac1n\sum_{j=1}^n\E\frac{\varepsilon_{j\nu}(R_{jj}+\frac1{z+ym_n^{(j)}(z)+\frac{y-1}z})\overline\Lambda_n}{z+y(m_n(z)+s_y(z))+\frac{y-1}z}.\notag
\end{align}
By the representation \eqref{repr01}, which is similar to \eqref{repr001} we have
\begin{align}
 Y_{\nu2}=\frac1n\sum_{j=1}^n\E\frac{\varepsilon_{j\nu}(\varepsilon_{j1}+\varepsilon_{j2})\overline\Lambda_n}{(z+y(m_n(z)+s_y(z))+\frac{y-1}z)(z+ym_n^{(j)}(z)+\frac{y-1}z)}.\notag
\end{align}
Using inequality \eqref{lar1}, we may write, for $z\in\mathbb G$
\begin{align}
 |Y_{\nu2}|\le\frac Cn\sum_{j=1}^n\E\frac{|\varepsilon_{j\nu}||\varepsilon_{j1}+\varepsilon_{j2}||\overline\Lambda_n|}{|z+y(m_n^{(j)}(z)+s_y(z))+\frac{y-1}z|
 |z+ym_n^{(j)}(z)+\frac{y-1}z|}. \notag
\end{align}
Applying the Cauchy -- Schwartz inequality and the inequality $ab\le \frac12(a^2+b^2)$, we get
\begin{align}\label{y21}
 |Y_{\nu2}|&\le\frac Cn\sum_{j=1}^n\E^{\frac12}\frac{|\varepsilon_{j\nu}|^2|\varepsilon_{j1}+\varepsilon_{j2}|^2}{|z+ym_n^{(j)}(z)+\frac{y-1}z+s(z)|^2|z+ym_n^{(j)}(z)+\frac{y-1}z|^2}
 \E^{\frac12}|\Lambda_n|^2\notag\\&\le\frac Cn\sum_{j=1}^n\E^{\frac12}\frac{|\varepsilon_{j1}+\varepsilon_{j2}|^4}{|z+y(m_n^{(j)}(z)+s(z))+\frac{y-1}z|^2
 |z+ym_n^{(j)}(z)+\frac{y-1}z|^2}
 \E^{\frac12}|\Lambda_n|^2 .
\end{align}
Using Corollary \ref{corgot} with $\alpha=2$, we arrive at
\begin{equation}\label{fnu2*}
 |Y_{\nu2}|\le \frac C{nv}\E^{\frac12}|\Lambda_n|^2.
\end{equation}
In order to estimate $Y_{\nu1}$ we introduce now the quantity

\begin{equation}\notag
 \Lambda_n^{(j,1)}=\frac1n\Tr\mathbf R^{(j)}-s_y(z)+\frac{s_y(z)}{2n}+\frac1{2nz}.
\end{equation}
Note that
\begin{align}\label{deltan1}
 \Lambda_n-\Lambda_n^{(j,1)}&=\frac1n(\frac12(\Tr\mathbf R-\Tr\mathbf R^{(j)}+\frac1z)-\frac{s(z)}{2n}-\frac1{2nz}\notag\\&=\frac{R_{jj}-s_y(z)}{2n}
 +\frac1{np}\sum_{l,k=1}^pX_{jl}X_{jk}[(\mathbf R^{(j)})^2]_{k+n,l+n}
 =\delta_{nj}.
\end{align}
We represent $Y_{\nu1}$ in the form
\begin{equation}\notag
 Y_{\nu1}=Z_{\nu1}+Z_{\nu2}+Z_{\nu3},
\end{equation}
where
\begin{align}
 Z_{\nu1}&=-\frac1n\sum_{j=1}^n\E\frac{\varepsilon_{j\nu}{\overline\Lambda}_n^{(j1)}}{(z+ym_n^{(j)}(z)+\frac{y-1}z)(z+y(m_n^{(j)}(z)+s_y(z))+\frac{y-1}z},\notag\\
 Z_{\nu2}&=\frac1n\sum_{j=1}^n\E\frac{\varepsilon_{j\nu}\overline\delta_{nj}}{(z+ym_n^{(j)}(z)+\frac{y-1}z)(z+y(m_n(z)+s(z))+\frac{y-1}z)},\notag\\
 Z_{\nu3}&=\frac1n\sum_{j=1}^n\E\frac{\varepsilon_{j\nu}{\overline\Lambda}_n^{(j1)}}{(z+ym_n^{(j)}(z)+\frac{y-1}z)
 (z+y(m_n^{(j)}(z)+s_y(z))+\frac{y-1}z)}\notag\\&\qquad\qquad\qquad\qquad\times\frac{\varepsilon_{j3}}{(z+y(m_n(z)+s_y(z))+\frac{y-1}z)}.\notag
\end{align}
First,  note that by \tc{conditional} independence
\begin{equation}\label{v1}
 Z_{\nu1}=0.
\end{equation}
Furthermore, applying H\"older's inequality, we get
\begin{align}
 |Z_{\nu3}|&\le \frac1n\sum_{j=1}^n\E^{\frac14}\frac{|\varepsilon_{j\nu}|^4}{|z+y(m_n^{(j)}(z)+s_y(z))+\frac{y-1}z|^4|z+ym_n^{(j)}(z)+\frac{y-1}z|^4}\notag\\&
 \qquad\qquad\qquad\qquad\qquad\times
 \E^{\frac14}\frac{|\varepsilon_{j3}|^4}{|z+y(m_n(z)+s_y(z))+\frac{y-1}z|^4}\E^{\frac12}|\Lambda_n^{(j,1)}|^2.\notag
\end{align}
Using  Corollary \ref{corgot} with $\alpha=4$ and Lemmas \ref{lem14} and \ref{lem00}, we obtain
\begin{align}
 |Z_{\nu3}|\le \frac{C}{(nv)^{\frac32}|(z+\frac{y-1}z)^2-4y|^{\frac14}}\frac1n\sum_{j=1}^n\E^{\frac12}|\Lambda_n^{(j,1)}|^2.\notag
\end{align}
Applying now Corollary \eqref{cor8}, we get
\begin{equation}\notag
 |Z_{\nu3}|\le \frac{C}{(nv)^{\frac32}|(z+\frac{y-1}z)^2-4y|^{\frac14}}\E^{\frac12}|\Lambda_n|^2
 +\frac{C}{(nv)^{\frac32}|(z+\frac{y-1}z)^2-4y|^{\frac14}}\frac1n\sum_{j=1}^n\E^{\frac12}|\delta_{nj}|^2.
\end{equation}
For $z\in\mathbb G$ we may rewrite this bound using Lemma \ref{lemG}
\begin{equation}\notag
 |Z_{\nu3}|\le \frac{C}{nv}\E^{\frac12}|\Lambda_n|^2
 +\frac{C}{nv}\frac1n\sum_{j=1}^n\E^{\frac12}|\delta_{nj}|^2.
\end{equation}
By definition of $\delta_{nj}$, see \eqref{deltan1}, we have
\begin{align}
 \E|\delta_{nj}|^2\le C\Big (\frac1{n^2}\E|R_{jj}-s(z)|^2+\E^{\frac12}\Big|\frac1{n^2}\sum_{l,k\in\mathbb T_j}X_{jl}X_{jk}[(\mathbf R^{(j)})^2]_{k+n,l+n}\Big|^4
 \E^{\frac12}|R_{jj}|^4\Big).\notag
\end{align}
By representation \eqref{repr001}, we have
\begin{align}
 \E|R_{jj}-s(z)|^2\le \E|\Lambda_n|^2+\E^{\frac12}|\varepsilon_j|^4\E^{\frac12}|R_{jj}|^4.
\end{align}
Note that by Lemmas  \ref{basic2},\ref{basic4}, \ref{basic6}, we have
\begin{equation}\label{forgot}
 \E|\varepsilon_j|^4\le \frac C{n^2v^2}.
\end{equation}

By  Corollaries \ref{cor8} and \ref{forgot}, we get
\begin{equation}\label{raz*}
 \E|R_{jj}-s(z)|^2\le \E|\Lambda_n|^2+\frac C{nv}.
\end{equation}
By Lemmas  \ref{bet1a}, inequality \eqref{bet12} and Corollary \ref{cor8}, we have
\begin{align}\label{dva}
 \E|\frac1{n^2}\sum_{l,k\in\mathbb T_j}X_{jl}X_{jk}[(\mathbf R^{(j)})^2]_{k+n,l+n}|^4=\frac1{n^4}\E|\eta_{j3}|^4\le \frac C{n^6v^6}.
\end{align}
Inequalities \eqref{raz*} and \eqref{dva} together imply
\begin{align}\label{delta2*}
 \E|\delta_{nj}|^2\le \frac1{n^2}\E|\Lambda_n|^2+\frac C{n^3v}+\frac C{n^6v^6}.
\end{align}
Therefore, for $z\in\mathbb G$,
\begin{align}\label{znu3*}
 |Z_{\nu3}|\le\frac{C}{nv}\E^{\frac12}|\Lambda_n|^2+\frac{C}{n^4v^4}+\frac1{n^2v}\E^{\frac12}|\Lambda_n|^2\le \frac C{nv}\E^{\frac12}|\Lambda_n|^2+\frac C{n^2v^2}.
\end{align}
To bound $Z_{\nu2}$ we first apply  inequality \eqref{lar1} and obtain
\begin{align}
  |Z_{\nu2}|&\le\frac Cn\sum_{j=1}^n\E\frac{|\varepsilon_{j\nu}||\delta_{nj}|}
  {|z+ym_n^{(j)}(z)+\frac{y-1}z||z+y(m_n^{(j)}(z)+s_y(z))+\frac{y-1}z|}.\notag
\end{align}
Applying now H\"older's inequality, we get
\begin{align}
 |Z_{\nu2}|&\le\frac Cn\sum_{j=1}^n\E^{\frac12}\frac{|\varepsilon_{j\nu}|^2}{|z+ym_n^{(j)}(z)+\frac{y-1}z|^2|z+ym_n^{(j)}(z)+\frac{y-1}z+s(z)|^2}\E^{\frac12}|\delta_{nj}|^2.\notag
\end{align}
Therefore, by Corollary \ref{corgot}
\begin{align}
 |Z_{\nu2}|\le \frac{C}{\sqrt{nv}|(z+\frac{y-1}z)^2-4y|^{\frac14}}\frac 1n\sum_{j=1}^n\E^{\frac12}|\delta_{nj}|^2.\notag
\end{align}
The last inequality together with Lemma \ref{lemG} and inequality \eqref{delta2*},  imply
\begin{align}\label{znu2}
 |Z_{\nu2}|&\le \frac{C}{n\sqrt{nv}|(z+\frac{y-1}z)^2-4y|^{\frac14}}\E^{\frac12}|\Lambda_n|^2+\frac C{n^{2}v}+\frac C{n^{\frac72}v^{\frac72}}
 \notag\\&\le\frac C{nv}\E^{\frac12}|\Lambda_n|^2+\frac C{n^2v^2}.
\end{align}
Combining now inequalities \eqref{tn4**}, \eqref{v1}, \eqref{znu3*}, \eqref{znu2}, we get
\begin{equation}\label{finalrek*}
 \E|\Lambda_n|^2\le \frac C{nv}\E^{\frac12}|\Lambda_n|^2+\frac C{n^2v^2}.
\end{equation}
Solving this inequality with respect to $\E|\Lambda_n|^2$  completes the proof of Lemma \ref{lam1*}. 
Thus Lemma \ref{lam1*} is proved.
\end{proof}
\begin{lem}\label{deltan4}
 There exists a positive constant $C$ such that
\begin{equation}
|\delta_{n3}|\le \frac1{nv}\im m_n(z).
\end{equation}
\end{lem}
\begin{proof}It is easy to check that
\begin{equation}\notag
 \sum_{k=1}^p(R_{k+n,k+n}-R^{(j)}_{k+n,k+n})=\frac12(\Tr\mathbf R-\Tr\mathbf R^{(j)})
+\frac{1}{2z}.
\end{equation}
 By formula (5.4) in \cite{GT:2003}, we have
\begin{align}
 (\Tr\mathbf R-\Tr\mathbf R^{(j)})R_{jj}=
(1+\frac1p\sum_{l,k=1}^nX_{jl}X_{jk}(R^{(j)})^2)_{l+n,k+n})R_{jj}^2
=-\frac{d}{dz}R_{jj}.\notag
\end{align}
From here it follows that
\begin{equation}\notag
 \frac1{n^2}\sum_{j=1}^n(\Tr\mathbf R-\Tr\mathbf R^{(j)})R_{jj}=-\frac1{n}\frac{d}{dz}m_n(z) .
\end{equation}
Note that
\begin{equation}\notag
m_n(z)=\frac{z}n\sum_{k=1}^n\frac1{s_k^2-z^2}
\end{equation}
and
\begin{equation}\notag
\frac{d}{dz}m_n(z)=\frac{m_n(z)}z-\frac{2z^2}n\sum_{k=1}^n\frac1{(s_k^2-z^2)^2}.
\end{equation}
This implies that
\begin{equation}\label{rrr7}
\delta_{n3}=\frac1{2n}(-m_n'(z)+\frac{m_n(z)}z)=\frac{z^2}{n^2}\sum_{k=1}^n\frac1{(s_k^2-z^2)^2}.
\end{equation}
Finally, we note that
\begin{equation}
 \im m_n(z)=\frac1n\sum_{k=1}^n\frac{v(s_k^2+|z|^2)}{|s_k^2-z^2|^2}.
\end{equation}
The last relation implies
\begin{equation}\label{last}
 \Big|\frac{z^2}{n^2}\sum_{k=1}^n\frac1{(s_k^2-z^2)^2}\Big|\le \frac1{nv}\im m_n(z).
\end{equation}
The inequality \eqref{last} concludes the proof.
Thus Lemma \ref{deltan4} is proved.
\end{proof}




We introduce the following quantity
\begin{align}\beta_{j1}&=\frac1p\sum_{l=1}^p[(\mathbf R^{(j)})^2]_{l+n,l+n}-\frac1p\sum_{l=1}^p[(\mathbf R)^2]_{l+n,l+n},\notag\\
\beta_{j2}&=\frac1p\sum_{1\le l\ne k\le p}X_{jl}X_{jk}[(\mathbf R^{(j)})^2]_{l+n,k+n},
\notag\\
\beta_{j3}&=\frac1n\sum_{l=1}^p(X^2_{jl}-1)[(\mathbf R^{(j)})^2]_{l+n,l+n}
.\notag
\end{align}
\begin{lem}\label{bet1*}Assuming the conditions of Theorem \ref{main}, we have, for $\nu=2,3$,
\begin{align}\label{bet3}
\E\{|\beta_{j\nu}|^2\Big|\mathfrak M^{(j)}\}\le\frac C{nv^3}\im m_n^{(j)}(z).
\end{align}
\end{lem}
\begin{proof}We recall that by $C$ we denote \tc{a} generic constant depending on $\mu_4$ and $D$ only.
By definition of $\beta_{j\nu}$ for $\nu=1,2$, conditioning on $\mathfrak M^{(j)}$, we get 
\begin{align}
\E\{|\beta_{j2}|^2\big|\mathfrak M^{(j)}\}\le &\frac C{n^2}\sum_{1\le l\ne k\le p}|[(\mathbf R^{(j)})^2]_{k+n,l+n}|^2
\le \frac C{n^2}\sum_{1\le l, k\le p}|[(\mathbf R^{(j)})^2]_{k+n,l+n}|^2,\notag\\
\E\{|\beta_{j3}|^2\big|\mathfrak M^{(j)}\}\le&\frac C{n^2}\sum_{l=1}^p|[(\mathbf R^{(j)})^2]_{l+n,l+n}|^2\le 
\frac C{n^2}\sum_{l, k=1}^p|[(\mathbf R^{(j)})^2]_{k+n,l+n}|^2.\notag
\end{align}

Applying Lemma \ref{resol00}, we get the claim.
Thus Lemma \ref{bet1*} is proved.
\end{proof}
\begin{lem}\label{bet1a}Assuming the conditions of Theorem \ref{main}, we have, 
\begin{align}\label{bet12}
\E\{|\beta_{j\nu}|^8\Big|\mathfrak M^{(j)}\}\le\frac C{n^{4}v^{12}}.
\end{align}

\end{lem}
\begin{proof} \tc{Let $C$  denote a generic constant depending on $\mu_4$ and $D$ only}. First we consider the case $q=8$. 
By definition of $\beta_{j\nu}$ for $\nu=2,3$, conditioning on $\mathfrak M^{(j)}$  , direct calculations \tc{shows}
\begin{align}
\E\{|\beta_{j2}|^8\big|\mathfrak M^{(j)}\}\le &\frac C{n^{8}}\Big((\sum_{1\le l\ne k\le p}|[(R^{(j)})^2]_{k+n,l+n}|^2)^4+
\mu_4^4(\sum_{1\le l\ne k\le p}|[(R^{(j)})^2]_{k+n,l+n}|^4)^2\notag\\&+\mu_6^2(\sum_{1\le l\ne k\le p}|[(R^{(j)})^2]_{k+n,l+n}|^6)
\mu_8^2(\sum_{1\le l\ne k\le p}|[(R^{(j)})^2]_{k+n,l+n}|^2)\notag\\&+\mu_8^2\sum_{1\le l\ne k\le p}|[(R^{(j)})^2]_{k+n,l+n}|^8\Big).\notag
\end{align}
Using that $\mu_6\le C\sqrt n\mu_4$ and $\mu_8\le Cn\mu_4$, we get
\begin{align}
 \E\{|\beta_{j2}|^8\big|\mathfrak M^{(j)}\}\le &\frac C{n^{8}}\Big((\sum_{1\le l\ne k\le p}|[(R^{(j)})^2]_{k+n,l+n}|^2)^4
\notag\\&+n(\sum_{1\le l\ne k\le p}|[(R^{(j)})^2]_{k+n,l+n}|^6)
(\sum_{1\le l\ne k\le p}|[(R^{(j)})^2]_{k+n,l+n}|^2)\notag\\&+n^2\sum_{1\le l\ne k\le p}|[(R^{(j)})^2]_{k+n,l+n}|^8\Big).\notag
\end{align}

Applying Lemma \ref{resol00}, inequalities \eqref{res5} and \eqref{res6} and Corollary \ref{cor8}, we get the claim.
Thus Lemma \ref{bet1a} is proved.
\end{proof}
\begin{lem}\label{beta} Assuming the conditions of Theorem \ref{main}, we have, for $j=1,\ldots,n$,
\begin{align}\notag
|\beta_{j1}|&\le \frac C{nv^2}.
\end{align}
\end{lem}
\begin{proof}
Let $\mathcal F_n^{(j)}(x)$ denote empirical spectral distribution function of \tc{the} matrix $\mathbf W^{(j)}$. According to {\it the interlacing eigenvalues  Theorem}
(see \cite{Horn}, Theorem 4.38) we have
\begin{equation}
 \sup_x|\mathcal F_n(x)-\mathcal F_n^{(j)}(x)|\le \frac Cn.
\end{equation}
Furthermore, we represent
\begin{equation}\notag
 \beta_{j1}=\int_{-\infty}^{\infty}\frac1{(x-z)^2}d(\mathcal F_n(x)-\mathcal F_n^{(j)}(x))+\frac1{n|z|^2}.
\end{equation}
Integrating by parts, we get the claim.

Thus Lemma \ref{beta} is proved.
\end{proof}
\begin{lem}\label{lem14}Assuming the conditions of Theorem \ref{main}, we have
\begin{align}
\E\frac{|\varepsilon_{j3}|^4}{|z+y(s_y(z)+m_n(z))+\frac{y-1}z|^4}&\le\frac {C(y)}{n^4v^4}. 
\end{align}
\end{lem}
\begin{proof}Using the representations \eqref{shur} we have 
\begin{align}\label{shu}
  \varepsilon_{j3}&=\frac y{2n}(1+\frac1p\sum_{l, k=1}^pX_{jl}X_{jk}[(\mathbf R^{(j)})^2]_{l+n,k+n})R_{jj}+\frac{y}{2nz}\notag\\&=\frac y{2n}((1+\beta_{j1} +\beta_{j2}+\beta_{j3}) R_{jj}+\frac1{2z}).
  \end{align}
 Applying the Cauchy--Schwartz inequality, we get
\begin{align}
\E&|\frac{\varepsilon_{j3}}{z+y(s_y(z)+m_n(z))+\frac{y-1}z}|^4\notag\\&\le \frac C{n^4}\Bigg(\Bigg(1+\E^{\frac12}\Bigg(\frac{|\frac1n\sum_{l,k=1}^p
X_{jl}X_{jk}[(\mathbf R^{(j)})^2]_{l+n,k+n}|}{|z+y(s_y(z)+m_n(z))+\frac{y-1}z|}\Bigg)^8\Bigg)\E^{\frac12}|R_{jj}|^8+\frac1{|z|^4}\Bigg).\notag
\end{align}

Using Corollary \ref{cor8},  we we may write
\begin{align}
\E&|\frac{\varepsilon_{j3}}{z+y(s_y(z)+m_n(z))+\frac{y-1}z}|^4\notag\\&\le\frac C{n^4}
\Big(1+\frac1{|z|^4}+\E^{\frac12}\Big|\frac{\beta_{j1}}{z+y(m_n(z)+s_y(z))+\frac{y-1}z}\Big|^8 \notag\\&
+\E^{\frac12}\Big|\frac{\beta_{j2}}{z+y(m_n(z)+s_y(z))+\frac{y-1}z}\Big|^8+
\E^{\frac12}\Big|\frac{\beta_{j3}}{z+y(m_n(z)+s_y(z))+\frac{y-1}z}\Big|^8\Big).\notag
\end{align}

Using Lemma \ref{beta}, we get  by definition of $\eta_{j1}$, Lemma \ref{resol00}, and  inequality \eqref{res1} that for $z\in\mathbb G$,
\begin{align}\notag
\E^{\frac12}\Big|\frac{\beta_{j1}}{z+y(m_n(z)+s_y(z))+\frac{y-1}z}\Big|^8
&\le C\E^{\frac12}\frac{C}{n^8v^{16}|z+y(m_n^{(j)}(z)+s_y(z))+\frac{y-1}z|^8}
\notag\\&\le Cv^{-4}.
\end{align}
Furthermore, applying inequality \eqref{lar1}, we obtain
\begin{align}\notag
\E^{\frac12}\Big|\frac{\beta_{j2}}{z+y(m_n(z)+s_y(z))+\frac{y-1}z}\Big|^8&\le\E^{\frac12}\Big|\frac{\beta_{j2}}{z+y(m_n^{(j)}(z)+s(z))+\frac{y-1}z}\Big|^8.
\end{align}
Conditioning with respect to $\mathfrak M^{(j)}$ and applying Lemma \ref{bet1*}, we obtain
\begin{align}
\E^{\frac12}&\Big|\frac{\beta_{j2}}{z+y(m_n(z)+s_y(z))+\frac{y-1}z}\Big|^8\notag\\&
\le\E^{\frac12}\frac C{|z+y(m_n^{(j)}(z)+s_y(z))+\frac{y-1}z|^8}\Big(\frac1{n^4v^{12}}(\im m_n^{(j)}(z))^4\notag\\&
\quad\quad\quad\quad\quad\quad\quad\quad\quad\quad\quad\quad\quad\quad\quad\quad+
\frac{\mu_4^4}{n^4v^8}\frac 1n\sum_{l=1}^p(\im R^{(j)}_{l+n,l+n})^8\Big).
\end{align}
Using  Lemma \ref{lem00}, inequality \eqref{lem00.2},
together with Corollary \ref{cor8} we get
\begin{align}
\E^{\frac12}\Big|\frac{\beta_{j2}}{z+y(m_n(z)+s_y(z))+\frac{y-1}z}\Big|^8&\le\frac C{n^2v^6|(z+\frac{y-1}z)^2-4y|}\notag\\&+\frac{C\mu_{4}^2}{n^{2}v^4|(z+\frac{y-1}z)^2-4y|^2}.\notag
\end{align}

Applying inequality \eqref{lar1} and conditioning with respect to $\mathfrak M^{(j)}$ and applying Lemma \ref{bet1*}, we get
\begin{align}\label{ka0}
\E^{\frac12}\Big|&\frac{\beta_{j3}}{z+y(m_n(z)+s_y(z))+\frac{y-1}z}\Big|^8\notag\\&\le 
\E^{\frac12}\frac1{|z+y(s_y(z)+m_n^{(j)}(z))+\frac{y-1}z|^8}\Bigg(\frac C{n^4v^{12}}(\im m_n^{(j)}(z))^4\notag\\&+
 \frac {C\mu_4^2}{n^4v^{8}}\Big(\frac1n\sum_{l\in\mathbb T_j}(\im R^{(j)}_{ll})^4\Big)^2\notag\\&+
 \frac {C\mu_4^2}{n^5v^{12}}\Big(\frac1n\sum_{l\in\mathbb T_j}(\im R^{(j)}_{ll})^3\Big)(\im m_n^{(j)}(z))\notag\\&+
 \frac {C\mu_4^4}{n^6v^{12}}\Big(\frac1n\sum_{l\in\mathbb T_j}(\im R^{(j)}_{ll})^2\Big)^2\notag\\&+
 \frac {C\mu_4^2}{n^6v^{12}}\Big(\frac1n\sum_{l\in\mathbb T_j}(\im R^{(j)}_{ll})^2\Big)(\im m_n^{(j)}(z))^2
\Bigg).
\end{align}
Using that $|z+y(m_n^{(j)}(z)+s_y(z))+\frac{y-1}z|\ge \im m_n^{(j)}(z)+\frac{(1-y)v}{|z|^2}$ together with  Lemma \ref{resol00}, we arrive at
\begin{align}
\E^{\frac12}&\Big|\frac{\beta_{j3}}{z+y(m_n(z)+s_y(z))+\frac{y-1}z}\Big|^8\le \frac C{n^{2}v^{6}|(z+\frac{y-1}z)^2-4y|}
\notag\\&+\frac C{n^{2}v^{4}|(z+\frac{y-1}z)^2-4y|^2}
+
\frac C{n^{\frac52}v^{6}|(z+\frac{y-1}z)^2-4y|^{\frac74}}\notag\\&
+
\frac C{n^{3}v^{6}|(z+\frac{y-1}z)^2-4y|^{2}}+
\frac C{n^{3}v^{6}|(z+\frac{y-1}z)^2-4y|^{\frac32}}.\notag
\end{align}
Summarizing we may write now, for $z\in\mathbb G$,
\begin{align}
\E&\frac{|\varepsilon_{j3}|^4}{|z+s(z)+m_n(z)|^4}\le\frac C{n^4v^4}+\frac C{n^{6}v^{6}|(z+\frac{y-1}z)^2-4y|}
\notag\\&+\frac C{n^{6}v^{4}|(z+\frac{y-1}z)^2-4y|^2}
+
\frac C{n^{\frac{13}2}v^{6}|(z+\frac{y-1}z)^2-4y|^{\frac74}}\notag\\&
+
\frac C{n^{7}v^{6}|(z+\frac{y-1}z)^2-4y|^{2}}+
\frac C{n^{7}v^{6}|(z+\frac{y-1}z)^2-4y|^{\frac32}}.\notag
\end{align}
For $z\in\mathbb G$,   see \eqref{region} and Lemma \ref{lemG}, this inequality may be simplified by means of the following bounds (with $v_0= A_0n^{-1}$)
\begin{align}\label{lowerbound}
 n^{\frac52}v^2|(z+\frac{y-1}z)^2-4y|^{\frac74}\ge n^{\frac52}v_0^2\gamma^{-1+\frac74}\ge C\sqrt n\gamma^{\frac34}\ge C,\notag\\
 n^3v^2|(z+\frac{y-1}z)^2-4y|\ge C,\quad n^3v^2|(z+\frac{y-1}z)^2-4y|^{\frac32}\ge C,\notag\\
 n^2|(z+\frac{y-1}z)^2-4y|^2\ge C.
\end{align}

Using these relation,  we obtain
\begin{align}\notag
\E\frac{|\varepsilon_{j3}|^4}{|z+y(s(z)+m_n(z))+\frac{y-1}z|^4}&\le\frac C{n^4v^4}.
\end{align}
Thus Lemma \ref{lem14} is proved.
\end{proof}



\begin{thebibliography}{99}
\itemsep=\smallskipamount
 \bibitem{Bai:02} Bai, Z. D., Miao, Tsay, J.
      {\em Convergence rates of the spectral distributions
      of large {W}igner matrices}.
Int. Math. J. {\bf 1} (2002), 65--90.
 \bibitem{Bai:93}Bai, Z. D.
 {\em Convergence rate of expected spectral distributions of large random matrices. I. Wigner matrices.} Ann. Probab. 21 (1993), no. 2, 625--648.

\bibitem{BaiSilv:2010}Bai Z. D., Silverstein J. W.
 {\em Spectral Analysis of Large Dimensional Random Matrices.} Springer 2010, 551 PP.

 

 

 \bibitem{BGT:08} Bobkov, S.; G\"otze, F.; Tikhomirov, A. N.
 {\em On concentration of empirical measures and convergence
 to the semi--circle law.}  Journal of Theoretical Probability, {\bf 23},
  (2010), 792--823.
\bibitem{Burkholder:1973} Burkholder, D. L. {\em Distribution function inequalities for martingales.} Ann. Probability 1 (1973), 19–-42.
\bibitem{ErdosYauYin:2010a} Erd\"os, L.; Yau, H.-T.;  Yin, J.
     {\em Rigidity of eigenvalues of generalized {W}igner matrices}. Preprint,  arXiv:1007.4652.

\bibitem{ErdosYauYin:2010}  Erd\"os, L.; Yau ; H.-T,  Yin, J.
     {\em Bulk universality for generalized {W}igner matrices}. Preprint, arXiv:1001.3453.
\bibitem{Girko:02} Girko, V. L. {\em Extended proof of the statement:
Convergence rate of expected spectral functions of symmetric random
matrices $\Sigma_n$ is equal $O(n^{-\frac12})$ and the method of critical
steepest  descent.} Random Oper. Stochastic Equations {\bf 10} (2002),
253--300.

\bibitem{GT:2003} G\"otze, F.; Tikhomirov, A. N. {\em Rate of convergence to the semi-circular law.}
Probab. Theory Related Fields 127 (2003), no. 2, 228--276.
\bibitem{GT:2004}\tc{G\"otze, F.; Tikhomirov, A. N.}
{\em Rate of convergence in probability to the Marchenko-Pastur law.} Bernoulli 10 (2004), no. 3, 503–548

\bibitem{GT:2005}G\"otze, F.; Tikhomirov, A. N.
 {\em The rate of convergence for spectra of {GUE} and {LUE} matrix ensembles.}
 Cent. Eur. J. Math. {\bf 3} (2005),  666--704.

\bibitem{GT:2009} G\"otze, F.; Tikhomirov, A. N. {\em The rate of convergence of spectra of sample covariance matrices.}
Teor. Veroyatn. Primen.  54  (2009),   196--206.
\bibitem{GT:2013a}G\"otze, F.; Tikhomirov, A. N. {\em On the Rate of Convergence to the  Marchenko--Pastur Distribution }
arXiv:1110.1284
\bibitem{GT:2013}G\"otze, F.; Tikhomirov, A. N.
{\em On the rate of convergence to the semi-circular law
}. Preprint, arXiv:1109.0611v3.
 \bibitem{GT:2014} G\"otze, F.; Tikhomirov, A. N.
     {\em On the optimal bounds of the rate of convergence of the expected spectral distribution functions to the semi-circle law.}\\
    Preprint. 2014, available
  on http://arxiv.org/abs/1405.7820. 
  \bibitem{GT:2014a} G\"otze, F.; Tikhomirov, A. N.
     {\em Rate of convergence of the empirical  spectral distribution functions to the semi-circle law.}\\
    Preprint. 2014, available
  on http://arxiv.org/abs/1407.2780. 
 \bibitem{Gustavsson:2005} Gustavsson, Jonas.
     {\em Gaussian fluctuations of eigenvalues in the GUE}.\\
     Ann. I. H. Poincare --PR 41 (2005),151--178.
 \bibitem{Hitczenko:1990}Hitczenko P. {\em Best constant in the decoupling inequality for nonnegative random variables}. 
 Statist. Probab. Lett. 9 (1990), no. 4, 327–-329.
\bibitem{Johnson:1985} Johnson, W. B., Schechtman G., Zinn J.
{\em Best Constants in Moment Inequalities for Linear Combinations of Independent and Exchangeable random Variables}. \tc{Ann.  Probab.} 13(1985) No.1,  234--253 
\bibitem{SchleinMaltseva:2013} Cacciapuoti C.,  Maltsev A.,  Schlein B. 
{\em Optimal Bounds on the Stieltjes Transform of Wigner Matrices}. Preprint, 2013, arXiv:1311.0326 , available
  on http://arxiv.org/abs/1311.0326  
  \bibitem{Horn}Horn R. A., Horn Ch. R.
  {Topics in Matrix Analysis}. Cambridge University Press 1991, 607 PP.
\bibitem{Rosenthal:1970} Rosenthal, H. P. {\em On the subspaces of $L_p,\, (p>2)$ spanned by sequences of independent random variables}. 
Israel J. Math. 8 (1970) 273--303.
\bibitem{TTKh:2008}Timushev, D. A.; Tikhomirov, A. N.; Kholopov, A. A. {\em On the accuracy of the approximation of the GOE spectrum 
 by the semi-circular law}. (Russian) Teor. Veroyatn. Primen. 52 (2007), no. 1, 180--185; translation in Theory Probab. Appl. 52 (2008), no. 1, 171–-177
\bibitem{T:09}Tikhomirov, A. N.
 {\em On the rate of convergence
of the expected spectral distribution function of a {W}igner matrix to the semi-circular law.}
  Siberian Adv. Math.  {\bf19},  (2009),  211--223.
 \end{thebibliography}
\end{document}